\documentclass[12pt]{amsart}
\usepackage{amsmath,amssymb,amsthm}
\usepackage{comment}
\usepackage{color}
\allowdisplaybreaks[1]

\title[effective actions of the general indefinite unitary groups]
{Effective actions of the general indefinite unitary groups on
holomorphically separable manifolds}

\author{Yoshikazu Nagata}

\address{Graduate School of Mathematics, Nagoya University,
Furocho, Chikusaku, Nagoya 464-8602, Japan}

\email{m10035y@math.nagoya-u.ac.jp}

\theoremstyle{plain}
\newtheorem{theorem}{Theorem}
\newtheorem{lemma}{Lemma}
\newtheorem{corollary}{Corollary}

\newtheorem{definition}{Definition}

\theoremstyle{remark}
\newtheorem{remark}{Remark}

\numberwithin{theorem}{section}
\numberwithin{lemma}{section}
\numberwithin{corollary}{section}
\numberwithin{remark}{section}
\numberwithin{equation}{section}
\numberwithin{claim}{section}
\numberwithin{definition}{section}

\newcommand{\aut}{\mathrm{Aut}}
\newcommand{\B}{\mathbb{B}}
\newcommand{\C}{\mathbb{C}}
\newcommand{\D}{\mathbb{D}}
\newcommand{\Q}{\mathbb{Q}}
\newcommand{\R}{\mathbb{R}}
\newcommand{\Z}{\mathbb{Z}}

\begin{document}
\maketitle
\footnote{
}

\begin{abstract}
We determine all holomorphically separable complex manifolds of
dimension $p+q$ which admits a smooth envelope of holomorphy such that
the general indefinite unitary group of size $p+q$ acts effectively by
holomorphic transformations.
Also we give exact description of the automorphism groups
of those complex manifolds.
As an application we consider a characterization
of those complex manifolds by their automorphism groups.
\end{abstract}

\tableofcontents

\section{Introduction}
\label{intro}

Isaev and Kruzhilin determined all complex manifolds of dimension $n$
on which the unitary group $U(n)$ acts effectively as holomorphic
transformations \cite{IK}.
As a consequence, they proved that the complex euclidean space $\C^n$
is characterized by its automorphism group.
We say that a complex manifold $M$ is characterized by its automorphism group,
if any complex manifold $N$ (sometimes with some assumption)
whose automorphism group
$\aut(N)$ is isomorphic to $\aut(M)$ as topological groups is biholomorphically
equivalent to $M$.
In this paper,
toward a classification of complex manifolds with a holomorphic
indefinite unitary group action, we treat complex manifolds
on which the general
indefinite unitary groups $GU(p,q)$ act effectively by holomorphic
transformations.
Moreover we give exact description of the automorphism groups
of those complex manifolds,
and the characterization by their automorphism groups.
The characterization problem is considered mainly for
homogeneous complex manifolds,
otherwise there exist many counterexamples.
For homogeneous complex manifolds, there exist many positive results
on the characterization,
e.g. \cite{BKS0}, \cite{BKS}, \cite{IK}, \cite{KS1}, \cite{MN}, etc,
and there exists a counterexample in \cite{MN}.
In order to describe our results,
let us fix notation here.
Let $\Omega$ be a complex manifold.
An {\it automorphism} of $\Omega$ means a biholomorphic mapping of $\Omega$ onto itself. 
We denote by $\mathrm{Aut}(\Omega)$ 
the group of all automorphisms of
$\Omega$ equipped with the compact-open topology.
If $A_1, \ldots, A_k$ are square matrices,
$\mathrm{diag}[A_1, \ldots, A_k]$ denotes the matrix
with $A_1, \ldots, A_k$ in the diagonal blocks and $0$ in all other blocks.
We put here two domains in $\C^{p+q}$:
\begin{equation*}
D^{p,q}=\{ (z_1, \ldots, z_q, z_{q+1}, \ldots, z_{p+q}) \in \mathbb{C}^{p+q}: -|z_{1}|^{2} - \cdots -|z_{q}|^{2} 
+ |z_{q+1}|^{2}+ \cdots + |z_{p+q}|^{2} > 0 \},
\end{equation*}
and the exterior of $D^{p,q}$ in $\mathbb{C}^{p+q}$
\begin{equation*}
C^{p,q}=\{ (z_1, \ldots, z_p, z_{p+1}, \ldots, z_{p+q}) \in \mathbb{C}^{p+q}: -|z_{1}|^{2} - \cdots -|z_{q}|^{2}
+ |z_{q+1}|^{2}+ \cdots + |z_{p+q}|^{2} < 0 \}.
\end{equation*}
Recall the indefinite unitary group of signature $(p,q)$
\begin{equation*}
U(p,q)=\{A \in GL(p+q, \mathbb{C}) : A^*JA = J \},
\end{equation*}
where 
\begin{eqnarray*}
J=
\mathrm{diag}[-E_q,E_p].
\end{eqnarray*}
Then we put the general indefinite unitary group of signature $(p,q)$ by
\begin{eqnarray*}
GU(p,q)
&=&
\{A \in GL(p+q, \mathbb{C}) : A^*JA = \nu(A)J, \mathrm{for\,some}\, \nu(A) \in \mathbb{R}_{>0} \}\\
&\simeq& U(p,q) \times \R_{>0},
\end{eqnarray*}
and identify $\mathbb{C}^*$ 
with the center of $GU(p,q)$:
\begin{equation*}
\mathbb{C}^* \simeq \{
\mathrm{diag}[\alpha,\ldots,\alpha] \in GL(p+q, \mathbb{C}):
\alpha \in \mathbb{C}^* \} \subset GU(p,q).
\end{equation*}
Since $U(p,q)$ acts 
on each level sets of
$-|z_{1}|^{2} \cdots -|z_{q}|^{2} + |z_{q+1}|^{2} + \cdots + |z_{p+q}|^{2}$,
and $\mathbb{C}^*$ acts on $D^{n,1}$ and $C^{n,1}$ as scalar multiplication,
the group $GU(p,q)$ is a subgroup of the automorphism groups of these two domains 
$D^{p,q}$ and $C^{p,q}$.
Furthermore,
$GU(n,1)$ acts on $\C^* \times \B^n$ and $\C \times \B^n$
effectively as holomorphic transformations, where
$\B^n = \{ z\in\C^n : |z|<1 \}$. Indeed, we have the action of 
$A=(a_{ij})_{0 \leq i,j \leq n} \in GU(n,1)$ on $\C \times \B^n$ by
\begin{eqnarray*}
\C \times \B^n \ni (z_0,z_1,\ldots,z_n) \rightarrow 
\left(\gamma_0(z_0, z_1, \ldots, z_n),\gamma_1(z_1, \ldots, z_n),\ldots,
\gamma_n(z_1, \ldots, z_n)\right) \in \C \times \B^n,
\end{eqnarray*}
where
\begin{eqnarray*}
\gamma_0(z_0, z_1, \ldots, z_n) =
z_0(a_{00} + \sum_{j=1}^{n} a_{0j}z_j),
\end{eqnarray*}
and
\begin{eqnarray*}
\gamma_i(z_1, \ldots, z_n) =
\frac{a_{i0} + \sum_{j=1}^{n} a_{ij}z_j}{a_{00} + \sum_{j=1}^{n} a_{0j}z_j},
\end{eqnarray*}
for $1 \leq i \leq n$.
This action preserves $\{0\} \times \B^n$, therefore
$GU(n,1)$ acts on $\C^* \times \B^n$.
We note that $C^{n,1}$ is biholomorphic to $\C^* \times \B^n$ by a biholomorphic
map
\begin{equation*}
C^{n,1} \ni (z_{0}, z_{1}, z_{2}, \ldots, 
z_{n}) \mapsto 
\left(z_{0}, \frac{z_{1}}{z_{0}}, \ldots, \frac{z_{n}}{z_{0}} \right) \in 
\mathbb{C}^* \times \mathbb{B}^{n}.
\end{equation*}
and again we see from this map that $GU(n,1)$ acts on $\C^* \times \B^n$.

Clearly $GU(p,q)$ acts on $\C^{p+q}$ and $\C^{p+q} \setminus \{0\}$.
We will show that, under certain conditions,
those domains are all complex manifolds
on which $GU(p,q)$ acts effectively by holomorphic transformations.
The precise statements are the following:
\let\temp\thetheorem
\renewcommand{\thetheorem}{\ref{M1}}
\begin{theorem}
Let $M$ be a connected complex manifold of dimension $n+1 > 2$ 
that is holomorphically separable
and admits a smooth envelope of holomorphy.
Assume that there exists an injective homomorphism of topological groups
$\rho_0 : GU(n,1) \longrightarrow \mathrm{Aut}(M)$.
Then $M$ is biholomorphic to one of the five domains
$\C^{n+1}$, $\C^{n+1} \setminus \{0\}$, $D^{n,1}$,
$C^{n,1} \simeq \C^* \times \B^n$ and $\C \times \B^n$.
\end{theorem}
\let\thetheorem\temp
\addtocounter{theorem}{-1}

\let\temp\thetheorem
\renewcommand{\thetheorem}{\ref{M2}}
\begin{theorem}
Let $M$ be a connected complex manifold of dimension $2$ 
that is holomorphically separable
and admits a smooth envelope of holomorphy.
Assume that there exists an injective homomorphism of topological groups
$\rho_0 : GU(1,1) \longrightarrow \mathrm{Aut}(M)$.
Then $M$ is biholomorphic to one of the four domains
$\C^{2}$, $\C^{2} \setminus \{0\}$, $D^{1,1}\simeq \C^* \times \D$ and
$\C \times \D$,
where $\D=\B^1$.
\end{theorem}
\let\thetheorem\temp
\addtocounter{theorem}{-1}

\let\temp\thetheorem
\renewcommand{\thetheorem}{\ref{M3}}
\begin{theorem}
Let $p, q >1$ and $n = p + q$.
Let $M$ be a connected complex manifold of dimension $n$ 
that is holomorphically separable
and admits a smooth envelope of holomorphy.
Assume that there exists an injective homomorphism of topological groups
$\rho_0 : GU(p,q) \longrightarrow \mathrm{Aut}(M)$.
Then $M$ is biholomorphic to one of the four domains
$\C^{n}$, $\C^{n} \setminus \{0\}$, $D^{p,q}$ and
$C^{p,q}$.
\end{theorem}
\let\thetheorem\temp
\addtocounter{theorem}{-1}

Observe that clearly $D^{p,p} \simeq C^{p,p}$,
and, if $p \neq q$, then $D^{p,q} \not \simeq C^{p,q}$ (see \cite{MN}).
We can put Theorem~\ref{M1} and Theorem~\ref{M2} together,
since $D^{1,1} \simeq C^{1,1}$.
When we prove those theorems, however, we need to divide those cases.
We give the proofs of Theorems~\ref{M1}, \ref{M2} and \ref{M3}
in Section \ref{sec:3}, \ref{sec:4} and \ref{sec:5},
respectively.

In Section \ref{sec:6}, we give precise descriptions of the automorphism groups
of the domains in Theorems~\ref{M1}, \ref{M2} and \ref{M3},
except for $\C^n$ and $\C^n \setminus \{0\}$.

\let\temp\thetheorem
\renewcommand{\thetheorem}{\ref{autCB}}
\begin{theorem}[\cite{BKS0}]
For $f=(f_{0}, f_{1}, \ldots, f_{n}) \in \aut(\C \times \B^n)$,
\begin{align*}
f_{0}(z_{0}, z_{1}, \ldots, z_{n})=
a(z_{1}, \ldots, z_{n}) z_{0} + b(z_{1}, \ldots, z_{n}),
\end{align*}
and
\begin{align*}
f_{i}(z_{0}, z_{1}, \ldots, z_{n}) =
\frac{a_{i 0}+\sum_{j=1}^{n} a_{i j} z_{j}}{a_{00}+\sum_{j=1}^{n} a_{0 j} z_{j}}
\end{align*}
for $i=1, \ldots, n$,
where $a$ is a nowhere vanishing holomorphic function on 
$\mathbb{B}^{n}$, 
$b$ is a holomorphic function on 
$\mathbb{B}^{n}$,
and the matrix $(a_{ij})_{0\leq i, j \leq n}$ is an element of $SU(n,1)$.  
\end{theorem}
\let\thetheorem\temp
\addtocounter{theorem}{-1}

\let\temp\thetheorem
\renewcommand{\thetheorem}{\ref{autC*B}}
\begin{theorem}[\cite{MN}]
For $f=(f_{0}, f_{1}, \ldots, f_{n}) \in \mathrm{Aut}(C^{n,1})$,
\begin{eqnarray*}
f_{0}(z_{0}, z_{1}, \ldots, z_{n}) =
c\left(\frac{z_{1}}{z_{0}}, \ldots, \frac{z_{n}}{z_{0}}\right) z_{0} 
 \ \mathrm{or} \ 
c\left(\frac{z_{1}}{z_{0}}, \ldots, \frac{z_{n}}{z_{0}}\right) z_{0}^{-1},
\end{eqnarray*}
and
\begin{eqnarray*}
f_{i}(z_{0}, z_{1}, \ldots, z_{n}) =f_{0}(z_{0}, z_{1}, \ldots, z_{n}) 
\frac{a_{i 0}+\sum_{j=1}^{n} a_{i j} z_{j}}{a_{00}+\sum_{j=1}^{n} a_{0 j} z_{j}},
\end{eqnarray*}
for $i=1, \ldots, n$,
where $c$ is a nowhere vanishing holomorphic function on 
$\mathbb{B}^{n}$,
and the matrix $(a_{ij})_{0\leq i,j \leq n}$ is an element of $SU(n,1)$.  
\end{theorem}
\let\thetheorem\temp
\addtocounter{theorem}{-1}

\let\temp\thetheorem
\renewcommand{\thetheorem}{\ref{autDpq}}
\begin{theorem} 
$\aut(D^{p,q}) \simeq GU(p,q)$ for $p>1$, $q>0$.
\end{theorem}
\let\thetheorem\temp
\addtocounter{theorem}{-1}

As a corollary of above theorems, together with the classification in \cite{IK},
we can state the characterization theorem for some domains
by their automorphism groups.

\begin{corollary}\label{cor1}
Let M be a connected complex manifold of dimension p+q that is holomorphically separable
and admits a smooth envelope of holomorphy. 
Assume that $\mathrm{Aut}(M)$ is isomorphic to $\mathrm{Aut}(D^{p,q})$
as topological groups. Then
\begin{description}
\item[$(i)$] If $q=1$, then $M\simeq D^{p,1}$,
\item[$(ii)$] If $p=1$, then $M\simeq D^{1,q}$,
\item[$(iii)$] If $p=q$, then $M\simeq D^{p,p}$.
\end{description}
\end{corollary}

\begin{proof}
For $(i)$, the case $p>1$ follows from Theorems~\ref{M1}, \ref{autCB},
\ref{autC*B} and \ref{autDpq}.
Indeed, $\aut(D^{p,1})$ is a linear Lie group by Theorem~\ref{autDpq},
while the automorphism groups of $\C^{p+1}$, $\C^{p+1} \setminus \{0\}$,
$C^{p,1}$ and $\C \times \B^n$ in the list of Theorem~\ref{M1}
are not Lie groups.
Thus $\aut(D^{p,1})$ is neither isomorphic to $\aut(\C^{p+1})$,
$\aut(\C^{p+1} \setminus \{0\})$, $\aut(\C \times \B^p)$
nor $\aut(\C^* \times \B^p)$ as topological groups,
and therefore
characterizes the domain $D^{p,1}$.

For $(ii)$, the case $q>1$
follows from above $(i)$, Theorems~\ref{M1},
\ref{autCB}, \ref{autC*B} and the fact that 
$U(1+q)$ does not acts effectively on $D^{1,q}$ as holomorphic
transformations (see \cite{IK}).
Since $U(1+q)$ acts effectively on $\C^{1+q}$ and $\C^{1+q} \setminus \{0\}$,
$\aut(D^{1,q})$ is neither isomorphic to $\aut(\C^{1+q})$ nor
$\aut(\C^{1+q} \setminus \{0\})$.
Observe that $\aut(D^{1,q})$ is not connected by Theorem~\ref{autC*B},
namely, there exists two
components which include the maps
$(z_0,z_1,\ldots,z_q)$ and $(z_0^{-1},z_1,\ldots,z_q)$, respectively,
while $\aut(\C \times \B^q)$ is connected since $\B^n$ is contractible.
Therefore $\aut(C^{q,1})$ and $\aut(\C \times \B^q)$ are not isomorphic
to each other as topological groups.
Hence, by Theorem~\ref{M1}, $\aut(D^{1,q})$ characterizes the domain $D^{1,q}$.

For $(iii)$, $\aut(D^{p,p})$ is neither isomorphic to 
$\aut(\C^{2p})$ nor
$\aut(\C^{2p} \setminus \{0\})$,
since $U(2p)$ does not act effectively
on $D^{p,p}$ as holomorphic transformations.
If $p=1$, then $\aut(D^{1,1})$ is not isomorphic to $\aut(\C \times \B^1)$
as the proof of $(ii)$ above.
Thus, by Theorem~\ref{M2}, this case is proven.
For $p>1$, the assertion holds by Theorem~\ref{M3}.

\end{proof}

For the case $q=1$, a direct proof of the characterization for $C^{p,1}$
was given in \cite{BKS},
and that for $D^{p,1}$ was given in \cite{MN}.
In the paper \cite{MN}, it is also proven that, if $p \neq q$, $p,q>1$,
then $D^{p,q} \not \simeq C^{p,q}$, while $\aut(D^{p,q}) \simeq \aut(C^{p,q})$.
This means that for the domains $D^{p,q}$, $p,q>1$,
the characterization by their automorphism groups does not hold.

\begin{corollary}
Let M be a connected complex manifold of dimension $n+1$ that is holomorphically separable
and admits a smooth envelope of holomorphy. 
If $\mathrm{Aut}(M)$ is isomorphic to $\mathrm{Aut}(\C \times \B^n)$ $($or $\mathrm{Aut}(\C^* \times \B^n))$
as topological groups,
then $M$ is biholomorphic to $\C \times \B^n$ $(\C^* \times \B^n$, respectively$)$.
\end{corollary}

The proof is similar to that of Corollary \ref{cor1}.
A direct proof of the characterization for $\C \times \B^n$
was given in \cite{BKS0},
and for $\C^* \times \B^n$ in \cite{BKS}
in which the characterization is given for any
direct product of a ball with Euclidean spaces and
punctured Euclidean spaces, respectively.

\section{Preliminary}
\label{prel}

\subsection{Reinhardt domains}
\label{reinh}

In order to establish terminology and notation, we recall some basic facts
about Reinhardt domains,
following Kodama and Shimizu \cite{KS1}, \cite{KS}.
Let $G$ be a Lie group and
$\Omega$ a domain in $\mathbb{C}^{n}$.
Consider a
continuous group homomorphism $\rho : G \longrightarrow \mathrm{Aut}(\Omega)$. 
Then the mapping
\begin{equation*} 
G \times \Omega \ni (g, x) \longmapsto (\rho(g))(x) \in \Omega
\end{equation*}
is continuous, and in fact $C^{\omega}$.
We say that $G$ acts on $\Omega$
as a Lie transformation group through $\rho$. 
Let $T^{n} = (U(1))^{n}$, the $n$-dimensional
torus. 
$T^{n}$ acts 
as a holomorphic automorphism group
on $\mathbb{C}^n$ in the following standard manner:
\begin{equation*} 
T^n \times \mathbb{C}^n \ni (\alpha, z) \longmapsto 
\alpha \cdot z := (\alpha_1 z_1, \ldots, \alpha_n z_n) \in \mathbb{C}^n.
\end{equation*}
A {\it Reinhardt domain} $\Omega$ in $\mathbb{C}^{n}$ is, by definition,  
a domain which is stable under
this standard action of $T^{n}$. Namely, there exists a continuous map
$T^{n} \hookrightarrow \mathrm{Aut}(\Omega)$.
We denote the image of $T^{n}$ of this inclusion map by $T(\Omega)$.

Let $f$ be a holomorphic function on a Reinhardt domain $\Omega$.  
Then $f$ can be expanded uniquely into a Laurent series
\begin{eqnarray*}
f(z) = \sum_{\nu \in \mathbb{Z}^n} a_{\nu}z^{\nu},
\end{eqnarray*}
which converges absolutely and uniformly on any compact set in $\Omega$.
Here $z^{\nu} = z_1^{\nu_1} \cdots z_n^{\nu_n}$ for 
$\nu = (\nu_1, \ldots, \nu_n) \in \mathbb{Z}^n$.

\begin{lemma}\label{LU}
Let $\Omega$ be a Reinhardt domain in $\mathbb{C}^n$ and $n=p+q$. 
If $p>1$ and $U(p)$ acts by linear transformations on $\C^n$
to first $p$ variables
$z_1, \ldots, z_p$, and the action preserves $\Omega$,
then the Laurent series of a holomorphic function on $\Omega$
does not have negative degree terms of $z_1, \ldots, z_p$.
\end{lemma}

\begin{proof}
Since $\Omega \cap \{z_i = 0\} \neq \emptyset$, for $1 \leq i \leq p$,
by the $U(p)$-action on $\Omega$,
and since the Laurent series are globally defined on the Reinhardt domain
$\Omega$,
the lemma is trivial.

\end{proof}

$(\mathbb{C}^*)^n$ acts holomorphically on $\mathbb{C}^n$ as follows: 
\begin{equation*} 
(\mathbb{C}^*)^n \times \mathbb{C}^n \ni 
((\alpha_1, \ldots, \alpha_n), (z_1, \ldots, z_n)) \longmapsto 
(\alpha_1 z_1, \ldots, \alpha_n z_n) \in  \mathbb{C}^n.
\end{equation*}
We denote by $\Pi(\mathbb{C}^n)$ the group of all automorphisms of $\mathbb{C}^n$ of this form.
For a Reinhardt domain $\Omega$ in $\mathbb{C}^n$, we denote by $\Pi(\Omega)$
the subgroup of $\Pi(\mathbb{C}^n)$
consisting of all elements of $\Pi(\mathbb{C}^n)$ leaving $\Omega$ invariant.

\begin{lemma}[\cite{KS1}]\label{KS1}
Let $\Omega$ be a Reinhardt domain in $\mathbb{C}^n$ . 
Then $\Pi(\Omega)$ is the centralizer of $T(\Omega)$ in $\mathrm{Aut}(\Omega)$.
\end{lemma}

Next lemma is the key to prove Theorems~\ref{M1}, \ref{M2} and \ref{M3}.

\begin{lemma}[Generalized Standardization Theorem~ \cite{KS}]\label{4}
Let M be a connected complex manifold of dimension $n$ that is holomorphically separable
and admits a smooth envelope of holomorphy, 
and let K be a connected compact Lie group of rank $n$.
Assume that an injective continuous group homomorphism 
$\rho$ of K into $\mathrm{Aut}(\Omega)$ is given.
Then there exists a biholomorphic map F of M onto a Reinhardt domain 
$\Omega$ in $\mathbb{C}^{n}$ such that
\begin{equation*} 
F\rho(K)F^{-1} = U(n_1) \times \cdots \times U(n_s) \subset \mathrm{Aut}(\Omega),
\end{equation*}
where $\sum_{j=1}^s n_j = n$.
\end{lemma}

We remark on an action of $U(n,1)$.
In contrast to Lemma \ref{4},
for a non-compact case, $GU(n,1)$ act on $\C \times \B^n$,
which is not linearizable.

Let us consider $U(q) \times U(p)$ as a subgroup of 
$GU(p,q)$ in the natural way.
Then $U(q) \times U(p)$
is a maximal compact subgroup of $GU(p,q)$.
We also identify $U(p)$ and $U(q)$ with $\{E_q\} \times U(p)$
and $U(q) \times \{E_p\}$ in $GU(p,q)$, respectively.
Put the center of the group $U(q) \times U(p)$ in $GU(p,q)$ by
\begin{equation}\label{Tqp}
T_{q,p}=
\left\{
\mathrm{diag}[u_1E_q, u_2E_p]: u_1, u_2 \in U(1) \right\} \subset GU(p,q).
\end{equation}
Here we apply Lemma~\ref{4} to the hypothesis of Theorems~\ref{M1}, \ref{M2}
and \ref{M3}.
Since there exists an injective homomorphism of topological groups
$\rho_0 : GU(p,q) \longrightarrow \mathrm{Aut}(M)$,
and $U(q)\times U(p) \subset GU(p,q)$, 
there is a biholomorphic map $F$ from $M$ onto a Reinhardt domain 
$\Omega$ in $\mathbb{C}^{p+q}$ such that
\begin{equation*} 
F\rho_0(U(q) \times U(p))F^{-1} 
= U(n_1) \times \cdots \times U(n_s) \subset \mathrm{Aut}(\Omega),
\end{equation*}
where $\sum_{j=1}^s n_j = p+q$.
Then, after a permutation of coordinates if we need, we may assume
$F\rho_0(U(q) \times U(p))F^{-1} = U(q) \times U(p)$,
whose action on $\Omega$ is matrix multiplication.
We define an injective homomorphism
\begin{equation*} 
\rho: GU(p,q) \longrightarrow \mathrm{Aut}(\Omega)
\end{equation*}
by
$\rho(g) := F \circ \rho_0(g) \circ F^{-1}$.
We prove in Section $3,4,5$ that $\Omega$ is biholomorphic to  one of  the domains
in the statement of the theorems.

\subsection{Some results on Lie group actions}

We record some results, which will be used in the proof of
the theorems several times.

\begin{lemma}  \label{notsimple}
Let $p, q, k$ be non-negative integers and $p + q \geq 2$. For $k < p + q$,
any Lie group homomorphism
\begin{equation*} 
\rho : SU(p,q) \longrightarrow GL(k, \mathbb{C})
\end{equation*}
is trivial.
\end{lemma}

\begin{proof}
Put $n = p + q$.
It is enough to show that the Lie algebra homomorphism
\begin{equation*} 
d\rho : \mathfrak{su}(p,q) \longrightarrow \mathfrak{gl}(k, \mathbb{C})
\end{equation*}
is trivial.
Consider its complex linear extension
\begin{equation*} 
d\rho_\mathbb{C} : \mathfrak{su}(p,q) \otimes_{\mathbb{R}} \mathbb{C} 
\longrightarrow \mathfrak{gl}(k, \mathbb{C}).
\end{equation*}
Since $\mathfrak{su}(p,q) \otimes_{\mathbb{R}} \mathbb{C} = \mathfrak{sl}(n, \mathbb{C})$
and $\mathfrak{sl}(n, \mathbb{C})$ is a simple Lie algebra,
$d\rho_\mathbb{C}$ is injective or trivial.
On the other hand, $\dim_{\mathbb{C}} \mathfrak{su}(p,q) \otimes_{\mathbb{R}} \mathbb{C} 
= n^2 - 1 > k^2 = \dim_{\mathbb{C}} \mathfrak{gl}(k, \mathbb{C})$.
Thus $d\rho_\mathbb{C}$ must be trivial,
and so is $d\rho$.

\end{proof}

Similarly we have
\begin{lemma}  \label{notsimple2}
Let $p, q, k$ be non-negative integers and $p, q > 0$.
For $p+q > k+1$, any Lie group homomorphism
\begin{equation*} 
\rho : SU(p,q) \longrightarrow PU(k,1)
\end{equation*}
is trivial.
\end{lemma}

\begin{lemma}[\cite {Kutzsche} Proposition 2.3]\label{extG}
Let $G$ be a real Lie group acting by holomorphic transformations on $\C^n$.
Then the action extends to a holomorphic action of the universal
complexification $G^{\C}$ on $\C^n$.
\end{lemma}

We recall the notion of categorical quotient for an action 
$G\times X\rightarrow X$ of a Lie group $G$ on a complex space $X$,
following Kutzschebauch \cite {Kutzsche}.

\begin{definition}[\cite {Kutzsche} P. 86]\label{cateQ}
A complex space $X/\hspace{-1mm}/G$ together with a $G$-invariant holomorphic map
$\pi_{X} : X\rightarrow X/\hspace{-1mm}/G$ is called categorical quotient for the action
$G\times X\rightarrow X$ if it satisfies the following universality property:

For every holomorphic $G$-invariant map $\psi : X\rightarrow Y$ from $X$ to
some complex $G$-space $Y$,
there exists a unique holomorphic $G$-invariant map
$\tilde{\psi} : X/\hspace{-1mm}/G\rightarrow Y$ such that the diagram
\begin{equation*} 
\begin{array}[c]{ccccc}
X\,\,\,\,\,\,\,\,&\stackrel{\psi}{\longrightarrow} &\,\,\,\,\,\,\,\,Y\\
\scriptstyle{\pi_{X}}\,\,\searrow&&\nearrow\,\,\scriptstyle{\tilde{\psi}}\\
&X/\hspace{-1mm}/G&
\end{array}
\end{equation*} 
commutes.
\end{definition}

As a topological space, when exists, $X/\hspace{-1mm}/G$ is just the topological quotient of
$X$ with respect to the equivalence relation $R$
associated to the algebra $\mathcal{O}^{G}(X)$ of $G$-invariant 
holomorphic functions on $X$:
\begin{equation*}
R=\{(x, y)\in X\times X: f(x)=f(y) \,\,\,\, \forall f\in \mathcal{O}^{G}(x)\}
\end{equation*}
It is known that,
if $X=\C^n$ and $G$ is a complex reductive group or a compact group,
there exists the categorical quotients.
In the proof of Theorems~\ref{M1}, \ref{M2} and \ref{M3},
we use the following lemma with $X=\C^n$ and $G=GL(n,\C)$, $SL(n,\C)$ or $U(n)$.
In these cases, $X/\hspace{-1mm}/G$ are a one-point set.

\begin{lemma}[\cite {Kutzsche} P. 87]\label{linG}
Any holomorphic action of a complex reductive group $G$ on $\C^n$
through $\rho$
with the zero dimensional categorical quotient is linearizable, i.e.
there exists $\gamma \in\aut(\C^n)$ such that
$\gamma^{-1}\rho(G)\gamma\subset GL(n,\C)$.
\end{lemma}

In the proof of Theorems~\ref{M3} and \ref{autDpq},
we need the following.

\begin{lemma}\label{condGU}
Let $p, q>0$.
If $f \in GL(p+q,\C)$ preserves $D^{p,q}$,
then we have $f \in GU(p,q)$.
\end{lemma}

\begin{proof}
Since $f \in GL(p+q,\C)$ preserves $D^{p,q}$, $f$ preserves $C^{p,q}$ and the null cone
\[
\partial D^{p,q} =
\{(z_1, \ldots, z_{p+q}) \in \mathbb{C}^{p+q}:
-|z_{1}|^{2} - \cdots -|z_{q}|^{2} + |z_{1+q}|^{2}+ \cdots + |z_{p+q}|^{2}=0\},
\]
Put
\[
f =
\begin{pmatrix}
A & B \\
C & D 
\end{pmatrix}
\in GL(p+q,\C),
\]
where $A = (a_{ij}) \in M(q,\C)$, $B = (b_{ij}) \in M(q,p,\C)$,
$C = (c_{ij}) \in M(p,q,\C)$
and $D = (d_{ij}) \in M(p,\C)$.
We will show that
\begin{eqnarray}\label{mat1}
\hspace{1.5cm}
\begin{pmatrix}
{}^t\!A & {}^t\!C \\
{}^t\!B & {}^t\!D 
\end{pmatrix}
\begin{pmatrix}
-E_q & 0 \\
0 & E_p
\end{pmatrix}
\begin{pmatrix}
\overline A & \overline B \\
\overline C & \overline D 
\end{pmatrix}
=
\left(\sum_{i=1}^q |a_{i1}|^2 - \sum_{i=1}^p |c_{i1}|^2 \right)
\begin{pmatrix}
-E_q & 0 \\
0 & E_p
\end{pmatrix}
.
\end{eqnarray}
Since $f \in GL(p+q,\C)$ preserves $C^{p,q}$ and $f(1,0.\ldots,0)=(a_{11},\ldots,a_{q1},c_{11},\ldots,c_{p1})$,
it follows that
$\sum_{i=1}^q |a_{i1}|^2 - \sum_{i=1}^p |c_{i1}|^2$
is positive,
and therefore we will find $f \in GU(p,q)$.

The left-hand side of (\ref{mat1}) equals
\[
\begin{pmatrix}
-{}^t\!A\overline A + {}^t\!C\overline C & -{}^t\!A\overline B + {}^t\!C\overline D \\
-{}^t\!B\overline A + {}^t\!D\overline C & -{}^t\!B\overline B + {}^t\!D\overline D
\end{pmatrix}
.
\]
Put coordinates $z = (z_1,\ldots,z_q) \in \C^q$ and
$z' = (z_1,\ldots,z_p) \in \C^p$.
For $(z,z') \in \partial D^{p,q}$,
we have
\[
-|| Az + Bz' ||^{2} + || Cz + Dz' ||^{2}=0,
\]
where $|| \cdot ||$ is a usual euclidean norm.
If $z=(0,\ldots,0,z_j,0,\ldots,0)$, $z'=(0,\ldots,0,z_k,0,\ldots,0)$
and
$-|z_j|^2+|z_k|^2=0$,
then we have
\[
-\sum_{i=1}^q |a_{ij}z_j + b_{ik}z_k|^{2} +
\sum_{i=1}^p |c_{ij}z_j + d_{ik}z_k|^{2} = 0.
\]
Putting $z_k=e^{\sqrt[]{-1}\theta}z_j$, $\theta\in\R$,
we can easily derive from this equation that
\begin{eqnarray*}
&&-\sum_{i=1}^q \left(|a_{ij}|^2 + |b_{ik}|^{2}\right) +
\sum_{i=1}^p \left(|c_{ij}|^2 + |d_{ik}|^{2}\right) = 0,\\
&&-\sum_{i=1}^q a_{ij} \overline{b_{ik}} +
\sum_{i=1}^p c_{ij} \overline{d_{ik}} = 0.
\end{eqnarray*}
The second equation means
$-{}^t\!A\overline B + {}^t\!C\overline D = 0$ and
$-{}^t\!B\overline A + {}^t\!D\overline C = 0$.
The first equation means
\[
-\sum_{i=1}^q |b_{ik}|^{2} +
\sum_{i=1}^p |d_{ik}|^{2}
= \sum_{i=1}^q |a_{ij}|^2 -
\sum_{i=1}^p |c_{ij}|^2.
\]
Therefore all diagonal components of $-{}^t\!A\overline A + {}^t\!C\overline C$
equal
$- \sum_{i=1}^q |a_{i1}|^2 + \sum_{i=1}^p |c_{i1}|^2$,
and those of $-{}^t\!B\overline B + {}^t\!D\overline D$ 
equal $\sum_{i=1}^q |a_{i1}|^2 - \sum_{i=1}^p |c_{i1}|^2$.

If $z=(0,\ldots,0,z_j,0,\ldots,0,z_{j'},0,\ldots,0)$,
$z'=(0,\ldots,0,z_k,0,\ldots,0,z_{k'},0,\ldots,0)$
and\\
$-|z_j|^2-|z_{j'}|^2+|z_k|^2+|z_{k'}|^2=0$,
then we have
\[
-\sum_{i=1}^q |a_{ij}z_j + a_{ij'}z_{j'} + b_{ik}z_k + b_{ik'}z_{k'}|^{2} +
\sum_{i=1}^p |c_{ij}z_j + c_{ij'}z_{j'} + d_{ik}z_k + d_{ik'}z_{k'}|^{2} = 0.
\]
Replacing $z_{j'}, z_k$ and $z_{k'}$ with $e^{\sqrt[]{-1}\theta_{j'}}z_{j'}, e^{\sqrt[]{-1}\theta_k}z_k$ and $e^{\sqrt[]{-1}\theta_{k'}}z_{k'}$,
respectively,
we can derive from this equation that
\begin{eqnarray*}
&&-\sum_{i=1}^q a_{ij} \overline{a_{ij'}} +
\sum_{i=1}^p c_{ij} \overline{c_{ij'}} = 0,\\
&&
-\sum_{i=1}^q b_{ik} \overline{b_{ik'}} +
\sum_{i=1}^p d_{ik} \overline{d_{ik'}} = 0.
\end{eqnarray*}
Therefore all non diagonal components of $-{}^t\!A\overline A + {}^t\!C\overline C$
and $-{}^t\!B\overline B + {}^t\!D\overline D$
vanish.
Thus we obtain the assertion.

\end{proof}

\section{The actions of $GU(n,1)$}
\label{sec:3}
 
Now we prove the following theorem.

\begin{theorem}\label{M1}
Let $M$ be a connected complex manifold of dimension $n+1 > 2$ 
that is holomorphically separable
and admits a smooth envelope of holomorphy.
Assume that there exists an injective homomorphism of topological groups
$\rho_0 : GU(n,1) \longrightarrow \mathrm{Aut}(M)$.
Then $M$ is biholomorphic to one of the domains
$\C^{n+1}$, $\C^{n+1} \setminus \{0\}$, $D^{n,1}$,
$C^{n,1} \simeq \C^* \times \B^n$ or $\C \times \B^n$.
\end{theorem}

\begin{proof}
By Lemma~\ref{4} and the comments after that, 
we can assume that $M$ is a Reinhardt domain $\Omega$ in $\C^{n+1}$,
$U(1) \times U(n)$-action on $\Omega \subset \C^{n+1}$ is linear
and $\rho(T^{n+1}) = T^{n+1}$.
We will 
prove that $\Omega$ is biholomorphic to  one of  the five domains
$\C^{n+1}$, $\C^{n+1} \setminus \{0\}$, $D^{n,1}$,
$C^{n,1} \simeq \C^* \times \B^n$ or $\C \times \B^n$.

Put a coordinate $(z_0, z_1, \ldots, z_n)$ of $\mathbb{C}^{n+1}$.
Since $\rho(\mathbb{C}^*)$ is 
commutative with $\rho(T^{n+1}) = T^{n+1} \subset \aut(\Omega)$,
Lemma~\ref{KS1} tells us that $\rho(\mathbb{C}^*) \subset \Pi(\Omega)$,
that is, $\rho(\mathbb{C}^*)$ is represented by diagonal matrices.
Furthermore, $\rho(\mathbb{C}^*)$ commutes with 
$\rho(U(1) \times U(n)) = U(1) \times U(n)$,
so that we have 
\begin{eqnarray}\label{C*1}
\rho\left(e^{2\pi i(s+it)}\right) =
\mathrm{diag}\left[
e^{2\pi i\{a_1s+(b_1+ic_1)t\}}, e^{2\pi i\{a_2s+(b_2+ic_2)t\}}E_n
\right]
\in \rho(\mathbb{C}^*),
\end{eqnarray}
where
$s,t \in \mathbb{R}$, $a_1, a_2 \in \mathbb{Z}, b_1,b_2,c_1,c_2 \in \mathbb{R}$.
Since $\rho$ is injective,
$a_1, a_2$ are
relatively prime
and $(c_1, c_2) \neq (0, 0)$.
Since $T_{1,n}$
is the center of the group $U(1) \times U(n)$ (see (\ref{Tqp})),
we have $\rho(T_{1,n}) = T_{1,n} \subset \mathrm{Aut}(\Omega)$.
Hence $\rho(T_{1,n})$
is described as
\begin{eqnarray}\label{T1n}
\rho\left(
\mathrm{diag}\left[
e^{2\pi is_1}, e^{2\pi is_2}E_n
\right]\right)
=
\mathrm{diag}\left[
e^{2\pi i(as_1+bs_2)}, e^{2\pi i(cs_1+ds_2)}E_n
\right]
\in \rho\left(T_{1,n}\right),
\end{eqnarray}
where
\begin{eqnarray*}
\begin{pmatrix}
a & b\\
c & d
\end{pmatrix}
\in GL(2, \mathbb{Z}),
\end{eqnarray*}
and $s_1, s_2 \in \mathbb{R}$, and we have $a+b=a_1$ and $c+d=a_2$.
To consider the actions of $\mathbb{C}^*$ and $U(1) \times U(n)$ 
on $\Omega$ together,
we put
\begin{equation*} 
G(U(1) \times U(n)) = 
\left\{e^{-2\pi t} \cdot
\mathrm{diag}\left[
u_0, U
\right]
\in GU(n,1) : t \in \mathbb{R}, u_0 \in U(1), U \in U(n)
\right\}.
\end{equation*}
Then we have
\begin{eqnarray*} 
G &:=&\rho(G(U(1) \times U(n))) \\
&{}=&
\left\{
\mathrm{diag}\left[
e^{-2\pi c_1t}u_0, e^{-2\pi c_2t}U
\right]
\in GL(n+1,\mathbb{C}) : t \in \mathbb{R}, u_0 \in U(1), U \in U(n)
\right\}.
\end{eqnarray*}
Note that $G$ is the centralizer
of $T_{1,n} = \rho(T_{1,n})$ in $\rho(GU(n, 1))$.

Let $f =(f_0, f_1, \ldots, f_n) \in \rho(GU(n, 1))$ and consider 
the Laurent series of the components:\\
\begin{eqnarray}\label{eqfi}
f_i(z_0, \ldots, z_n) 
= \sum_{\nu \in \mathbb{Z}^{n+1}} a^{(i)}_{\nu}z^{\nu},
\end{eqnarray}
for $0 \leq i \leq n$.
By Lemma~\ref{LU},
there are no negative degree terms 
of $z_1,\ldots,z_n$ in $(\ref{eqfi})$.
Write $\nu=(\nu_0, \nu')=(\nu_0, \nu_1, \ldots, \nu_n) \in \Z^{n+1}$
and $|\nu'|=\nu_1 + \cdots + \nu_n$.
Let us consider $\nu' \in \mathbb{Z}^n_{\geq 0}$
and put
\begin{eqnarray*}
{\sum_{\nu}}' = \sum_{\nu_0\in \mathbb{Z}, \nu' \in \mathbb{Z}^n_{\geq 0}}\,\,\,
\mathrm{and}\,\,\,\,\,\,\,\,\,\,
(z')^{\nu'} = z_1^{\nu_1}\cdots z_n^{\nu_n},
\end{eqnarray*}
from now on.
When we need to distinguish $\nu$ for $f_i$, $0 \leq i \leq n$,
we write
$\nu = \nu^{(i)} = (\nu_0^{(i)}, \nu_1^{(i)}, \ldots, \nu_n^{(i)})$.

If $f \in \rho(GU(n, 1)) \setminus G$ is a linear map of the form
\begin{equation*}
\begin{pmatrix}
a^{(0)}_{(1,0,\ldots,0)} & 0 & \cdots & 0 &\\
0 & a^{(1)}_{(0,1,0,\ldots,0)} & \cdots & a^{(1)}_{(0,\ldots,0,1)}\\
\vdots  & \vdots &  \ddots & \vdots\\
0 & a^{(n)}_{(0,1,0,\ldots,0)} & \cdots & a^{(n)}_{(0,\ldots,0,1)}
\end{pmatrix}
\in GL(n+1, \mathbb{C}),
\end{equation*}
then $f$ commutes with $\rho(T_{1,n})$, which contradicts $f \notin G$.
Thus 
we have:

\begin{lemma}\label{cl1}
For any $f \in \rho(GU(n, 1)) \setminus G$,
there exists $\nu \in \mathbb{Z}^{n+1}$, $\neq e_1$,
such that $a^{(0)}_{\nu} \neq 0$ in $(\ref{eqfi})$,
or 
there exists
$\nu \in \mathbb{Z}^{n+1}$, $\neq e_2$, 
$\ldots,$ $e_{n+1}$,
such that
$a^{(i)}_{\nu} \neq 0$ in $(\ref{eqfi})$ for some $1 \leq i \leq n$,
where $e_1=(1,0,\ldots,0)$, $e_2=(0,1,0\ldots,0)$,
$\ldots,e_{n+1}=(0,\ldots,0,1)$ are
the natural basis of $\Z^{n+1}$.
\end{lemma}

\medskip
Since $\mathbb{C}^*$ is the center of $GU(n,1)$, 
it follows that, 
for $f \in \rho(GU(n,1))$, $s,t \in \R$,
\begin{eqnarray*}
\rho\left(e^{2\pi i(s+it)}\right) \circ f =
f \circ \rho\left(e^{2\pi i(s+it)}\right).
\end{eqnarray*}
By $(\ref{eqfi})$,
this equation means, for $i=0$,
\begin{eqnarray*}
e^{2\pi i\{a_1s+(b_1+ic_1)t\}} {\sum_{\nu}}' a^{(0)}_{\nu}z^{\nu}
&=& 
{\sum_{\nu}}' a^{(0)}_{\nu}
\left(e^{2\pi i\{a_1s+(b_1+ic_1)t\}}z_0\right)^{\nu_0^{(0)}}
\left(e^{2\pi i\{a_2s+(b_2+ic_2)t\}}z'\right)^{\nu'}\\
&=&
{\sum_{\nu}}' a^{(0)}_{\nu}
e^{2\pi i\{a_1s+(b_1+ic_1)t\}\nu_0^{(0)}} e^{2\pi i\{a_2s+(b_2+ic_2)t\}|\nu'|}
z^{\nu},
\end{eqnarray*}
and, for $1 \leq i \leq n$,
\begin{eqnarray*}
e^{2\pi i\{a_2s+(b_2+ic_2)t\}} {\sum_{\nu}}' a^{(i)}_{\nu}z^{\nu}
&=&
{\sum_{\nu}}' a^{(i)}_{\nu}
\left(e^{2\pi i\{a_1s+(b_1+ic_1)t\}}z_0\right)^{\nu_0^{(i)}}
\left(e^{2\pi i\{a_2s+(b_2+ic_2)t\}}z'\right)^{\nu'}\\
 &=&
{\sum_{\nu}}' a^{(i)}_{\nu}
e^{2\pi i\{a_1s+(b_1+ic_1)t\}\nu_0^{(i)}} e^{2\pi i\{a_2s+(b_2+ic_2)t\}|\nu'|} z^{\nu}.
\end{eqnarray*}
Thus for each $\nu \in \mathbb{Z}^{n+1}$, we have
\begin{eqnarray*}
e^{2\pi i\{a_1s+(b_1+ic_1)t\}} a^{(0)}_{\nu}
= e^{2\pi i\{a_1s+(b_1+ic_1)t\}\nu_0^{(0)}} e^{2\pi i\{a_2s+(b_2+ic_2)t\}|\nu'|} a^{(0)}_{\nu},
\end{eqnarray*}
and, for $1 \leq i \leq n$,
\begin{eqnarray*}
e^{2\pi i\{a_2s+(b_2+ic_2)t\}} a^{(i)}_{\nu}
= e^{2\pi i\{a_1s+(b_1+ic_1)t\}\nu_0^{(i)}} e^{2\pi i\{a_2s+(b_2+ic_2)t\}|\nu'|} a^{(i)}_{\nu}.
\end{eqnarray*}
Therefore, if
$a_{\nu}^{(0)}\neq0$ for 
$\nu=(\nu_0^{(0)}, \nu') =
(\nu_0^{(0)}, \nu_1^{(0)}, \ldots, \nu_n^{(0)}) \in \Z \times \Z^n_{\geq 0}$,
we have
\begin{eqnarray}\label{ac0}
\hspace{-3cm}
\left\{\begin{array}{l}
a_1(\nu_0^{(0)}-1) + a_2(\nu_1^{(0)} + \cdots + \nu_n^{(0)}) = 0,\\
c_1(\nu_0^{(0)}-1) + c_2(\nu_1^{(0)} + \cdots + \nu_n^{(0)}) = 0,
\end{array}
\right.
\end{eqnarray}
and if $a_{\nu}^{(i)} \neq 0$
for $\nu=(\nu_0^{(i)}, \nu')
= (\nu_0^{(i)}, \nu_1^{(i)}, \ldots, \nu_n^{(i)}) \in \Z \times \Z^n_{\geq 0}$,
we have
\begin{eqnarray}\label{aci}
\hspace{-33mm}
\left\{\begin{array}{l}
a_1\nu_0^{(i)} + a_2(\nu_1^{(i)} + \cdots + \nu_n^{(i)}-1) = 0,\\
c_1\nu_0^{(i)} + c_2(\nu_1^{(i)} + \cdots + \nu_n^{(i)}-1) = 0,
\end{array}
\right.
\end{eqnarray}
for $1 \leq i \leq n$.

\bigskip
First, we consider the case $c_1c_2 = 0$.

\begin{lemma}\label{c_20}
If $c_1 \neq 0$, $c_2 = 0$,
then $\Omega$ is biholomorphic to $\C^* \times \B^n$ or $\C \times \B^n$.
\end{lemma}

\begin{proof}
In this case,
$\Omega \subset  \C^{n+1}$ can be written
of the form $(\C \times D) \cup  (\C^* \times D')$,
where $D$ and $D'$ are open sets in $\C^n$.
Indeed, 
$\Omega = (\Omega \cap \{z_0 = 0\}) \cup (\Omega \cap \{z_0 \neq 0\})$.
Then $\{0\} \times D := \Omega \cap \{z_0 = 0\} \subset \Omega$
implies $\C \times D \subset \Omega$ 
by $G$-action on $\Omega$.
On the other hand,
$\Omega \cap \{z_0 \neq 0\} = \C^* \times D'$
for some open set $D' \subset \C^n$ 
by the $G$-action.
Thus, $\Omega = (\C \times D) \cup  (\C^* \times D')$.
Furthermore $D \subset D'$, and
since $\Omega$ is connected, $D'$ is a connected open set in $\subset \C^n$.

We shall prove now that $D'$ is biholomorphic to $\B^n$
and $D$ coincides with $D'$ or $\emptyset$.
For $f = (f_0, \ldots, f_n) \in \rho(GU(n,1))$, the functions
$f_i$, for $1 \leq i < n$, do not depend on $z_0$ by (\ref{aci}),
therefore $GU(n,1)$ acts on $D$ and $D'$.
Since $U(n)$ acts linearly on $D'$,
the domain $D'$ must be one of the open sets $\C^n$, 
$\C^n \setminus \{0\}$, $\C^n \setminus \overline{\B^n_r}$,
$\B^n_r$, $\B^n_r \setminus \{0\}$ and $\B^n_r \setminus \overline{\B^n_{r'}}$,
where $\B^n_r = \{ z \in \C^n : |z| < r\}$ and $r>r'>0$.
Thus we have a topological group homomorphism from $GU(n,1)$ to 
one of the topological groups
$\aut(\C^n)$,
$\aut(\C^n \setminus \{0\})$, $\aut(\C^n \setminus \overline{\B^n_r})$,
$\aut(\B^n_r)$, $\aut(\B^n_r \setminus \{0\})$
and 
$\aut(\B^n_r \setminus \overline{\B^n_{r'}})$.

We now prove $D'\neq\C^n$.
Suppose the contrary were the case.
Then, by Lemma~\ref{extG},
the $U(n,1)$-action extends to a holomorphic action of $GL(n+1,\C)$.
The categorical quotient $\C^n/\hspace{-1mm}/GL(n+1,\C)$ is then one point
(see the sentence before Lemma~\ref{linG}).
By Lemma~\ref{linG}, the $GL(n+1,\C)$-action is linearizable.
However, the restriction of this action to $SU(n,1)$ is non-trivial
since $SU(n)\subset \rho(SU(n,1))$, $n>1$, acts non-trivially on $\C^n$.
This contradicts Lemma~\ref{notsimple}.
Furthermore, $D'\neq \C^n \setminus \{0\}$,
$\C^n \setminus \overline{\B^n}$,
since the $U(n,1)$-action extends to the action on $\C^n$
by the Hartogs extension theorem.
These cases come down to the $\C^n$-case.

$SU(n,1)$ can neither act non-trivially on $\B^n_r \setminus \{0\}$ nor
$\B^n_r \setminus \overline{\B^n_{r'}}$.
Indeed, if $SU(n,1)$ acts,
then $\rho(SU(n,1)) \subset U(n)$.
Here we used $\aut(\B^n_r \setminus \{0\})=\aut(\B^n_r \setminus \overline{\B^n_{r'}})=U(n)$.
By Lemma~\ref{notsimple}, this action is trivial.
Thus $D'\neq\B^n_r \setminus \{0\}$, $\B^n_r \setminus \overline{\B^n_{r'}}$,
since $SU(n) \subset \rho(SU(n,1))$ acts non-trivially on $D'$.

Clearly, $\B_r^n$ is biholomorphic to the unit ball $\B^n$,
and therefore,
$GU(n,1)$ acts on $\C\times\B_r^n$ as fractional
linear transformations (see Introduction).
Consequently, we see that $D' = \B_r^n$ for some $r>0$.

If $D = \emptyset$, then $\Omega = \C^* \times \B_r^n$.
If $D \neq \emptyset$, then $D = D' = \B_r^n$.
Indeed,
by $U(n)$-action,
$D$ is an union of certain sets of the forms $\B^n_{r'}$, $\B^n_s \setminus \overline{\B^n_{s'}}$
and $\B^n_s \setminus \{0\}$,
where $r'\leq r$, $s'<s\leq r$.
If $D \subsetneq D'$,
the $GU(n,1)$-action does not preserve $D$,
since $U(n) \subsetneq \rho(GU(n,1))|_D \subset PU(n,1)$ as above.
Here $\rho(GU(n,1))|_D$ is the induced action on $D'$ of $\rho(GU(n,1))$.
Thus $D = D' = \B_r^n$,
and therefore $\Omega = \C \times \B_r^n$.
This proves the lemma.
\end{proof}

\begin{remark}\label{rem0}
We remark here that this Lemma~\ref{c_20} is also valid when $n=1$.
It should be checked that $D'\neq \C$, $\C^*$, $\C\setminus  \overline{\D_r}$,
$\D_r \setminus \{0\}$ and $\D_r \setminus \overline{\D_r'}$,
where $\D_r = \{ z \in \C : |z|<r\}$ and $r>r'>0$.
This can be proven as the proof of the following lemma.
\end{remark}

\begin{lemma}\label{3.3}
The case $c_1 = 0$ and $c_2 \neq 0$ does not occur.
\end{lemma}

\begin{proof}
If $a_1 \neq 0$, then by (\ref{ac0}) and (\ref{aci}), $\nu_0^{(0)} = 1$, 
$\nu^{(0)}_1+\cdots+\nu^{(0)}_n=0$,
$\nu_0^{(i)} = 0$ and
$\nu^{(i)}_1+\cdots+\nu^{(i)}_n=1$ for $1 \leq i \leq n$. 
However this is a contradiction to Lemma \ref{cl1}.

We consider the case $a_1 = 0$.
As the proof of Lemma \ref{c_20}, $\Omega \subset  \mathbb{C}^{n+1}$ can be written
of the form $(D'' \times \mathbb{C}^n) \cup  (D''' \times (\mathbb{C}^n \setminus \{0\}) )$
by the $G$-action on $\Omega$,
where $D''$ and $D'''$ are open sets in $\mathbb{C}$, $D''\subset D'''$, and
since $\Omega$ is connected, $D'''$ is a connected open set in $\C$.
On the other hand,
the function $f_0$ of $f = (f_0, \ldots, f_n) \in \rho(GU(n,1))$ does
not depend on the variables $(z_1, \ldots, z_n)$ by the second equation of (\ref{ac0}).
Hence $f_0$ is an automorphism of $D'''$.
Since $U(1)$ acts linearly on $D'''$,
the domain $D'''$ must be one of the open sets
$\C$, $\C^*$, $\C \setminus \overline{\D_r}$, 
$\D_r$, $\D_r \setminus \{0\}$ and $\D_r \setminus \overline{\D_r'}$, 
where $\D_r = \{ z \in \C : |z|<r\}$ and $r>r'>0$.

By $a_1 = 0$, we have $b=-a$ in (\ref{T1n}), and therefore, if $s_2 = -s_1/n$,
we have $as_1+bs_2 \not\equiv k \in \Z$ for $s_1\in \R$.
Note that
\begin{eqnarray}
\mathrm{diag}\left[
e^{2\pi is_1}, e^{-2\pi is_1/n}E_n
\right]
\in S(U(1) \times U(n)).
\end{eqnarray}
Thus, by (\ref{T1n}), 
we have a nontrivial $SU(n,1)$-action on the domain $D'''$ in $\C$.

We now prove $D'''\neq\C$.
By Lemma~\ref{extG},
the $SU(n,1)$-action extends to a holomorphic action of $SL(n+1,\C)$.
The categorical quotient $\C^n/\hspace{-1mm}/SL(n+1,\C)$ is then one point
(see the sentence before Lemma~\ref{linG}).
By Lemma~\ref{linG}, the $SL(n+1,\C)$-action is linearizable.
However, the restriction of this action to $SU(n,1)$ is non-trivial.
This contradicts Lemma~\ref{notsimple}.

Next we prove $D'''\neq\C^*$.
Assume the contrary.
Since $\aut(\C^*)=\{az^{\pm 1}:a\in \C^* \}\simeq \C^* \rtimes \{\pm 1\}$
and $SU(n,1)$ is connected Lie group,
we have a non-trivial Lie group homomorphism from $SU(n,1)$
to $\C^*$. However this contradicts Lemma~\ref{notsimple}.

$SU(n,1)$ can not act non-trivially on $\C \setminus \overline{\D_r}$,
$\D_r$, $\D_r \setminus \{0\}$ and $\D_r \setminus \overline{\D_r'}$.
Indeed, the automorphism groups of these domains are isomorphic to $U(1)$.
If $SU(n,1)$ acts,
then $\rho(SU(n,1)) \subset U(1)$.
By Lemma~\ref{notsimple}, this action is trivial.
Since $SU(n,1)$ acts nontrivially on $D'''$, we see that $D'''\neq \C \setminus \overline{\D_r}$,
$\D_r$, $\D_r \setminus \{0\}$ and $\D_r \setminus \overline{\D_r'}$.

Finally, we prove $D'''\neq \D_r$.
Assume the contrary. Then we have
a non-trivial Lie group homomorphism from $SU(n,1)$
to $\aut(\D_r) = PU(1,1)$.
However this contradicts Lemma~\ref{notsimple2}, since$n>1$.
We have proven Lemma~\ref{3.3}.
\end{proof}

\bigskip
We consider, henceforce, the case $c_1c_2 \neq 0$.

\begin{lemma}\label{cla1}
If $c_1c_2 \neq 0$, then $\lambda := c_2/c_1 = a_2/a_1 \in \Z  \setminus \{0\}$.
\end{lemma}

\begin{proof}
Take $f \in \rho(GU(n, 1)) \setminus G$
with the Laurent expansions (\ref{eqfi}).
Recall Lemma~\ref{cl1}.
If $a_{\nu}^{(0)}\neq0$ for some 
$\nu=(\nu_0^{(0)}, \nu_1^{(0)}, \ldots, \nu_n^{(0)}) \in \Z \times \Z^n_{\geq 0}$,
$\neq (1,0,\ldots,0)$,
then it follows form
(\ref{ac0}) and the assumption $c_1c_2 \neq 0$
that $\nu_0^{(0)}-1 \neq 0$ and $\nu_1^{(0)} + \cdots + \nu_n^{(0)} \neq 0$.
Hence $c_2/c_1 = a_2/a_1 \in \mathbb{Q}$ and $(a_1,a_2) \neq (\pm1,0), (0,\pm1)$
by (\ref{ac0}).
On the other hand,
if $a_{\nu}^{(i)}\neq0$ for some 
$1 \leq i \leq n$ and
$\nu=(\nu_0^{(i)}, \nu_1^{(i)}, \ldots, \nu_n^{(i)}) \in \Z \times \Z^n_{\geq 0}$,
$\neq (0,1,0 \ldots,0)$, 
$\ldots,$ $(0,0 \ldots,0,1)$,
then it follows from (\ref{aci}) and the assumption $c_1c_2 \neq 0$
that
$\nu_0^{(i)} \neq 0$ and $\nu_1^{(i)} + \cdots + \nu_n^{(i)} -1 \neq 0$.
In this case, we also obtain $c_2/c_1 = a_2/a_1 \in \mathbb{Q}$ and
$(a_1,a_2) \neq (\pm1,0), (0,\pm1)$ by (\ref{aci}).
Consequently, we have
\begin{eqnarray*}
\lambda= a_2/a_1 = c_2/c_1 \in \mathbb{Q} \setminus \{0\}.
\end{eqnarray*}

\medskip

We now prove that $\lambda$ is a nonzero integer.
For the purpose, we assume $\lambda \notin \mathbb{Z}$, that is, $a_1 \neq \pm1$.
First we consider the case $\lambda < 0$.
Since 
$\nu_1^{(i)} + \cdots + \nu_n^{(i)} \geq 0$ for $0 \leq i \leq n$
and $a_1, a_2$ are relatively prime (see (\ref{C*1})),
we have, by (\ref{ac0}),
\begin{eqnarray*}
\nu_0^{(0)} = 1+k|a_2| \geq 1
\,\,\,\,\,\mathrm{and}\,\,\,\,\, 
\nu_1^{(0)} + \cdots + \nu_n^{(0)} = k|a_1| \geq 0,
\end{eqnarray*}
where $k \in \Z_{\geq 0}$,
and, by (\ref{aci}), for $1 \leq i \leq n$,
\begin{eqnarray*}
\nu_0^{(i)} = l|a_2| \geq 0
\,\,\,\,\,\mathrm{and}\,\,\,\,\,
\nu_1^{(i)} + \cdots + \nu_n^{(i)} = 1+l|a_1| \geq 1,
\end{eqnarray*}
where $l \in \Z_{\geq 0}$.
Hence, the Laurent series of the components of $f \in \rho(GU(n,1))$ are
\begin{eqnarray}\label{form0}
f_0(z_0, \ldots, z_n) 
= \sum_{k=0}^{\infty} \, {\sum_{|\nu'|=k|a_1|}}^{\hspace{-10pt}\prime}\,\,\,\,\,
a^{(0)}_{\nu'}z_0^{1+k|a_2|}(z')^{\nu'}
\end{eqnarray}
and
\begin{eqnarray}\label{formi}
f_i(z_0, \ldots, z_n) 
= \sum_{k=0}^{\infty} \, {\sum_{|\nu'|=1+k|a_1|}}^{{\hspace{-14pt}\prime}}\,\,\,\,\,
a^{(i)}_{\nu'}z_0^{k|a_2|}(z')^{\nu'}
\end{eqnarray}
for $1 \leq i \leq n$.
Here we have written $a^{(0)}_{\nu'}=a^{(0)}_{(1+k|a_2|,\nu')}$
and $a^{(i)}_{\nu'}=a^{(i)}_{(k|a_2|,\nu')}$, for $1\leq i \leq n$,
and we will use this notation if it is clear from the context what it means.
We focus on the first degree terms of the Laurent expansions.
We put
\begin{eqnarray}\label{Pf}
Pf (z) :=
\left(a^{(0)}_{(1,0,\ldots,0)}z_0, {\sum_{|\nu'|=1}}^{{\hspace{-4pt}\prime}}\,\,\,\, a^{(1)}_{\nu'}(z')^{\nu'}, \ldots,
{\sum_{|\nu'|=1}}^{{\hspace{-4pt}\prime}}\,\,\,\, a^{(n)}_{\nu'}(z')^{\nu'}\right).
\end{eqnarray}
As a matrix we can write
\begin{eqnarray*}
Pf = 
\begin{pmatrix}
a^{(0)}_{(1,0,\ldots,0)} & 0 & \cdots & 0 &\\
0 & a^{(1)}_{(0,1,0,\ldots,0)} & \cdots & a^{(1)}_{(0,\ldots,0,1)}\\
 \vdots  & \vdots &  \ddots & \vdots\\
0 & a^{(n)}_{(0,1,0,\ldots,0)} & \cdots & a^{(n)}_{(0,\ldots,0,1)}
\end{pmatrix}.
\end{eqnarray*}
Then it follows from (\ref{form0}) and (\ref{formi}) that
\begin{eqnarray*}\label{}
P(f \circ h) = Pf \circ Ph, \,\,\,{\rm and}\,\,\, P{\rm id} = {\rm id},
\end{eqnarray*}
where $h \in \rho(GU(n,1))$,
and therefore 
\begin{eqnarray*}\label{}
Pf  \in GL(n+1, \mathbb{C})
\end{eqnarray*}
since $f$ is an automorphism.
Hence we have a representation of $GU(n,1)$ given by
\begin{eqnarray*}
GU(n,1) \ni g \longmapsto Pf \in GL(n+1,\mathbb{C}),
\end{eqnarray*}
where $f = \rho(g)$.
The restriction of this representation to the simple Lie group $SU(n,1)$
is nontrivial since $\rho(U(1) \times U(n)) = U(1) \times U(n)$.
However this contradicts Lemma~\ref{notsimple}.
Thus it does not occur that $\lambda$ is a negative non-integer.

\medskip
Next we consider the case $\lambda >0$ and $\lambda \not \in \mathbb{Z}$.
Since $\nu_1^{(i)} + \cdots + \nu_n^{(i)} \geq 0$ for $0 \leq i \leq n$
and $a_1, a_2$ are relatively prime,
we have, by (\ref{ac0}), 
\begin{eqnarray*}
\nu_0^{(0)} = 1-k|a_2| \leq 1
\,\,\,\,\,\mathrm{and}\,\,\,\,\, 
\nu_1^{(0)} + \cdots + \nu_n^{(0)} = k|a_1| \geq 0,
\end{eqnarray*}
where $k \in \Z_{\geq 0}$,
and, by (\ref{aci}), for $1 \leq i \leq n$,
\begin{eqnarray*}
\nu_0^{(i)} = -l|a_2| \leq 0
\,\,\,\,\,\mathrm{and}\,\,\,\,\,
\nu_1^{(i)} + \cdots + \nu_n^{(i)} = 1+l|a_1| \geq 1,
\end{eqnarray*}
where $l \in \Z_{\geq 0}$.
Hence, the Laurent series of components of $f \in \rho(GU(n,1))$ are
\begin{eqnarray*}
f_0(z_0, \ldots, z_n) 
= \sum_{k=0}^{\infty} \, {\sum_{|\nu'|=k|a_1|}}^{\hspace{-10pt}\prime}\,\,\,\,\,
a^{(0)}_{\nu'}z_0^{1-k|a_2|}(z')^{\nu'}
\end{eqnarray*}
and
\begin{eqnarray*}
f_i(z_0, \ldots, z_n) 
= \sum_{k=0}^{\infty} \, {\sum_{|\nu'|=1+k|a_1|}}^{\hspace{-14pt}\prime}\,\,\, 
a^{(i)}_{\nu'}z_0^{-k|a_2|}(z')^{\nu'}
\end{eqnarray*}
for $1\leq i \leq n$.
We claim that $a^{(0)}_{(1,0,\ldots,0)} \neq 0$.
Suppose the contrary.
$f$ and $f^{-1}$ are defined near the points $(z_0,0,\ldots,0)$ for fixed $z_0$,
and therefore 1-to-1 near that point.
However, since $a^{(0)}_{(1,0,\ldots,0)} = 0$,
$f(z_0, 0, \ldots, 0) = (0, \ldots, 0) \in \mathbb{C}^{n+1}$ for each $z_0$,
a contradiction.
Take another $h \in \rho(GU(n,1))$
and put the Laurent series of its components: 
\begin{eqnarray*}
h_0(z_0, \ldots, z_n) 
= \sum_{k=0}^{\infty} \, {\sum_{|\nu'|=k|a_1|}}^{\hspace{-10pt}\prime}\,\,\,
b^{(0)}_{\nu'}z_0^{1-k|a_2|}(z')^{\nu'}
\end{eqnarray*}
and
\begin{eqnarray*}
h_i(z_0, \ldots, z_n)
= \sum_{k=0}^{\infty} \, {\sum_{|\nu'|=1+k|a_1|}}^{\hspace{-14pt}\prime}\,\,\, 
b^{(i)}_{\nu'}z_0^{-k|a_2|}(z')^{\nu'}
\end{eqnarray*}
for $1\leq i \leq n$. We have $b^{(0)}_{(1,0,\ldots,0)} \neq 0$ as above.
We consider the first degree terms of $f \circ h$.
For the first component
\begin{eqnarray*}
f_0(h_0, \ldots, h_n) 
&=& a^{(0)}_{(1,0,\ldots,0)}h_0 + 
\sum_{k=1}^{\infty} \, {\sum_{|\nu'|=k|a_1|}}^{\hspace{-10pt}\prime}\,\,\,\,\,\,
a^{(0)}_{\nu'}h_0^{1-k|a_2|}(h')^{\nu'},
\end{eqnarray*}
where $h=(h_1,\ldots,h_n)$.
Then, for $k>0$,
\begin{eqnarray*}
h_0(z)^{1-k|a_2|} &=& \left(\sum_{l=0}^{\infty} \, {\sum_{|\nu'|=l|a_1|}}^{\hspace{-8pt}\prime}
b^{(0)}_{\nu'}z_0^{1-l|a_2|}(z')^{\nu'}\right)^{1-k|a_2|}
=
z_0^{1-k|a_2|} \left(\sum_{l=0}^{\infty} \, {\sum_{|\nu'|=l|a_1|}}^{\hspace{-8pt}\prime} 
b^{(0)}_{\nu'}z_0^{-l|a_2|}(z')^{\nu'}\right)^{1-k|a_2|}\\
&=&
\left(b^{(0)}_{(1,0,\ldots,0)}z_0\right)^{1-k|a_2|} \left(1 +
\frac{1-k|a_2|}{b^{(0)}_{(1,0,\ldots,0)}} z_0^{-|a_2|} {\sum_{|\nu'|=|a_1|}}^{\hspace{-8pt}\prime}b^{(0)}_{\nu'}(z')^{\nu'}
+ \cdots \right)
\end{eqnarray*}
Thus
$h_0(z)^{1-k|a_2|}$ has the maximum degree of $z_0$ at most $1 - k|a_2| < 1$
and has the minimum degree of $z'$ at least $|a_1| > 1$
in its Laurent expansion.
For $|\nu'|=k|a_1|$ and $k>0$,
$(h')^{\nu'}$ has the maximum degree of $z_0$ at most $-|a_2| < 0$
and
the first degree terms of $z'$ are with coefficients of a negative degree $z_0$ term
in its Laurent expansion.
Hence the first degree term of Laurent expansion of $f_0(h_0, \ldots, h_n)$ 
is $a^{(0)}_{(1,0,\ldots,0)}b^{(0)}_{(1,0,\ldots,0)}z_0$.

Similarly, consider
\begin{equation*}
f_i(h_0, \ldots, h_n) 
= {\sum_{|\nu'|=1}}^{\hspace{-5pt}\prime}\,\,\, a^{(i)}_{\nu'}(h')^{\nu'} + 
\sum_{k=1}^{\infty} \, 
{\sum_{|\nu'|=1+k|a_1|}}^{\hspace{-14pt}\prime}\,\,\,a^{(i)}_{\nu'}h_0^{-k|a_2|}(h')^{\nu'}
\end{equation*}
for $1\leq i \leq n$.
Then, for $k>0$,
\begin{eqnarray*}
h_0^{-k|a_2|}
&=& \left(b^{(0)}_{(1,0,\ldots,0)}z_0\right)^{-k|a_2|} \left(1 + 
\frac{-k|a_2|}{b^{(0)}_{(1,0,\ldots,0)}} z_0^{-|a_2|} {\sum_{|\nu'|=|a_1|}}^{\hspace{-8pt}\prime}b^{(0)}_{\nu'}(z')^{\nu'}
+ \cdots \right).
\end{eqnarray*}
Thus $h_0^{-k|a_2|}$ has the maximum degree of $z_0$ at most $- k|a_2| < 0$
 and has the minimum degree of $z'$ at least $|a_1| > 1$
 in its Laurent expansion.
For $|\nu'|=1+k|a_1|$ and $k>0$, $(h')^{\nu'}$
has the maximum degree of $z_0$ at most $-|a_2| < 0$
and the first degree terms of $z'$ are with coefficients of negative degree $z_0$ term 
in its Laurent expansion.
Hence the first degree terms of the Laurent expansions of $f_i(h_0, \ldots, h_n)$ 
is 
\begin{eqnarray*}
\sum_{j=1}^n {\sum_{|\nu'|=1}}^{\hspace{-5pt}\prime}\,\,\,\, a^{(i)}_{e'_j}b^{(j)}_{\nu'}(z')^{\nu'},
\end{eqnarray*}
where $e'_1=(1,0,\ldots,0)$, $\ldots$, $e'_n=(0,\ldots,0,1) \in \Z^n$.
We put $Pf$ as (\ref{Pf}).
Consequently, 
\begin{eqnarray*}\label{}
P(f \circ h) = Pf \circ Ph, \,\,\,{\rm and}\,\,\, P{\rm id} = {\rm id},
\end{eqnarray*}
and therefore 
\begin{eqnarray*}\label{}
Pf  \in GL(n+1, \mathbb{C})
\end{eqnarray*}
since $f$ is an automorphism.
Then the same argument as that in previous case, $\lambda < 0$, shows that
this is a contradiction.
Thus it does not occur that $\lambda$ is positive non-integer.
We have shown that $\lambda \in \mathbb{Z} \setminus \{0\}$.

\end{proof}

\begin{remark}\label{rem1}
We remark for the Laurent series of the components of $f \in \rho(GU(n,1))$.
We have $\lambda = a_2/a_1 = c_2/c_1 \in \mathbb{Z}$
and $a_1 = \pm 1$ by Lemma~\ref{cla1}.
By Lemma~\ref{LU}, $\nu_1^{(i)} + \cdots + \nu_n^{(i)} \geq 0$ 
for $0 \leq i \leq n$.
For the Laurent series of $f_0$ in (\ref{eqfi}),
we have
$\nu_0^{(0)} = 1-k\lambda$ and 
$\nu_1^{(0)} + \cdots + \nu_n^{(0)} = k \geq 0$ by (\ref{ac0}), 
where $k \in \Z_{\geq 0}$.
For the Laurent series of $f_i$ in (\ref{eqfi}),
$\nu_0^{(i)} = -l\lambda$
and 
$\nu_1^{(i)} + \cdots + \nu_n^{(i)} = 1+l$, for $1 \leq i \leq n$,
where $l \in \Z$ and $l\geq -1$ by (\ref{aci}).
This is the difference between the case $\lambda \not \in \Z$ and $\lambda \in \Z$.
Hence, the Laurent series of components of
$f \in \rho(GU(n,1))$ are
\begin{eqnarray}\label{ac0'}
f_0(z_0, \ldots, z_n) 
= \sum_{k=0}^{\infty} \, {\sum_{|\nu'|=k}}^{\hspace{-4pt}\prime}\,\,\, a^{(0)}_{\nu'}z_0^{1-k\lambda}(z')^{\nu'},
\end{eqnarray}
and
\begin{eqnarray}\label{aci'}
f_i(z_0, \ldots, z_n) 
= a^{(i)}_{(\lambda,0,\ldots,0)}z_0^{\lambda} + 
\sum_{k=0}^{\infty} \, {\sum_{|\nu'|=1+k}}^{\hspace{-9pt}\prime}\,\,\,\,
a^{(i)}_{\nu'}z_0^{-k\lambda}(z')^{\nu'}
\end{eqnarray}
for $1\leq i \leq n$.
\end{remark}

We record a lemma which can be proven
in a similar way of the proof of Lemma~\ref{cla1}.

\begin{lemma}\label{5.0}
There exists an automorphism $f \in \rho(GU(n,1)) \setminus G$ such that,
in the Laurent expansions $(\ref{aci'})$ above,
at least one of $a^{(i)}_{(\lambda,0,\ldots,0)}$, for $1 \leq i \leq n$, does not vanish.
\end{lemma}

\begin{proof}
If, for any $f$, all $a^{(i)}_{(\lambda,0,\ldots,0)}$ vanish,
then 
\begin{eqnarray*}
Pf = 
\begin{pmatrix}
a^{(0)}_{(1,0,\ldots,0)} & a^{(0)}_{(0,1,0,\ldots,0)} & \cdots & a^{(0)}_{(0,\ldots,0,1)} &\\
0 & a^{(1)}_{(0,1,0,\ldots,0)} & \cdots & a^{(1)}_{(0,\ldots,0,1)}\\
 \vdots  & \vdots &  \ddots & \vdots\\
0 & a^{(n)}_{(0,1,0,\ldots,0)} & \cdots & a^{(n)}_{(0,\ldots,0,1)}
\end{pmatrix}
\end{eqnarray*}
gives a homomorphism from $SU(n,1)$ to $GL(n+1,\C)$
(when $\lambda=1$, $a^{(0)}_{(0,1,0,\ldots,0)}, \ldots, a^{(0)}_{(0,\ldots,0,1)}$ may not be zero),
as the proof of Lemma~\ref{cla1}.
Moreover it is nontrivial on $\{1\}\times SU(n)$. 
However,
this contradicts Lemma~\ref{notsimple}.

\end{proof}

\medskip
Since $G=\rho(G(U(1) \times U(n)))$ acts as linear transformations on 
$\Omega \subset \mathbb{C}^{n+1}$, 
it preserves the boundary $\partial \Omega$ of $\Omega$.
We now study the action of $G$ on $\partial \Omega$.
The $G$-orbits of points in $\mathbb{C}^{n+1}$ 
consist of four types as follows:

(i) If 
$p = (p_0, p_1, \ldots, p_n) \in \mathbb{C}^* \times
(\mathbb{C}^n \setminus \{0\})$,
then
\begin{equation} \label{GU1}
G \cdot p = \{ (z_0, \ldots, z_n) \in \mathbb{C}^{n+1} \setminus \{0\}
 : -a|z_0|^{2\lambda} + |z_1|^2 + \cdots + |z_n|^2 = 0\},
\end{equation}
where
$a:=(|p_1|^2 + \cdots + |p_n|^2)/|p_0|^{2\lambda} > 0$
and $\lambda \in \Z \setminus \{0\}$ by Lemma \ref{cla1}.

\smallskip

(ii) If $p' = (0, p'_1, \ldots, p'_n) \in \mathbb{C}^{n+1} \setminus \{0\}$,
then
\begin{equation} \label{GU2}
G \cdot p' = \{0\} \times (\mathbb{C}^n \setminus \{0\}).
\end{equation}

\smallskip

(iii) If $p'' = (p''_0, 0, \ldots, 0) \in \mathbb{C}^{n+1} \setminus \{0\}$,
then
\begin{equation} \label{GU3}
G \cdot p'' = \mathbb{C}^* \times \{(0,\ldots,0)\in \C^n\}.
\end{equation}

\smallskip

(iv) If $p''' = (0, \ldots, 0) \in \mathbb{C}^{n+1}$,
then
\begin{equation} \label{GU4}
G \cdot p''' = \{0\} \subset \mathbb{C}^{n+1}.
\end{equation}

\smallskip

\begin{lemma}\label{3.6}
Assume $c_1c_2 \neq 0$.
If 
\begin{equation*} 
\Omega \cap (\C^* \times (\C^n \setminus \{0\})) 
= \C^* \times (\C^n \setminus \{0\}),
\end{equation*}
then $\Omega$ is equal to $\C^{n+1}$ or $\C^{n+1}  \setminus \{0\}$.
\end{lemma}

\begin{proof}
If $\Omega \cap (\mathbb{C}^* \times (\mathbb{C}^n \setminus \{0\}))
= \mathbb{C}^* \times (\mathbb{C}^n \setminus \{0\})$,
then 
$\Omega$ equals one of the following domains 
by the $G$-actions of the type (\ref{GU2}), (\ref{GU3}) and (\ref{GU4}) above:
\begin{equation*} 
\C^{n+1},
\C^{n+1}  \setminus \{0\},
\C \times (\C^n \setminus \{0\}),
\C^* \times \C^n
\,\,\mathrm{or} \,\,\C^* \times (\C^n \setminus \{0\}).
\end{equation*}
Clearly $GU(n,1)$ acts on $\C^{n+1}$ and $\C^{n+1} \setminus \{0\}$
by matrix multiplications.
We will prove that, 
for the latter three domains, $GU(n,1)$ does not act effectively.
By (\ref{ac0'}) and (\ref{aci'}), Laurent expansions of 
$f \in \rho(GU(n,1)) \subset \aut(\Omega)$ are
\begin{eqnarray*}
f_0(z_0, \ldots, z_n) 
= \sum_{k=0}^{\infty} \, {\sum_{|\nu'|=k}}^{\hspace{-5pt}\prime}\,\,\,\, 
a^{(0)}_{\nu'}z_0^{1-k\lambda}(z')^{\nu'},
\end{eqnarray*}
and
\begin{eqnarray*}
f_i(z_0, \ldots, z_n) 
= a^{(i)}_{(\lambda,0,\ldots,0)}z_0^{\lambda} + 
\sum_{k=0}^{\infty} \, {\sum_{|\nu'|=1+k}}^{\hspace{-9pt}\prime}\,\,\,\,\,
a^{(i)}_{\nu'}z_0^{-k\lambda}(z')^{\nu'}
\end{eqnarray*}
for $1\leq i \leq n$, where $\lambda \in \Z \setminus \{0\}$.

Suppose $\Omega = \C^* \times (\C^n \setminus \{0\})$.
$f$ and $f^{-1} \in \aut(\C^* \times (\C^n \setminus \{0\}))$
extend holomorphically on $\C^* \times \C^n$, therefore 
$f \in \aut(\C^* \times \C^n)$.
There exists $f \in \rho(GU(n,1))$ such that some 
$a^{(i)}_{(\lambda,0,\ldots,0)}$ does not vanish, by Lemma \ref{5.0}.
However such $f$ does not
preserve $\C^* \times \{0\}$, this contradicts that 
$f \in \aut(\C^* \times (\C^n \setminus \{0\}))$ and
$f \in \aut(\C^* \times \C^n)$.

We can also prove $\Omega \neq \C \times (\C^n \setminus \{0\})$
in the same way.

Finally we suppose $\Omega = \C^* \times \C^n$.
Put an automorphism $\varpi_{\lambda} \in \aut(\Omega)$ by 
$\varpi_\lambda(z) = (z_0, z_0^{\lambda}z_1, \ldots, z_0^{\lambda}z_n)$.
Then
$\varpi_{\lambda}^{-1} \rho(GU(n,1)) \varpi_{\lambda}$
is a subgroup of $\aut(\Omega)$.
For
$h = \varpi_{\lambda}^{-1} \circ f \circ \varpi_{\lambda} =:(h_0,\ldots,h_n)$,
$f \in \rho(GU(n,1))$,
we have
\begin{eqnarray}\label{Lh0}
h_0(z_0, z_1, \ldots, z_n) 
= z_0 \sum_{k=0}^{\infty} \, {\sum_{|\nu'|=k}}^{\hspace{-4pt}\prime}\,\,\,\,
a^{(0)}_{\nu'}(z')^{\nu'},
\end{eqnarray}
and
\begin{eqnarray}\label{Lhi}
h_i(z_0, z_1, \ldots, z_n) 
= \frac{a^{(i)}_{(\lambda,0,\ldots,0)} + 
\sum_{k=0}^{\infty} \, {\sum'_{|\nu'|=1+k}} \,\,
a^{(i)}_{\nu'}(z')^{\nu'}}{\left(\sum_{k=0}^{\infty} \, {\sum'_{|\nu'|=k}} \,a^{(0)}_{\nu'}(z')^{\nu'}\right)^{\lambda}}
\end{eqnarray}
for $1 \leq i \leq n$.
Thus we have an action of $GU(n,1)$ on $\C^n$ by $(h_1,\ldots,h_n)$,
since $h_i$ for $1 \leq i \leq n$ 
do not depend on $z_0$.
Note that $U(n)$-action on $\C^n$ is still linear.
Then, by Lemma~\ref{extG} and Lemma~\ref{linG},
$GU(n,1)$-action on $\C^n$ is linearizable, 
since the categorical quotient $\C^n/\hspace{-3pt}/U(n)$ is just one point and 
$GU(n,1)$ is a reductive group.
However this contradicts Lemma~\ref{notsimple}
since $\{1\} \times SU(n) \subset SU(n,1)$, $n>1$,
acts non-trivially.

\end{proof}

Let us consider the case 
$\Omega \cap (\mathbb{C}^* \times (\mathbb{C}^n \setminus \{0\}))
\neq \mathbb{C}^* \times (\mathbb{C}^n \setminus \{0\})$.
Then we have
$\partial \Omega \cap (\mathbb{C}^* \times (\mathbb{C}^n \setminus \{0\}))
\neq \emptyset$.
Thus we can take a point 
\begin{eqnarray*} 
p=(p_0, \ldots, p_n) \in \partial \Omega \cap 
(\mathbb{C}^* \times (\mathbb{C}^n \setminus \{0\})).
\end{eqnarray*} 
Let 
\begin{eqnarray*} 
&&a := (|p_1|^2 + \cdots + |p_n|^2)/|p_0|^{2\lambda} > 0,\\
&&A_{a,\lambda} := \{ (z_0, \ldots, z_n) \in \mathbb{C}^{n+1}: 
-a|z_0|^{2\lambda} + |z_1|^2 + \cdots + |z_n|^2 = 0\}.
\end{eqnarray*}
Note that 
\begin{equation*} 
\partial \Omega  \supset A_{a,\lambda},
\end{equation*}
by the $G$-action of the type (\ref{GU1}).
If $\lambda > 0$,
then $\Omega$ is contained in
\begin{equation*} 
D^{+}_{a,\lambda} = \{ |z_1|^2 + \cdots + |z_n|^2 > a|z_0|^{2\lambda}\}
\end{equation*}
or 
\begin{equation*} 
C^{+}_{a,\lambda} = \{ |z_1|^2 + \cdots + |z_n|^2 < a|z_0|^{2\lambda}\}.
\end{equation*}
If $\lambda < 0$,
then $\Omega$ is contained in
\begin{equation*} 
D^{-}_{a,\lambda} = \{ (|z_1|^2 + \cdots + |z_n|^2)|z_0|^{-2\lambda} > a\}
\end{equation*}
or 
\begin{equation*} 
C^{-}_{a,\lambda} = \{ (|z_1|^2 + \cdots + |z_n|^2)|z_0|^{-2\lambda} < a\}.
\end{equation*}

Let us first consider the case $\partial \Omega = A_{a,\lambda}$,
that is, $\Omega = D^{+}_{a,\lambda}$, $C^{+}_{a,\lambda}$,
$D^{-}_{a,\lambda}$ or $C^{-}_{a,\lambda}$.

\begin{lemma}\label{eqthm}
If $\Omega=D^{+}_{a,\lambda}$,  
then $\lambda=1$ and $\Omega$ is biholomorphic to $D^{n,1}$.
\end{lemma}

\begin{proof}
If $\lambda=1$, there exists a biholomorphic map from $\Omega$ to $D^{n,1}$ by
\begin{eqnarray*}
\Phi : \C^{n+1} \ni (z_0, z_1, \ldots, z_n) \longmapsto (a^{1/2}z_0, z_1, \ldots, z_n) \in \C^{n+1}.
\end{eqnarray*}
If $\lambda \neq 1$, then by {\it Remark} \ref{rem1},
for $f \in \rho(GU(n,1))$, the Laurent expansions of its components are
\begin{eqnarray*}
f_0(z_0, \ldots, z_n) 
= \sum_{k=0}^{\infty} {\sum_{|\nu'|=k}}^{\hspace{-4pt}\prime}\,\,\,\, 
a^{(0)}_{\nu'}z_0^{1-k\lambda}(z')^{\nu'},
\end{eqnarray*}
and
\begin{eqnarray*}
f_i(z_0, \ldots, z_n) 
= a^{(i)}_{(\lambda,0,\ldots,0)}z_0^{\lambda} 
+ \sum_{k=0}^{\infty} {\sum_{|\nu'|=1+k}}^{\hspace{-9pt}\prime}\,\,\,\,
a^{(i)}_{\nu'}z_0^{-k\lambda}(z')^{\nu'},
\end{eqnarray*}
for $1\leq i \leq n$.
Since $D^{+}_{a,\lambda} \cap \{z_0=0\} \neq \emptyset$,
it follows that the 
negative degree terms of $z_0$ do not appear in the Laurent expansions.
Therefore
\begin{eqnarray*}
f_0(z_0, \ldots, z_n) 
= a^{(0)}_{(1,0,\ldots,0)}z_0,
\end{eqnarray*}
and
\begin{eqnarray*}
f_i(z_0, \ldots, z_n) 
= a^{(i)}_{(\lambda,0,\ldots,0)}z_0^{\lambda} 
+ {\sum_{|\nu'|=1}}^{\hspace{-4pt}\prime}\,\,\,\,
a^{(i)}_{\nu'}(z')^{\nu'},
\end{eqnarray*}
for $1\leq i \leq n$.
Consider 
\begin{eqnarray*}
Pf (z) = \left(a^{(0)}_{(1,0,\ldots,0)}z_0,
{\sum_{|\nu'|=1}}^{\hspace{-4pt}\prime}\,\,\,\, a^{(1)}_{\nu'}(z')^{\nu'},
\ldots,
{\sum_{|\nu'|=1}}^{\hspace{-4pt}\prime}\,\,\,\, a^{(n)}_{\nu'}(z')^{\nu'}\right).
\end{eqnarray*}
Then 
$Pf$ gives a representation of $GU(n,1)$ by
\begin{eqnarray*}
P\rho : GU(n,1) \ni g \longmapsto P(\rho(g)) \in GL(n+1,\mathbb{C}),
\end{eqnarray*}
as in the proof of Lemma \ref{cla1},
and
we showed that this can not occur by Lemma~\ref{notsimple}.
Thus $\lambda=1$ and $\Omega$ is biholomorphic to $D^{n,1}$.

\end{proof}

\begin{lemma}\label{eqthm1}
If $\Omega=C^{+}_{a,\lambda}$,  
then $\Omega$ is biholomorphic to $C^{n,1}$.
\end{lemma}

\begin{proof}
Indeed, there exists a biholomorphic map from $\Omega$ to $C^{n,1}$ by
\begin{eqnarray*}
\Phi' : \C^{n+1} \ni (z_0, z_1, \ldots, z_n) \longmapsto 
(a^{1/2}z_0, z_0^{1-\lambda}z_1, \ldots, z_0^{1-\lambda}z_n) \in \C^{n+1}.
\end{eqnarray*}

\end{proof}

\begin{lemma}\label{CB}
$\Omega \neq D^{-}_{a,\lambda}$.
\end{lemma}

\begin{proof}
Assume $\Omega = D^{-}_{a,\lambda}$.
Without loss of generality, we may assume $a=1$.
Then $D^{-}_{a,\lambda}$ is biholomorphic to
$\mathbb{C}^* \times (\C^n \setminus \overline{\B^n})$ by a biholomorphism
\begin{equation}\label{A-1}
\varpi_{\lambda}^{-1}:(z_0, z_1, \ldots, z_n) \mapsto (z_0, z_0^{-\lambda}z_1, \ldots, z_0^{-\lambda}z_n).
\end{equation}
The $GU(n,1)$-action on $\C^* \times (\C^n \setminus \overline{\B^{n}})$ 
induced by this biholomorphism are given by 
$\varpi_{\lambda}^{-1} \circ f \circ \varpi_{\lambda}$ for $f \in \rho(GU(n,1))$,
and $U(1) \times U(n)$-action on 
$\mathbb{C}^* \times (\mathbb{C}^n \setminus \overline{\mathbb{B}^n})$ 
is still linear.
Then, by $n>1$,
the $GU(n,1)$-action extends holomorphically on $\C^* \times \C^n$.
However we have shown in the proof of Lemma~\ref{3.6} that
$GU(n,1)$ does not act effectively on $\C^* \times \C^n$ through $\varpi_{\lambda}^{-1} \rho \varpi_{\lambda}$.
Thus this is a contradiction, and the lemma is proven.

\end{proof}

\begin{lemma} \label{cla4}
$\Omega \neq C^{-}_{a,\lambda}$.
\end{lemma}

\begin{proof}
Suppose $\Omega = C^{-}_{a,\lambda}$.
Then, for $f \in \rho(GU(n,1))$,
the Laurent series of the components are
\begin{eqnarray*}
f_0(z_0, \ldots, z_n) 
=  \sum_{k=0}^{\infty} {\sum_{|\nu'|=k}}^{\hspace{-4pt}\prime}\,\,  
a^{(0)}_{\nu'}z_0^{1-k\lambda}(z')^{\nu'},
\end{eqnarray*}
and
\begin{eqnarray*}
f_i(z_0, \ldots, z_n) = 
a^{(i)}_{(\lambda,0,\ldots,0)}z_0^{\lambda} + 
 \sum_{k=0}^{\infty} {\sum_{|\nu'|=1+k}}^{\hspace{-9pt}\prime}\,\,\,\,
a^{(i)}_{\nu'}z_0^{-k\lambda}(z')^{\nu'}
\end{eqnarray*}
for $1\leq i \leq n$.
Since $C^{-}_{a,\lambda} \cap \{z_0=0\} \neq \emptyset$,
$a^{(i)}_{(\lambda,0,\ldots,0)}$ must be vanish.
However this contradicts Lemma \ref{5.0}.

\end{proof}

\medskip
Let us consider the case $\partial \Omega \neq A_{a,\lambda}$.
We prove that, in this case, $GU(n,1)$ does not act
effectively on $\Omega$, except Case (I-iii) below.

\smallskip

{\bf Case (I)}:
$(\partial \Omega \setminus A_{a,\lambda}) \cap (\C^* \times (\C^n \setminus \{0\}))
 = \emptyset$.
\\
In this case,
$\partial \Omega$ is the union of $A_{a,\lambda}$ and some of the following sets
\begin{equation}\label{sets}
\{0\} \times (\C^n \setminus \{0\}),
\C^* \times \{0\}
\,\,\mathrm{or} \,\,
\{0\} \subset \C^{n+1},
\end{equation}
by the $G$-actions on the boundary of the type 
(\ref{GU2}), (\ref{GU3}) and (\ref{GU4}).
If $\Omega \subset D^{-}_{a,\lambda}$,
then the sets in (\ref{sets})
can not be contained in the boundary of $\Omega$.
Thus we consider only the cases $\Omega \subsetneq D^{+}_{a,\lambda}$,
$C^{+}_{a,\lambda}$ and $C^{-}_{a,\lambda}$.

\smallskip
{\bf Case (I-i)}:
$\Omega \subsetneq D^{+}_{a,\lambda}$.\\
In this case, $\mathbb{C}^* \times \{0\}$ can not be a subset of the boundary of $\Omega$,
and $\{0\} \in A_{a,\lambda}$.
Thus 
\begin{eqnarray*}
&&\partial \Omega = A_{a,\lambda} \cup (\{0\} \times \C^n),\\
&&\Omega = D^{+}_{a,\lambda} \setminus (\{0\} \times \C^n).
\end{eqnarray*}
Then $\Omega$ is biholomorphic to $\C^* \times (\C^n \setminus \overline{\B^{n}})$ by a biholomorphism of the form in (\ref{A-1}).
Thus, as the proof of Lemma~\ref{CB},
this case does not occur.

\smallskip
{\bf Case (I-ii)}:
$\Omega \subsetneq C^{+}_{a,\lambda}$.\\
In this case, $\{0\} \times (\mathbb{C}^n \setminus \{0\})$ 
can not be a subset of the boundary of $\Omega$,
and $\{0\} \in A_{a,\lambda}$.
Thus 
\begin{eqnarray*}
&&\partial \Omega = A_{a,\lambda} \cup (\mathbb{C} \times \{0\}),\\
&&\Omega = C^{+}_{a,\lambda} \setminus (\mathbb{C} \times \{0\}).
\end{eqnarray*}
Then $a^{(i)}_{(\lambda,0,\ldots,0)}$ in (\ref{aci'})
must vanish since $f \in \rho(GU(n,1))$ preserves $\C^* \times \{0\}$,
and this contradicts Lemma~\ref{5.0}.
Thus this case does not occur.

\smallskip

{\bf Case (I-iii)}:
$\Omega \subsetneq C^{-}_{a,\lambda}$.\\
In this case, $\Omega$ coincides with one of the followings:
\begin{eqnarray*}
&&C_1=C^{-}_{a,\lambda} \setminus (\{0\} \times \mathbb{C}^n) \cup(\mathbb{C} \times \{0\}),
\\
&&C_2=C^{-}_{a,\lambda} \setminus (\{0\} \times \mathbb{C}^n),\\
&&C_3=C^{-}_{a,\lambda} \setminus (\mathbb{C} \times \{0\}),\\
&&C_4=C^{-}_{a,\lambda} \setminus \{0\in \C^{n+1}\}.
\end{eqnarray*}
Then
$\Omega \neq C_1$, since $f \in \rho(GU(n,1))$
preserves $\C^* \times \{0\}$ only if all
$a^{(i)}_{(\lambda,0,\ldots,0)}$ in (\ref{aci'})
vanish, and this does not occur by Lemma \ref{5.0}.
$C_2$ is biholomorphic to $C^{n,1}$ by a biholomorphism
\begin{equation*}
(z_0, z_1, \ldots, z_n) \mapsto 
(z_0, a^{-1/2}z_0^{1-\lambda}z_1, \ldots, a^{-1/2}z_0^{1-\lambda}z_n),
\end{equation*}
and therefore $GU(n,1)$ acts effectively.
If $\Omega = C_3$ or $C_4$, 
then $\Omega \cap \{z_0= 0\} \neq \emptyset$.
Hence in the Laurent expansions (\ref{aci'}) of $f \in \rho(GU(n,1))$,
$a^{(i)}_{(\lambda,0,\ldots,0)}$, for $1 \leq i \leq n$, must vanish,
and this contradicts Lemma \ref{5.0}.
Thus these cases do not occur either.

\bigskip

{\bf Case (I\hspace{-.1em}I)}: 
$(\partial \Omega \setminus A_{a,\lambda}) \cap (\mathbb{C}^* \times (\mathbb{C}^n \setminus \{0\})) \neq \emptyset$.
\\
In this case, we can take a point 
$p'=(p'_0, \ldots, p'_n) \in (\partial \Omega \setminus A_{a,\lambda}) \cap (\mathbb{C}^* \times (\mathbb{C}^n \setminus \{0\}))$.
We put 
\begin{eqnarray*}
&&b := (|p'_1|^2 + \cdots + |p'_n|^2)/|p'_0|^{2\lambda} > 0,\\
&&B_{b,\lambda} := \{(z_0, \ldots, z_n) \in \mathbb{C}^{n+1} : -b|z_0|^{2\lambda} + |z_1|^2 + \cdots + |z_n|^2 = 0\}.
\end{eqnarray*}
We may assume $a>b$ without loss of generality.

\smallskip

{\bf Case (I\hspace{-.1em}I-i)}:
$\partial \Omega = A_{a,\lambda} \cup B_{b,\lambda}$. \\
Since $\Omega$ is connected, it coincides with
\begin{eqnarray*} 
C^{+}_{a,\lambda} \cap D^{+}_{b,\lambda} 
= \{ b|z_0|^{2\lambda} < |z_1|^2 + \cdots + |z_n|^2 < a|z_0|^{2\lambda} \},
\end{eqnarray*}
or
\begin{eqnarray*} 
C^{-}_{a,\lambda} \cap D^{-}_{b,\lambda} 
= \{ b < (|z_1|^2 + \cdots + |z_n|^2)|z_0|^{-2\lambda} < a \}.
\end{eqnarray*}
By the biholomorphic map in (\ref{A-1}),
these domains are biholomorphic to $\mathbb{C}^* \times \mathbb{B}^n(a,b)$, 
where 
\begin{equation*} 
\mathbb{B}^n(a,b) = \{(z_1, \ldots, z_n) \in \mathbb{C}^n : b<|z_{1}|^{2} + \cdots + |z_{n}|^{2} < a\}.
\end{equation*}
However this implies that $a^{(i)}_{(\lambda,0,\ldots,0)} = 0$ in (\ref{aci'}) (see also (\ref{Lhi}))
since the 
$GU(n,1)$-action extends to the domain $\mathbb{C}^* \times \mathbb{B}^n_a$ 
and preserves $\mathbb{C}^* \times \mathbb{B}^n_b$.
Therefore this contradicts Lemma \ref{5.0}.
Thus this case does not occur.

\smallskip

{\bf Case (I\hspace{-.1em}I-ii)}: 
$\partial \Omega \supsetneq  A_{a,\lambda} \cup B_{b,\lambda}$.
\\
Suppose $(\partial \Omega \setminus (A_{a,\lambda} \cup B_{b,\lambda}) )\cap (\mathbb{C}^* \times \mathbb{C}^n \setminus \{0\})
\neq \emptyset$.
Then we can take 
\begin{equation*} 
p''=(p''_0, \ldots, p''_n) \in (\partial \Omega \setminus (A_{a,\lambda} \cup B_{b,\lambda})) \cap (\mathbb{C}^* \times \mathbb{C}^n 
\setminus \{0\}).
\end{equation*}
We put
\begin{eqnarray*} 
&&c=(|p''_1|^2 + \cdots + |p''_n|^2)/|p''_0|^{2\lambda},\\
&&C_{c,\lambda} = \{(z_0, \ldots, z_n) \in \mathbb{C}^{n+1} : 
-c|z_0|^{2\lambda} + |z_1|^2 + \cdots + |z_n|^2 = 0\}.
\end{eqnarray*}
Then we have $A_{a,\lambda} \cup B_{b,\lambda} \cup C_{c,\lambda} \subset \partial \Omega$.
However $\Omega$ is connected, this is impossible. 
Therefore this case does not occur.
Let us consider the remaining case:
\begin{eqnarray*} 
(\partial \Omega \setminus (A_{a,\lambda} \cup B_{b,\lambda}) )\cap (\mathbb{C}^* \times \mathbb{C}^n \setminus \{0\})
= \emptyset.
\end{eqnarray*}
However,
$\mathbb{C}^* \times \{0\}$, $\{0\} \times (\mathbb{C}^n \setminus \{0\})$
and $\{0\} \in \mathbb{C}^{n+1}$
can not be subsets of the boundary of $\Omega$
since $\Omega \subset C^{+}_{a,\lambda} \cap D^{+}_{b,\lambda}$
or
$\Omega \subset C^{-}_{a,\lambda} \cap D^{-}_{b,\lambda}$.
Thus this case does not occur either.

\smallskip
We have shown that if $c_1c_2 \neq 0$,
then $\Omega$ is biholomorphic to one of the following domains:
\begin{eqnarray*} 
\C^{n+1}, \C^{n+1} \setminus \{0\}, D^{n,1}\,\, \mathrm{or}\,\, C^{n,1} \simeq \C^* \times \B^n.
\end{eqnarray*}
This completes the proof.

\end{proof}

\section{The actions of $GU(1,1)$}
\label{sec:4}

Now we prove the following theorem.

\begin{theorem}  \label{M2}
Let $M$ be a connected complex manifold of dimension $2$ 
that is holomorphically separable
and admits a smooth envelope of holomorphy.
Assume that there exists an injective homomorphism of topological groups
$\rho_0 : GU(1,1) \longrightarrow \mathrm{Aut}(M)$.
Then $M$ is biholomorphic to one of the four domains $\C^{2}$, $\C^{2} \setminus \{0\}$, $D^{1,1}\simeq\C^*\times \D$ or
$\C \times \D$.
\end{theorem}

\begin{proof}
By Lemma~\ref{4} and the comments after that, 
we can assume $M$ is a Reinhardt domain $\Omega$ in $\C^{2}$,
and the $U(1) \times U(1)$-action on $\Omega \subset \C^{2}$ is linear.
We will 
prove that $\Omega$ is biholomorphic to  one of  the four domains
$\C^{2}$, $\C^{2} \setminus \{0\}$,
$D^{1,1} \simeq \C^* \times \D$ or $\C \times \D$.

Put a coordinate $(z_1, z_2)$ of $\C^2$.
Since $\rho(\mathbb{C}^*)$ is 
commutative with $\rho(U(1) \times U(1)) = U(1) \times U(1)$,
Lemma~\ref{KS1} tells us that $\rho(\mathbb{C}^*) \subset \Pi(\Omega)$,
so that we have 
\begin{eqnarray*}
\rho\left(e^{2\pi i(s+it)}\right) =
\mathrm{diag}\left[
e^{2\pi i\{a_1s+(b_1+ic_1)t\}}, e^{2\pi i\{a_2s+(b_2+ic_2)t\}}
\right]
\in \rho(\mathbb{C}^*),
\end{eqnarray*}
where
$s,t \in \mathbb{R}$, $a_1, a_2 \in \mathbb{Z}, b_1,b_2,c_1,c_2 \in \mathbb{R}$.
Since $\rho$ is injective,
$a_1, a_2$ are
relatively prime
and $(c_1, c_2) \neq (0, 0)$.
For $(e^{2\pi is_1}, e^{2\pi is_2}) \in U(1)\times U(1)$, $s_1, s_2 \in \mathbb{R}$,
we have
\begin{eqnarray}\label{T11}
\rho\left(
\mathrm{diag}\left[
e^{2\pi is_1}, e^{2\pi is_2}
\right]\right)
=
\mathrm{diag}\left[
e^{2\pi i(as_1+bs_2)}, e^{2\pi i(cs_1+ds_2)}
\right]
\in \rho\left(U(1)\times U(1)\right),
\end{eqnarray}
where
\begin{eqnarray*}
\begin{pmatrix}
a & b\\
c & d
\end{pmatrix}
\in GL(2, \mathbb{Z}),
\end{eqnarray*}
since $\rho$ is injective,
and we have $a+b=a_1$ and $c+d=a_2$.
To consider the actions of $\mathbb{C}^*$ and $U(1) \times U(1)$ 
on $\Omega$ together,
we put
\begin{equation*} 
G(U(1) \times U(1)) = 
\left\{e^{-2\pi t}
\mathrm{diag}\left[
u_1, u_2
\right]
\in GU(1,1) : t \in \mathbb{R}, u_1, u_2 \in U(1)
\right\}.
\end{equation*}
Then we have
\begin{eqnarray*} 
G &:=&\rho(G(U(1) \times U(n))) \\
&{}=&
\left\{
\mathrm{diag}\left[
e^{-2\pi c_1t}u_1, e^{-2\pi c_2t}u_2
\right]
\in GL(2,\mathbb{C}) : t \in \mathbb{R}, u_1, u_2 \in U(1)
\right\}.
\end{eqnarray*}
Note that $G$ is the centralizer
of $U(1) \times U(1)$ in $\rho(GU(n, 1))$.

Let $f =(f_1, f_2) \in \rho(GU(1, 1))$ and consider 
the Laurent series of its components:\\
\begin{eqnarray}\label{eqf1}
f_i(z_1, z_2) 
= \sum_{\nu \in \mathbb{Z}^{2}} a^{(i)}_{\nu}z^{\nu},
\end{eqnarray}
for $i=1,2$.
As Lemma \ref{cl1} in the proof of Theorem~\ref{M1}, we have:
\begin{lemma}\label{cl1'}
For any $f \in \rho(GU(1, 1)) \setminus G$,
there exists $\nu \in \mathbb{Z}^{2}$, $\neq (1,0)$,
such that $a^{(1)}_{\nu} \neq 0$ in $(\ref{eqf1})$,
or
there exists
$\nu \in \mathbb{Z}^{2}$, $\neq (0,1)$,
such that
$a^{(2)}_{\nu} \neq 0$ in $(\ref{eqf1})$.
\end{lemma}

By {\it Remark} \ref{rem0} after Lemma \ref{c_20} in the proof of Theorem~\ref{M1},
we have
\begin{lemma}\label{c_200}
If $c_1c_2=0$, then
$\Omega$ is biholomorphic to $\C^* \times \D$ or $\C \times \D$.
\end{lemma}

We consider the case $c_1c_2 \neq 0$.
Since $\mathbb{C}^*$ is the center of $GU(1,1)$, 
it follows that, 
for $f \in \rho(GU(1,1))$, 
\begin{eqnarray*}
\rho\left(e^{2\pi i(s+it)}\right) \circ f = f \circ \rho\left(e^{2\pi i(s+it)}\right).
\end{eqnarray*}
By $(\ref{eqf1})$,
this equation means
\begin{eqnarray*}
e^{2\pi i\{a_1s+(b_1+ic_1)t\}} \sum_{\nu} a^{(1)}_{\nu}z^{\nu}
&=& 
\sum_{\nu} a^{(1)}_{\nu}
\left(e^{2\pi i\{a_1s+(b_1+ic_1)t\}}z_1\right)^{\nu_1^{(1)}}
\left(e^{2\pi i\{a_2s+(b_2+ic_2)t\}}z_2\right)^{\nu_2^{(1)}}\\
&=&
\sum_{\nu} a^{(1)}_{\nu}
e^{2\pi i\{a_1s+(b_1+ic_1)t\}\nu_1^{(1)}} e^{2\pi i\{a_2s+(b_2+ic_2)t\}\nu_2^{(1)}}
z^{\nu},
\end{eqnarray*}
and
\begin{eqnarray*}
e^{2\pi i\{a_2s+(b_2+ic_2)t\}} \sum_{\nu} a^{(2)}_{\nu}z^{\nu}
&=&
\sum_{\nu} a^{(2)}_{\nu}
\left(e^{2\pi i\{a_1s+(b_1+ic_1)t\}}z_1\right)^{\nu_1^{(2)}}
\left(e^{2\pi i\{a_2s+(b_2+ic_2)t\}}z_2\right)^{\nu_2^{(2)}}\\
 &=&
\sum_{\nu} a^{(1)}_{\nu}
e^{2\pi i\{a_1s+(b_1+ic_1)t\}\nu_1^{(2)}} e^{2\pi i\{a_2s+(b_2+ic_2)t\}\nu_2^{(2)}} z^{\nu}.
\end{eqnarray*}
Thus for each $\nu \in \mathbb{Z}^{2}$, we have
\begin{eqnarray*}
e^{2\pi i\{a_1s+(b_1+ic_1)t\}} a^{(1)}_{\nu}
= e^{2\pi i\{a_1s+(b_1+ic_1)t\}\nu_1^{(1)}} e^{2\pi i\{a_2s+(b_2+ic_2)t\}\nu_2^{(1)}} a^{(1)}_{\nu},
\end{eqnarray*}
and
\begin{eqnarray*}
e^{2\pi i\{a_2s+(b_2+ic_2)t\}} a^{(2)}_{\nu}
= e^{2\pi i\{a_1s+(b_1+ic_1)t\}\nu_1^{(2)}} e^{2\pi i\{a_2s+(b_2+ic_2)t\}\nu_2^{(2)}} a^{(2)}_{\nu}
\end{eqnarray*}
Therefore, if
$a_{\nu}^{(1)}\neq0$ for $\nu=(\nu_1^{(1)}, \nu_2^{(1)})$,
we have
\begin{eqnarray}\label{f1}
\hspace{-3cm}
\left\{\begin{array}{l}
a_1(\nu_1^{(1)}-1) + a_2\nu_2^{(1)} = 0,\\
c_1(\nu_1^{(1)}-1) + c_2\nu_2^{(1)} = 0,
\end{array}
\right.
\end{eqnarray}
and if 
$a_{\nu}^{(2)} \neq 0$ for $\nu=(\nu_1^{(2)}, \nu_2^{(2)})$,
we have
\begin{eqnarray}\label{f2}
\hspace{-33mm}
\left\{\begin{array}{l}
a_1\nu_1^{(2)} + a_2(\nu_2^{(2)} - 1) = 0,\\
c_1\nu_1^{(2)} + c_2(\nu_2^{(2)} - 1) = 0.
\end{array}
\right.
\end{eqnarray}

\begin{lemma}\label{n1cl}
If $c_1c_2 \neq 0$, then $\lambda := c_2/c_1 = a_2/a_1 \in \Q$.
\end{lemma}

\begin{proof}
The proof appeared at the beginning of that of Lemma \ref{cla1}.
The lemma follows from Lemma \ref{cl1'} (\ref{f1}) and (\ref{f2}).
We omit the proof.
\end{proof}

Since $a_1$ and $a_2$ are relatively prime, 
we have
$(\nu_1^{(1)},\nu_2^{(1)}) = (1+ka_2,-ka_1)$ by (\ref{f1}),
and
$(\nu_1^{(2)},\nu_2^{(2)}) = (la_2,1-la_1)$ by (\ref{f2}),
where $k, l \in \Z$.
Then 
the Laurent series of the components of $f \in \rho(GU(1,1))$ are
\begin{eqnarray}\label{form1}
f_1(z_1, z_2) 
= \sum_{k \in \Z}
a^{(1)}_{\nu}z_1^{1+ka_2} z_2^{-ka_1},
\end{eqnarray}
and
\begin{eqnarray}\label{form2}
f_2(z_1, z_2) 
= \sum_{k \in \Z}
a^{(2)}_{\nu}z_1^{ka_2} z_2^{1-ka_1}.
\end{eqnarray}

Since $G=\rho(G(U(1) \times U(1)))$ acts as linear transformations on 
$\Omega \subset \mathbb{C}^{2}$, 
it preserves the boundary $\partial \Omega$ of $\Omega$.
We now study the action of $G$ on $\partial \Omega$.
The $G$-orbits of points in $\mathbb{C}^{2}$ 
consist of four types as follows:

(i) If 
$p = (p_1, p_2) \in \mathbb{C}^* \times \mathbb{C}^*$,
then
\begin{equation} \label{GU11}
G \cdot p = \{ (z_1, z_2) \in \mathbb{C}^* \times \mathbb{C}^*
 : -a|z_1|^{2|a_1|\lambda} + |z_2|^{2|a_1|} = 0\},
\end{equation}
where
$a := (|p_2|^2/|p_1|^{2\lambda})^{|a_1|} > 0$
and $\lambda = c_2/c_1 = a_2/a_1 \in \Q$ by Lemma \ref{n1cl}.

\smallskip

(ii) If $p' = (0, p'_2) \in \mathbb{C}^{2} \setminus \{0\}$,
then
\begin{equation} \label{GU22}
G \cdot p' = \{0\} \times \mathbb{C}^*.
\end{equation}

\smallskip

(iii) If $p'' = (p''_1, 0) \in \mathbb{C}^{2} \setminus \{0\}$,
then
\begin{equation} \label{GU33}
G \cdot p'' = \mathbb{C}^* \times \{0\}.
\end{equation}

\smallskip

(iv) If $p''' = (0, 0) \in \mathbb{C}^{2}$,
then
\begin{equation} \label{GU44}
G \cdot p''' = \{0\} \subset \mathbb{C}^{2}.
\end{equation}

\smallskip

\begin{lemma}\label{4.3}
If 
\begin{equation*} 
\Omega \cap (\mathbb{C}^* \times \mathbb{C}^*) 
= \mathbb{C}^* \times \mathbb{C}^*,
\end{equation*}
then $\Omega$ is equal to $\C^{2}$ or $\C^{2}  \setminus \{0\}$.
\end{lemma}

Lemma~\ref{4.3} is similar to Lemma~\ref{3.6}, but the proof is a bit different
from that of Lemma~\ref{3.6} and longer than it.

\begin{proof}
If $\Omega \cap (\C^* \times \C^*)
= \C^* \times \C^*$,
then 
$\Omega$ equals one of the following domains 
by the $G$-actions of the type (\ref{GU22}), (\ref{GU33}) and (\ref{GU44}) above:
\begin{equation*} 
\C^{2},
\C^{2}  \setminus \{0\},
\C \times \C^*,
\C^* \times \C
\,\,\mathrm{or} \,\,\C^* \times \C^*.
\end{equation*}
Clearly $GU(1,1)$ acts on $\C^{2}$ and $\C^{2} \setminus \{0\}$
by matrix multiplications.
For the latter three domains, $GU(1,1)$ does not act effectively.

{\bf Case: $\C^* \times \C$}\\
Suppose $\Omega = \C^* \times \C$.
Then, in (\ref{form1}) and (\ref{form2}),
there is no negative degree term of $z_2$ since
$\Omega \cap \{z_2=0\} \neq \emptyset$.

First we assume $|a_1| \neq 1$.
If $\lambda < 0$, then the Laurent series of the components of
$f \in \rho(GU(1,1))$ are
\begin{eqnarray*}\label{}
f_1(z_1, z_2) 
= \sum_{k \in \Z_{\geq 0}}
a^{(1)}_{\nu}z_1^{1+k|a_2|} z_2^{k|a_1|},
\end{eqnarray*}
and
\begin{eqnarray*}\label{}
f_2(z_1, z_2) 
= \sum_{k \in \Z_{\geq 0}}
a^{(2)}_{\nu}z_1^{k|a_2|} z_2^{1+k|a_1|}.
\end{eqnarray*}
We put
\begin{eqnarray*}\label{}
Pf (z) := \left(a^{(1)}_{(1,0)}z_1, a^{(2)}_{(0,1)}z_2\right).
\end{eqnarray*}
As a matrix we can write
\begin{eqnarray*}
Pf = 
\begin{pmatrix}
a^{(1)}_{(1,0)} & 0\\
0 & a^{(2)}_{(0,1)}
\end{pmatrix}.
\end{eqnarray*}
Then it follows that
\begin{eqnarray*}\label{}
P(f \circ h) = Pf \circ Ph, \,\,\,{\rm and}\,\,\, P{\rm id} = {\rm id},
\end{eqnarray*}
where $h \in \rho(GU(1,1))$, and therefore 
\begin{eqnarray*}\label{}
Pf  \in GL(2, \mathbb{C})
\end{eqnarray*}
since $f$ is an automorphism.
Hence we have a representation of $GU(1,1)$ given by
\begin{eqnarray*}
GU(1,1) \ni g \longmapsto Pf \in GL(2,\mathbb{C}),
\end{eqnarray*}
where $f = \rho(g)$.
The restriction of this representation to the simple Lie group $SU(1,1)$
is nontrivial since $\rho(U(1) \times U(1)) = U(1) \times U(1)$.
However this contradicts Lemma~\ref{notsimple}.
Thus the case $\lambda < 0$ does not occur.

If $\lambda > 0$, then the Laurent series of the components of
$f \in \rho(GU(1,1))$ are
\begin{eqnarray}\label{tay1}
f_1(z_1, z_2) 
= \sum_{k \in \Z_{\geq 0}}
a^{(1)}_{\nu}z_1^{1-k|a_2|} z_2^{k|a_1|}
\end{eqnarray}
and
\begin{eqnarray}\label{tay2}
f_2(z_1, z_2) 
= \sum_{k \in \Z_{\geq 0}}
a^{(2)}_{\nu}z_1^{-k|a_2|} z_2^{1+k|a_1|}.
\end{eqnarray}
We claim that $a^{(1)}_{(1,0)} \neq 0$.
Indeed, if $a^{(1)}_{(1,0)} = 0$,
then $f(z_0, 0) = (0, 0)$.
This is a contradiction
since $f$ is an automorphism.
Take another $h \in \rho(GU(1,1))$
and put the Laurent series of its components:
\begin{eqnarray*}
h_1(z_1, z_2) 
&=& \sum_{k \in \Z_{\geq 0}}
b^{(1)}_{\nu}z_1^{1-k|a_2|} z_2^{k|a_1|},
\\
h_2(z_1, z_2) 
&=& \sum_{k \in \Z_{\geq 0}}
b^{(2)}_{\nu}z_1^{-k|a_2|} z_2^{1+k|a_1|}.
\end{eqnarray*}
We have $b^{(1)}_{(1,0)} \neq 0$ as above.
We mention the first degree terms of $f \circ h$.
For the first component
\begin{eqnarray*}
f_1(h_1, h_2) 
&=& a^{(1)}_{(1,0)}h_1 + 
\sum_{k \in \Z_{\geq 1}}
a^{(1)}_{\nu}h_1^{1-k|a_2|} h_2^{k|a_1|}.
\end{eqnarray*}
Then, for $k \geq 1$,
\begin{eqnarray*}
h_1(z)^{1-k|a_2|} &=& \left(\sum_{l = 0}^{\infty}
b^{(1)}_{\nu} z_1^{1-l|a_2|} z_2^{l|a_1|}\right)^{1-k|a_2|}
=
z_1^{1-k|a_2|} \left(\sum_{l = 0}^{\infty}
b^{(1)}_{\nu} z_1^{-l|a_2|} z_2^{l|a_1|}\right)^{1-k|a_2|}\\
&=&
\left(b^{(1)}_{(1,0)}z_1\right)^{1-k|a_2|} 
\left(1+\frac{1-k|a_2|}{b^{(1)}_{(1,0)}} b^{(1)}_{(1-|a_2|, |a_1|)} z_1^{-|a_2|} z_2^{|a_1|}
+ \cdots \right)
\end{eqnarray*}
Thus
$h_1(z)^{1-k|a_2|}$ has the maximum degree of $z_1$ at most $1 - k|a_2| < 1$
and has the minimum degree of $z_2$ at least $|a_1| > 1$
in its Laurent expansion.
For $k>0$,
$h_2^{k|a_1|}$ has the maximum degree of $z_1$ at most $-|a_2| < 0$
and
has the minimum degree of $z_2$ at least $k|a_1| > 1$
in its Laurent expansion.
Hence the first degree term of Laurent expansion of $f_1(h_1, h_2)$ 
is $a^{(1)}_{(1,0)} b^{(1)}_{(1,0)} z_1$.
Similarly, consider
\begin{equation*}
f_2(h_1, h_2) 
= a^{(2)}_{(0,1)} h_2 + 
\sum_{k=1}^{\infty} a^{(2)}_{\nu} h_1^{-k|a_2|}h_2^{1+k|a_1|}.
\end{equation*}
Then, for $k>0$,
\begin{eqnarray*}
h_1^{-k|a_2|}
&=& \left(b^{(1)}_{(1,0)}z_1\right)^{-k|a_2|} 
\left(1-\frac{k|a_2|}{b^{(1)}_{(1,0)}} b^{(1)}_{(1-|a_2|, |a_1|)} z_1^{-|a_2|} z_2^{|a_1|}
+ \cdots \right),
\end{eqnarray*}
therefore, $h_1^{-k|a_2|}$ has the maximum degree of $z_1$ at most $- k|a_2| < 0$
 and has the minimum degree of $z_2$ at least $|a_1| > 1$
 in its Laurent expansion.
For $k>0$, $h_2^{1+k|a_1|}$
has the maximum degree of $z_1$ at most $-|a_2| < 0$
and has the minimum degree of $z_2$ at least $1+k|a_1| > 1$
in its Laurent expansion.
Hence the first degree term of the Laurent expansion of $f_2(h_1, h_2)$ 
is 
$a^{(2)}_{(0,1)} b^{(2)}_{(0,1)} z_2$.
Consequently, the first degree terms of the Laurent series of the
components of the composite
$f \circ h$ are the composites of the first degree terms of Laurent expansions of
$f$ and $h$.
We put
\begin{eqnarray*}\label{}
Pf (z) := \left(a^{(1)}_{(1,0)}z_1, a^{(2)}_{(0,1)}z_2\right).
\end{eqnarray*}
As a matrix we can write
\begin{eqnarray*}
Pf = 
\begin{pmatrix}
a^{(1)}_{(1,0)} & 0\\
0 & a^{(2)}_{(0,1)}
\end{pmatrix}.
\end{eqnarray*}
Then it follows from above computation that
\begin{eqnarray*}\label{}
P(f \circ h) = Pf \circ Ph, \,\,\,{\rm and}\,\,\, P{\rm id} = {\rm id},
\end{eqnarray*}
and therefore 
\begin{eqnarray*}\label{}
Pf  \in GL(2, \mathbb{C})
\end{eqnarray*}
since $f$ is an automorphism.
Hence we have a representation of $GU(1,1)$ given by
\begin{eqnarray*}
GU(1,1) \ni g \longmapsto Pf \in GL(2,\mathbb{C}),
\end{eqnarray*}
where $f = \rho(g)$.
The restriction of this representation to the simple Lie group $SU(1,1)$
is nontrivial since $\rho(U(1) \times U(1)) = U(1) \times U(1)$.
However this contradicts Lemma~\ref{notsimple}.
Thus the case $\lambda < 0$ does not occur either.

Next we assume $|a_1| = 1$, that is, $\lambda \in \Z$.
We put an automorphism $\varpi_{\lambda} \in \aut(\Omega)$ by 
$\varpi_{\lambda}(z) = (z_1,z_1^{\lambda}z_2)$.
Then $\varpi_{\lambda}^{-1} \rho(GU(1,1)) \varpi_{\lambda}$ is
a subgroup of $\aut(\Omega)$.
Then, 
\begin{eqnarray*}
\varpi_{\lambda}^{-1}\circ  \rho(e^{-2\pi t})\circ  \varpi_{\lambda} (z_1, z_2)
&=&\varpi_{\lambda}^{-1}(e^{2\pi i(b_1+ic_1)t}z_1, e^{2\pi i(b_2+ic_2)t}z_1^{\lambda}z_2)\\
&=&(e^{2\pi i(b_1+ic_1)t}z_1, e^{2\pi i\{-\lambda(b_1+ic_1)+(b_2+ic_2)\}t}z_2)\\
&=&(e^{2\pi i(b_1+ic_1)t}z_1, e^{2\pi i\lambda(-b_1+b_2)t}z_2).
\end{eqnarray*}
Here we used the equations
$-\lambda c_1+c_2=0$
by Lemma~\ref{n1cl}.
Hence this case turns out to the case $c_2=0$ of the $\C^*$-action.
However, by Lemma~\ref{c_200}, this is a contradiction.
Thus this case does not occur.

Consequently, the case $\Omega = \C^* \times \C$
does not occur.
Either, the case $\Omega = \C \times \C^*$ does not occur.

\medskip
{\bf Case: $\C^* \times \C^*$}\\
Suppose $\Omega = \C^* \times \C^*$.
Since $a_1$ and $a_2$ are relatively prime,
there exists 
\begin{eqnarray*}\label{}
A =
\begin{pmatrix}
a_{11} & a_{12}\\
a_{21} & a_{22}
\end{pmatrix}
\in SL(2,\Z),
\end{eqnarray*}
such that
\begin{eqnarray*}\label{}
\begin{pmatrix}
a_{11} & a_{12}\\
a_{21} & a_{22}
\end{pmatrix}
\begin{pmatrix}
a_2 \\
-a_1 
\end{pmatrix}
=
\begin{pmatrix}
1 \\
0 
\end{pmatrix}.
\end{eqnarray*}
We put an automorphism $\gamma_A \in \aut(\Omega)$ given by
$\gamma_A(z) = (z_1^{a_{11}}z_2^{a_{21}}, z_1^{a_{12}}z_2^{a_{22}})$.
The inverse of $\gamma_A$
is $\gamma_A^{-1}(z)=(z_1^{a_{22}}z_2^{-a_{21}}, z_1^{-a_{12}}z_2^{a_{11}})=\gamma_{A^{-1}}(z)$.
Then the group
$\gamma_A^{-1} \rho(GU(1,1)) \gamma_A$ is a subgroup of $\aut(\Omega)$,
and 
\begin{eqnarray*}
\gamma_A^{-1}\circ  \rho(e^{-2\pi t})\circ  \gamma_A (z_1, z_2)
&=&\gamma_A^{-1}(e^{2\pi i\{(b_1+ic_1)t\}}z_1^{a_{11}}z_2^{a_{21}}, e^{2\pi i\{(b_2+ic_2)t\}}z_1^{a_{12}}z_2^{a_{22}})\\
&=&(e^{2\pi i\{a_{22}(b_1+ic_1)-a_{21}(b_2+ic_2)\}t}z_1, e^{2\pi i\{-a_{12}(b_1+ic_1)+a_{11}(b_2+ic_2)\}t}z_2)\\
&=&(e^{2\pi i(a_{22}b_1-a_{21}b_2)t}z_1, e^{2\pi i\{(-a_{12}b_1+a_{11}b_2)+ic_1/a_1\}t}z_2).
\end{eqnarray*}
Here we used the equations
$a_{22}c_1-a_{21}c_2=c_1(a_{22}-a_{21}c_2/c_1)=c_1(a_{22}-a_{21}a_2/a_1)=c_1/a_1(a_{22}a_1-a_{21}a_2)=0$
and $-a_{12}c_1+a_{11}c_2=c_1(-a_{12}+a_{11}c_2/c_1)=c_1(-a_{12}+a_{11}a_2/a_1)=c_1/a_1$
which follows from Lemma~\ref{n1cl}.
Hence this case turns out to the case $c_1=0$ of the $\C^*$-action.
However, by Lemma~\ref{c_200}, this is a contradiction.
Thus this case does not occur.

\end{proof}

\medskip
Let us consider the case 
$\Omega \cap (\mathbb{C}^* \times \mathbb{C}^*)
\neq \mathbb{C}^* \times \mathbb{C}^*$.
Then we have
$\partial \Omega \cap (\mathbb{C}^* \times \mathbb{C}^*)
\neq \emptyset$.
Thus we can take a point 
\begin{eqnarray*} 
p=(p_1, p_2) \in \partial \Omega \cap 
(\mathbb{C}^* \times \mathbb{C}^*).
\end{eqnarray*} 
Let 
\begin{eqnarray*} 
&&a := \left(|p_2|^2/|p_1|^{2\lambda}\right)^{|a_1|} > 0,\\
&&A_{a,\lambda} := \{ (z_1, z_2) \in \mathbb{C}^{2} : 
-a|z_1|^{2|a_1|\lambda} + |z_2|^{2|a_1|} = 0\}.
\end{eqnarray*}
Note that 
\begin{equation*} 
\partial \Omega  \supset A_{a,\lambda},
\end{equation*}
by the $G$-action of type (\ref{GU11}).
If $\lambda > 0$,
then $\Omega$ is contained in
\begin{equation*} 
D^{+}_{a,a_1,a_2} = \{ |z_2|^{2|a_1|} > a|z_1|^{2|a_2|}\}
\end{equation*}
or 
\begin{equation*} 
C^{+}_{a,a_1,a_2} = \{ |z_2|^{2|a_1|} < a|z_1|^{2|a_2|}\}.
\end{equation*}
If $\lambda < 0$,
then $\Omega$ is contained in
\begin{equation*} 
D^{-}_{a,a_1,a_2} = \{ |z_2|^{2|a_1|}|z_1|^{2|a_2|} > a\}
\end{equation*}
or 
\begin{equation*} 
C^{-}_{a,a_1,a_2} = \{ |z_2|^{2|a_1|}|z_1|^{2|a_2|} < a\}.
\end{equation*}
Clearly, $D^{+}_{a,a_1,a_2}$ and
$C^{+}_{a^{-1},a_2,a_1}$ are biholomorphically equivalent,
so we analyze $D^{+}_{a,a_1,a_2}$.
Let us first consider the case $\partial \Omega = A_{a,\lambda}$,
that is, $\Omega = D^{+}_{a,a_1,a_2}$,
$D^{-}_{a,a_1,a_2}$ or $C^{-}_{a,a_1,a_2}$.

\begin{lemma}\label{4.4}
If $\Omega = D^{+}_{a,a_1,a_2}$,
then $|a_2|=1$ and $\Omega$ is biholomorphic to $D^{1,1}$.
\end{lemma}

\begin{proof}
Since
$\Omega \cap \{z_1=0\} \neq \emptyset$,
by (\ref{form1}) and (\ref{form2}), 
the Laurent series of the components of $f \in \rho(GU(1,1))$ are
\begin{eqnarray*}\label{}
f_1(z_1, z_2) 
= \sum_{k \in \Z_{\geq -1}}
a^{(1)}_{\nu}z_1^{1+k|a_2|} z_2^{-k|a_1|}
\end{eqnarray*}
and
\begin{eqnarray*}\label{}
f_2(z_1, z_2) 
= \sum_{k \in \Z_{\geq 0}}
a^{(2)}_{\nu}z_1^{k|a_2|} z_2^{1-k|a_1|}
\end{eqnarray*}
where, $a^{(1)}_{(1-|a_2|,|a_1|)}=0$ if $|a_2| \neq 1$.
However if $|a_2| \neq 1$, this contradicts Lemma~\ref{notsimple}
as the proof of Lemma \ref{4.3}, for {\bf Case: $\C^* \times \C$}, $\lambda>0$.
Therefore $\Omega = \{ |z_2|^{2|a_1|} > a|z_1|^2\}$,
and $\Omega$ is biholomorphic to $D^{1,1}$ by
\begin{eqnarray*}
\Phi : \Omega \ni (z_1, z_2) \longmapsto (a^{1/2}z_1z_2^{1-|a_1|}, z_2) \in D^{1,1}.
\end{eqnarray*}

\end{proof}

\begin{lemma}\label{4.61}
$\Omega \neq D^{-}_{a,a_1,a_2}$.
\end{lemma}

\begin{proof}
Assume $\Omega = D^{-}_{a,a_1,a_2}$.
Then
we see that, using the biholomorphism
$\gamma_A^{-1}(z)=(z_1^{a_{22}}z_2^{-a_{21}}, z_1^{-a_{12}}z_2^{a_{11}})$(see {\bf Case:} $\C^* \times \C^*$ in the proof of Lemma \ref{4.3}),
$D^{-}_{a,\lambda}$ is biholomorphic to
$(\mathbb{C} \setminus \overline{\D_a}) \times \C^*$.
Then the group
$\gamma_A^{-1} \rho(GU(1,1)) \gamma_A$ is a subgroup of $\aut(\gamma_A^{-1}(\Omega))$,
and
\begin{eqnarray*}
\gamma_A^{-1}\circ  \rho(e^{-2\pi t})\circ  \gamma_A (z_1, z_2)
=(e^{2\pi i(a_{22}b_1-a_{21}b_2)t}z_1, e^{2\pi i\{(-a_{12}b_1+a_{11}b_2)+ic_1/a_1\}t}z_2),
\end{eqnarray*}
by Lemma~\ref{n1cl}.
Hence this case turns out to the case $c_1=0$ of the $\C^*$-action.
However, by Lemma~\ref{c_200}, this is a contradiction.
Thus this case does not occur.

\end{proof}

\begin{lemma} \label{4.6}
$\Omega \neq C^{-}_{a,a_1,a_2}$.
\end{lemma}

\begin{proof}
Suppose $\Omega = C^{-}_{a,a_1,a_2}$.
By (\ref{form1}) and (\ref{form2}),
the Laurent series of the components of $f \in \rho(GU(1,1))$ are
\begin{eqnarray*}
f_1(z_1,z_2) 
=  \sum_{k=0}^{\infty} 
a^{(1)}_{\nu}z_1^{1+k|a_2|}z_2^{k|a_1|},
\end{eqnarray*}
and
\begin{eqnarray*}
f_2(z_1,z_2)
= \sum_{k=0}^{\infty}
a^{(2)}_{\nu}z_1^{k|a_2|}z_2^{1+k|a_1|},
\end{eqnarray*}
since $C^{-}_{a,a_1,a_2} \cap \{z_i=0\} \neq \emptyset$, for $i=1,2$.
Consider 
\begin{eqnarray*}
Pf (z) = \left(a^{(1)}_{(1,0)}z_1, a^{(2)}_{(0,1)}z_2\right).
\end{eqnarray*}
As in the proof of Lemma \ref{4.3},
it follows that
\begin{eqnarray*}\label{}
P(f \circ h) = Pf \circ Ph, \,\,\,{\rm and}\,\,\, P{\rm id} = {\rm id},
\end{eqnarray*}
and therefore 
\begin{eqnarray*}\label{}
Pf  \in GL(2, \mathbb{C})
\end{eqnarray*}
since $f$ is an automorphism.
Hence we have a representation of $GU(1,1)$ given by
\begin{eqnarray*}
GU(1,1) \ni g \longmapsto Pf \in GL(2,\mathbb{C}),
\end{eqnarray*}
where $f = \rho(g)$.
The restriction of this representation to the simple Lie group $SU(1,1)$
is nontrivial since $\rho(U(1) \times U(1)) = U(1) \times U(1)$.
However this contradicts Lemma~\ref{notsimple}.
Thus the lemma is proved.

\end{proof}

\medskip
Let us consider the case $\partial \Omega \neq A_{a,\lambda}$.
We prove that in this case $GU(1,1)$ does not act
effectively on $\Omega$, except {\bf Case (I'-ii)} below.

\smallskip

{\bf Case (I')}:
$(\partial \Omega \setminus A_{a,\lambda}) \cap (\C^* \times \C^*) = \emptyset$.
\\
In this case,
$\partial \Omega$ is the union of $A_{a,\lambda}$ and some of the following sets
\begin{equation}\label{set}
\{0\} \times \C^*,
\C^* \times \{0\}
\,\,\mathrm{or} \,\,
\{0\} \subset \C^{2},
\end{equation}
by the $G$-actions on the boundary of type 
(\ref{GU22}), (\ref{GU33}) and (\ref{GU44}).
If $\Omega \subset D^{-}_{a,a_1,a_2}$,
then the sets in (\ref{set})
can not be contained in the boundary of $\Omega$.
Thus we must consider only the case $\Omega \subsetneq D^{+}_{a,a_1,a_2}$
or $C^{-}_{a,a_1,a_2}$.

\smallskip
{\bf Case (I'-i)} : $\Omega \subsetneq D^{+}_{a,a_1,a_2}$.\\
In this case, $\mathbb{C}^* \times \{0\}$ can not be a subset of the boundary of
$\Omega$,
and $\{0\} \in A_{a,\lambda}$.
Thus 
\begin{eqnarray*}
&&\partial \Omega = A_{a,\lambda} \cup (\{0\} \times \C),\\
&&\Omega = D^{+}_{a,a_1,a_2} \setminus (\{0\} \times \C).
\end{eqnarray*}
Then $\Omega$ is biholomorphic to $D^{-}_{a,a_1,a_2}$,
and this contradicts Lemma~\ref{4.61}.
Thus this case does not occur.

\smallskip
{\bf Case (I'-ii)}:
$\Omega \subsetneq C^{-}_{a,a_1,a_2}$.\\
In this case, $\Omega$ coincides with one of the followings:
\begin{eqnarray*}
&&C_1=C^{-}_{a,a_1,a_2} \setminus (\mathbb{C} \times \{0\}) \cup (\{0\} \times \mathbb{C}),
\\
&&C_2=C^{-}_{a,a_1,a_2} \setminus (\mathbb{C} \times \{0\}),\\
&&C_3=C^{-}_{a,a_1,a_2} \setminus (\{0\} \times \mathbb{C}),\\
&&C_4=C^{-}_{a,a_1,a_2} \setminus \{0\}.
\end{eqnarray*}
Then we see that
$C_1 \subset \mathbb{C}^* \times \mathbb{C}^*$,
and using
biholomorphism
$\gamma_A^{-1}$ (see {\bf Case:} $\C^* \times \C^*$ in the proof of Lemma \ref{4.3}),
$C_1$ is biholomorphic to $\D_{a}\setminus\{0\} \times \C^*$.
Then, as the proof of Lemma~\ref{4.61}, this contradicts Lemma~\ref{c_200}.
Thus this case does not occur.

Suppose $\Omega = C_2$.
In this case, we have $\Omega \simeq D^{+}_{a,a_1,a_2}$, and therefore $\Omega \simeq D^{1,1}$
by Lemma~\ref{4.4}.

If $\Omega = C_3$,
then as the case $C_2$ above,
we see that $\Omega \simeq C^{+}_{a,a_1,a_2} \simeq D^{+}_{a^{-1},a_2,a_1}$, and $\Omega \simeq D^{1,1}$
by Lemma~\ref{4.4}.

Suppose $\Omega = C_4$.
Then $\Omega \cap \{z_i = 0\} \neq \emptyset$ for $i=1,2$.
Hence as the proof of Lemma \ref{4.6}, we derive a contradiction.
Thus this case does not occur.

\bigskip

{\bf Case (I\hspace{-.1em}I')}: 
$(\partial \Omega \setminus A_{a,\lambda}) \cap (\mathbb{C}^* \times \mathbb{C}^*) \neq \emptyset$.
\\
In this case, we can take a point 
$p'=(p'_1, p'_2) \in (\partial \Omega \setminus A_{a,\lambda}) \cap (\mathbb{C}^* \times \mathbb{C}^*)$.
We put 
\begin{eqnarray*}
&&b := (|p'_2|^2/|p'_1|^{2\lambda})^{|a_1|} > 0,\\
&&B_{b,\lambda} := \{(z_1, z_2) \in \mathbb{C}^{2} : -b|z_1|^{2|a_1|\lambda} + |z_2|^{2|a_1|} = 0\}.
\end{eqnarray*}
We may assume $a>b$ without loss of generality.

\smallskip

{\bf Case (I\hspace{-.1em}I'-i)}:
$\partial \Omega = A_{a,\lambda} \cup B_{b,\lambda}$. \\
Since $\Omega$ is connected, it coincides with
\begin{eqnarray*} 
C^{+}_{a,a_1,a_2} \cap D^{+}_{b,a_1,a_2} 
= \{ b|z_1|^{2|a_2|} < |z_2|^{2|a_1|} < a|z_1|^{2|a_2|} \},
\end{eqnarray*}
or
\begin{eqnarray*} 
C^{-}_{a,a_1,a_2} \cap D^{-}_{b,a_1,a_2} 
= \{ b < |z_1|^{2|a_2|}|z_2|^{2|a_1|} < a \}.
\end{eqnarray*}
These domains are subdomains of $\mathbb{C}^* \times \mathbb{C}^*$.
We see that 
using the biholomorphism
$\gamma_A^{-1}$ (see {\bf Case} $\C^* \times \C^*$ in the proof of Lemma \ref{4.3}),
these domains are biholomorphic to $\D_{(a,b)} \times \C^*$,
where
$\D_{(a,b)} = \{b < |z_1|^{2} < a\}$.
Indeed,
\[
\gamma_A^{-1}(C^{+}_{a,a_1,a_2} \cap D^{+}_{b,a_1,a_2}) = 
\{ b|z_1^{a_{11}}z_2^{a_{21}}|^{2|a_2|} < |z_1^{a_{12}}z_2^{a_{22}}|^{2|a_1|} < a|z_1^{a_{11}}z_2^{a_{21}}|^{2|a_2|} \}
= \{ b < |z_1|^{\pm 2} < a \},
\]
and similarly
\[
\gamma_A^{-1}(C^{-}_{a,a_1,a_2} \cap D^{-}_{b,a_1,a_2}) = 
\{ b < |z_1^{a_{11}}z_2^{a_{21}}|^{2|a_2|}|z_1^{a_{12}}z_2^{a_{22}}|^{2|a_1|} < a \}
= \{ b < |z_1|^{\pm 2} < a \}.
\]
Then, as the proof of Lemma~\ref{4.61}, this contradicts Lemma~\ref{c_200}.
Thus this case does not occur.

\smallskip

{\bf Case (I\hspace{-.1em}I'-ii)}: 
$\partial \Omega \neq A_{a,\lambda} \cup B_{b,\lambda}$.
\\
Suppose $(\partial \Omega \setminus A_{a,\lambda} \cup B_{b,\lambda}) \cap 
(\mathbb{C}^* \times \mathbb{C}^*) \neq \emptyset$,
then we can take 
\begin{equation*} 
p''= (p''_1, p''_2) \in (\partial \Omega \setminus A_{a,\lambda} \cup B_{b,\lambda}) 
\cap (\mathbb{C}^* \times \mathbb{C}^*).
\end{equation*}
Then put
\begin{eqnarray*} 
&&c=(|p''_2|^2/|p''_1|^{2\lambda})^{|a_1|},\\
&&C_{c,\lambda} = \{(z_1, z_2) \in \mathbb{C}^{2} : 
-c|z_1|^{2|a_1|\lambda} + |z_2|^{2|a_1|} = 0\}.
\end{eqnarray*}
We have $A_{a,\lambda} \cup B_{b,\lambda} \cup C_{c,\lambda} \subset \partial \Omega$.
However $\Omega$ is connected, this is impossible. 
Therefore this case does not occur.
Let us consider the remaining case:
\begin{eqnarray*} 
(\partial \Omega \setminus A_{a,\lambda} \cup B_{b,\lambda}) \cap
(\mathbb{C}^* \times \mathbb{C}^*) = \emptyset.
\end{eqnarray*}
However,
$\mathbb{C}^* \times \{0\}$, $\{0\} \times \mathbb{C}^*$
and $\{0\} \in \mathbb{C}^{2}$
can not be subsets of the boundary of $\Omega$
since $\Omega \subset C^{+}_{a,\lambda} \cap D^{+}_{b,\lambda}$
or
$\Omega \subset C^{-}_{a,\lambda} \cap D^{-}_{b,\lambda}$.
Thus this case does not occur either.

\smallskip
We have shown that if $c_1c_2 \neq 0$,
then $\Omega$ is biholomorphic to one of the following:
\begin{eqnarray*} 
\C^{2}, \C^{2} \setminus \{0\}\,\, \mathrm{or}\,\, D^{1,1}.
\end{eqnarray*}

\end{proof}

\section{The actions of $GU(p,q)$ for $p, q > 1$}
\label{sec:5}

Now we prove the following theorem.

\begin{theorem}  \label{M3}
Let $p, q >1$ and $n = p + q$.
Let $M$ be a connected complex manifold of dimension $n$ 
that is holomorphically separable
and admits a smooth envelope of holomorphy.
Assume that there exists an injective homomorphism of topological groups
$\rho_0 : GU(p,q) \longrightarrow \mathrm{Aut}(M)$.
Then $M$ is biholomorphic to one of $\C^{n}$, $\C^{n} \setminus \{0\}$, $D^{p,q}$ or
$C^{p,q}$.
\end{theorem}

\begin{proof}
The proof is similar to that of Theorem~\ref{M1} and in this case simpler.
By Lemma~\ref{4} and the comments after that, 
we can assume that $M$ is a Reinhardt domain $\Omega$ in $\C^{p+q}$,
$U(q) \times U(p)$-action on $\Omega \subset \C^{p+q}$ is linear
and $\rho(T^{p+q}) = T^{p+q}$.
We will 
prove that $\Omega$ is biholomorphic to  one of  the four domains
$\C^{p+q}$, $\C^{p+q} \setminus \{0\}$, $D^{p,q}$
or $C^{p,q}$.

Put a coordinate $(z_1, \ldots,z_q, z_{1+q}, \ldots, z_{p+q})$ of
$\mathbb{C}^{p+q}$.
Since $\rho(\mathbb{C}^*)$ is 
commutative with $\rho(T^{p+q}) = T^{p+q} \subset \aut(\Omega)$,
Lemma~\ref{KS1} tells us that $\rho(\mathbb{C}^*) \subset \Pi(\Omega)$,
that is, $\rho(\mathbb{C}^*)$ is represented by diagonal matrices.
Furthermore, $\rho(\mathbb{C}^*)$ commutes with 
$\rho(U(q) \times U(p)) = U(q) \times U(p)$,
so that we have 
\begin{eqnarray*}
\rho\left(e^{2\pi i(s+it)}\right) =
\mathrm{diag}\left[
e^{2\pi i\{a_1s+(b_1+ic_1)t\}}E_q, e^{2\pi i\{a_2s+(b_2+ic_2)t\}}E_p
\right]
\in \rho(\mathbb{C}^*),
\end{eqnarray*}
where
$s,t \in \mathbb{R}$, $a_1, a_2 \in \mathbb{Z}, b_1,b_2,c_1,c_2 \in \mathbb{R}$.
Since $\rho$ is injective,
$a_1, a_2$ are
relatively prime
and $(c_1, c_2) \neq (0, 0)$.
Since $T_{q,p}$
is the center of the group $U(q) \times U(p)$,
we have $\rho(T_{q,p}) = T_{q,p} \subset \mathrm{Aut}(\Omega)$.
To consider the actions of $\mathbb{C}^*$ and $U(q) \times U(p)$ 
on $\Omega$ together,
we put
\begin{equation*} 
G(U(q) \times U(p)) = 
\left\{e^{-2\pi t}\cdot
\mathrm{diag}\left[
U_1, U_2
\right]
\in GU(p,q) : t \in \mathbb{R}, U_1 \in U(q), U_2 \in U(p)
\right\}.
\end{equation*}
Then we have
\begin{eqnarray*} 
G &:=&\rho(G(U(q) \times U(p))) \\
&{}=&
\left\{
\mathrm{diag}\left[
e^{-2\pi c_1t}U_1, e^{-2\pi c_2t}U_2
\right]
\in GL(p+q,\mathbb{C}) : t \in \mathbb{R}, U_1 \in U(q), U_2 \in U(p)
\right\}.
\end{eqnarray*}
Note that $G$ is the centralizer
of $T_{q,p} = \rho(T_{q,p})$ in $\rho(GU(p, q))$.

Let $f =(f_1, \ldots, f_q, f_{1+q}, \ldots, f_{p+q}) \in \rho(GU(p, q))$ 
and consider the Laurent series of its components:\\
\begin{eqnarray}\label{eqfi''}
f_i(z_1, \ldots, z_q, z_{1+q} \ldots, z_{p+q}) 
= \sum_{\nu \in \mathbb{Z}^{p+q}} a^{(i)}_{\nu}z^{\nu},
\end{eqnarray}
As Lemma \ref{cl1} in the proof of Theorem~\ref{M1}, we have:
\begin{lemma}\label{cl1''}
For any $f \in \rho(GU(p, q)) \setminus G$,
there exists $\nu \in \mathbb{Z}^{p+q}$, $\nu \neq e_1$,
$\ldots,$ $e_q$,
such that $a^{(i)}_{\nu} \neq 0$ in $(\ref{eqfi''})$  for some $1 \leq i \leq q$,
or 
there exists
$\nu \in \mathbb{Z}^{p+q}$, $\nu \neq e_{1+q}$, 
$\ldots,$ $e_{p+q}$
such that
$a^{(i)}_{\nu} \neq 0$ in $(\ref{eqfi''})$ for some $1+q \leq i \leq p+q$,
where $e_1=(1,0,\ldots,0)$, $e_2=(0,1,0\ldots,0)$, $\ldots,e_{p+q}=(0,\ldots,0,1)$ are
the natural basis of $\Z^{p+q}$.
\end{lemma}

\medskip
By Lemma~\ref{LU},
there are no negative degree terms 
of $z_1,\ldots,z_{p+q}$ in $(\ref{eqfi''})$.
Write $\nu=(\nu', \nu'')$, where $\nu'=(\nu_1, \ldots, \nu_q)$ and 
$\nu''=(\nu_{1+q}, \ldots, \nu_{p+q})$,
and $|\nu'| = \nu_1 + \cdots + \nu_q$, $|\nu''| = \nu_{1+q} + \cdots + \nu_{p+q}$.
Let us consider $\nu' \in (\Z_{\geq 0})^q$ and 
$\nu'' \in (\Z_{\geq 0})^p$,
and put
\begin{eqnarray*}
{\sum_{\nu}}' = \sum_{\nu' \in (\Z_{\geq 0})^q, \nu'' \in (\Z_{\geq 0})^p},
\end{eqnarray*}
$(z')^{\nu'} = z_1^{\nu_1}\cdots z_n^{\nu_q}$ and
$(z'')^{\nu''} = z_{1+q}^{\nu_{1+q}}\cdots z_{p+q}^{\nu_{p+q}}$
from now on.
When we need to distinguish $\nu$ for $f_i$, $1 \leq i \leq p+q$,
we write
$\nu = \nu^{(i)} = (\nu_1^{(i)}, \ldots, \nu_{p+q}^{(i)})$.

Since $\mathbb{C}^*$ is the center of $GU(p,q)$, 
it follows that, 
for $f \in \rho(GU(n,1))$, 
\begin{eqnarray*}
\rho(e^{2\pi i(s+it)}) \circ f = f \circ \rho(e^{2\pi i(s+it)}).
\end{eqnarray*}
By $(\ref{eqfi''})$,
this equation means
\begin{eqnarray*}
e^{2\pi i\{a_1s+(b_1+ic_1)t\}} {\sum_{\nu}}' a^{(i)}_{\nu}z^{\nu}
&=& 
{\sum_{\nu}}' a^{(i)}_{\nu}
(e^{2\pi i\{a_1s+(b_1+ic_1)t\}}z')^{\nu'}
(e^{2\pi i\{a_2s+(b_2+ic_2)t\}}z'')^{\nu''}\\
&=&
{\sum_{\nu}}' a^{(i)}_{\nu}
e^{2\pi i\{a_1s+(b_1+ic_1)t\}|\nu'|}
e^{2\pi i\{a_2s+(b_2+ic_2)t\}|\nu''|}
z^{\nu}
\end{eqnarray*}
for $1 \leq i \leq q$,
and
\begin{eqnarray*}
e^{2\pi i\{a_2s+(b_2+ic_2)t\}} {\sum_{\nu}}' a^{(i)}_{\nu}z^{\nu}
&=&
{\sum_{\nu}}' a^{(i)}_{\nu}
(e^{2\pi i\{a_1s+(b_1+ic_1)t\}}z')^{\nu'}
(e^{2\pi i\{a_2s+(b_2+ic_2)t\}}z'')^{\nu''}\\
 &=&
{\sum_{\nu}}' a^{(i)}_{\nu}
e^{2\pi i\{a_1s+(b_1+ic_1)t\}|\nu'|}
e^{2\pi i\{a_2s+(b_2+ic_2)t\}|\nu''|}
z^{\nu},
\end{eqnarray*}
for $1+q \leq i \leq p+q$.
Thus for each $\nu \in \mathbb{Z}^{p+q}$, we have
\begin{eqnarray*}
e^{2\pi i\{a_1s+(b_1+ic_1)t\}} a^{(i)}_{\nu}
= e^{2\pi i\{a_1s+(b_1+ic_1)t\}|\nu'|}
e^{2\pi i\{a_2s+(b_2+ic_2)t\}|\nu''|}
a^{(i)}_{\nu},
\end{eqnarray*}
for $1 \leq i \leq q$, and
\begin{eqnarray*}
e^{2\pi i\{a_2s+(b_2+ic_2)t\}} a^{(i)}_{\nu}
= e^{2\pi i\{a_1s+(b_1+ic_1)t\}|\nu'|} e^{2\pi i\{a_2s+(b_2+ic_2)t\}|\nu''|} a^{(i)}_{\nu},
\end{eqnarray*}
for $1+q \leq i \leq p+q$.
Therefore, if
$a_{\nu}^{(i)}\neq0$ for
$1 \leq i \leq q$,
we have
\begin{eqnarray}\label{ac0''}
\hspace{-2cm}
\left\{\begin{array}{l}
a_1(\nu_1^{(i)} + \cdots + \nu_q^{(i)}-1) + a_2(\nu_{1+q}^{(i)} + \cdots + \nu_{p+q}^{(i)}) = 0,\\
c_1(\nu_1^{(i)} + \cdots + \nu_q^{(i)}-1) + c_2(\nu_{1+q}^{(i)} + \cdots + \nu_{p+q}^{(i)}) = 0.
\end{array}
\right.
\end{eqnarray}
and if $a_{\nu}^{(i)} \neq 0$
for
$1+q \leq i \leq p+q$,
we have
\begin{eqnarray}\label{aci''}
\hspace{-20mm}
\left\{\begin{array}{l}
a_1(\nu_{1}^{(i)} + \cdots + \nu_q^{(i)}) + a_2(\nu_{1+q}^{(i)} + \cdots + \nu_{p+q}^{(i)}-1) = 0,\\
c_1(\nu_1^{(i)} + \cdots + \nu_q^{(i)}) + c_2(\nu_{1+q}^{(i)} + \cdots + \nu_{p+q}^{(i)}-1) = 0.
\end{array}
\right.
\end{eqnarray}

\begin{lemma}\label{c_20''}
$c_1c_2 \neq 0$.
\end{lemma}

\begin{proof}
We assume $c_1=0$, $c_2 \neq 0$.
If $a_1 \neq 0$, then by (\ref{ac0''}),
$\nu_1^{(i)} + \cdots + \nu_q^{(i)} = 1$, 
$\nu_{1+q}^{(i)} + \cdots + \nu_{p+q}^{(i)} = 0$,
for $1 \leq i \leq q$,
and by (\ref{aci''})
$\nu_{1}^{(i)} + \cdots + \nu_q^{(i)} = 0$ and
$\nu_{1+q}^{(i)} + \cdots + \nu_{p+q}^{(i)} = 1$ for $1+q \leq i \leq p+q$. 
However this contradicts Lemma \ref{cl1''}.
Therefore we assume $a_1 = 0$.
As the proof of Lemma \ref{c_20},
$\Omega \subset  \mathbb{C}^{p+q}$ can be written
of the form
$(D \times \mathbb{C}^p) \cup \left(D' \times (\mathbb{C}^p \setminus \{0\})\right)$
by $G$-action on $\Omega$,
where $D$ and $D'$ are open sets in $\mathbb{C}^q$, $D\subset D'$ and $D'$ is connected.
On the other hand,
the functions $f_i$, $1 \leq i \leq q$, of
$f = (f_1, \ldots, f_{p+q}) \in \rho(GU(p,q))$ do
not depend on the variables $(z_{1+q}, \ldots, z_{p+q})$ by (\ref{ac0''}).
Hence $(f_1,\ldots,f_q)$ is an automorphism of $D$ and $D'$,
i.e. $GU(p,q)$ acts on $D$ and $D'$.
Since $U(q)$ acts linearly on $D'$,
the domain $D'$ must coincide with one of the following open sets:
$\C^q$, $\C^q \setminus \{0\}$, $\C^q \setminus \overline{\B^q_r}$,
$\B^q_r$, $\B^q_r \setminus \{0\}$ and $\B^q_r \setminus \overline{\B^q_{r'}}$,
where $\B^q_r = \{ z \in \C : |z|<r\}$ and $r>r'>0$.
Thus we have a topological group homomorphism from $GU(p,q)$ to 
one of the topological groups
$\aut(\C^q)$,
$\aut(\C^q \setminus \{0\})$, $\aut(\C^q \setminus \overline{\B^q_r})$,
$\aut(\B^q_r)$, $\aut(\B^q_r \setminus \{0\})$
and 
$\aut(\B^q_r \setminus \overline{\B^q_{r'}})$.

We now prove $D'\neq\C^q$.
By Lemma~\ref{extG},
the $U(p,q)$-action extends to a holomorphic action of $GL(p+q,\C)$.
The categorical quotient $\C^{p+q}/\hspace{-1mm}/GL(p+q,\C)$ is then one point
(see the sentence before Lemma~\ref{linG}).
By Lemma~\ref{linG}, the $GL(p+q,\C)$-action is linearizable.
However, the restriction of this action to $SU(p,q)$ is non-trivial
since $S(U(q)\times SU(p)) \subset \rho(SU(p,q))$ acts non-trivially on $\C^q$.
This contradicts Lemma~\ref{notsimple}.
Thus $D'\neq\C^q$.
Furthermore, $D'\neq \C^q \setminus \{0\}$,
$\C^q \setminus \overline{\B^q}$,
since the $U(p,q)$-action extends to the action on $\C^q$
by the Hartogs extension theorem.
Thus these cases come down to the $\C^q$-case.

$SU(p,q)$ can not act non-trivially on $\B^q_r$, $\B^q_r \setminus \{0\}$ and
$\B^q_r \setminus \overline{\B^q_{r'}}$ either.
Indeed, if $SU(p,q)$ acts,
then $\rho(SU(p,q)) \subset PU(q,1)$ for the first case,
and $\rho(SU(p,q)) \subset U(q)$ for the latter two cases.
Here we used $\aut(\B^q_r)=PU(q,1)$
and $\aut(\B^q_r \setminus \{0\})=\aut(\B^q_r \setminus \overline{\B^q_{r'}})=U(q)$.
By Lemma~\ref{notsimple2} and Lemma~\ref{notsimple}, these actions are trivial.
Thus, these cases do not occur. 
Consequently, we see that the case $c_1=0$, $c_2\neq 0$ does not occur.
In the same manner, we can see that the case $c_1\neq 0$, $c_2 = 0$
does not occur.
\end{proof}

Then we have
\begin{lemma}\label{cla1''}
$a_1=a_2$ and $|a_1|=|a_2|=1$.
\end{lemma}

\begin{proof}

Take $f \in \rho(GU(p,q)) \setminus G$
with the Laurent expansions (\ref{eqfi''}).
By Lemma \ref{cl1''}, (\ref{ac0''}) and (\ref{aci''}),
we have
\begin{eqnarray*}
\lambda:= a_2/a_1 = c_2/c_1 \in \mathbb{Q} \setminus \{0\}
\end{eqnarray*}
We first prove that $\lambda$ is positive.
For this, we suppose that $\lambda$ is negative.
Since 
$|\nu'|, |\nu''| \geq 0$ and
$a_1, a_2$ are relatively prime,
we have,
by (\ref{ac0''}), for $1 \leq i \leq q$,
\begin{eqnarray*}
\nu_{1}^{(i)} + \cdots + \nu_q^{(i)} = 1+k|a_2|
\,\,\,\,\,\mathrm{and}\,\,\,\,\, 
\nu_{1+q}^{(i)} + \cdots + \nu_{p+q}^{(i)} = k|a_1|,
\end{eqnarray*}
where $k \in \Z_{\geq 0}$,
and, by (\ref{aci''}), for $1+q \leq i \leq p+q$,
\begin{eqnarray*}
\nu_{1}^{(i)} + \cdots + \nu_q^{(i)} = l|a_2|
\,\,\,\,\,\mathrm{and}\,\,\,\,\,
\nu_{1+q}^{(i)} + \cdots + \nu_{p+q}^{(i)} = 1+l|a_1|,
\end{eqnarray*}
where $l \in \Z_{\geq 0}$.
Hence, the Laurent series of the components of $f \in \rho(GU(n,1))$ are
\begin{eqnarray}\label{form0''}
f_i(z', z'') 
= \sum_{k=0}^{\infty} \, {\sum_{|\nu'|=1+k|a_2|,\atop |\nu''|=k|a_1|}}^{\hspace{-13pt}\prime} \,\,\,
a^{(i)}_{\nu}(z')^{\nu'}(z'')^{\nu''},
\end{eqnarray}
for $1 \leq i \leq q$, and
\begin{eqnarray}\label{formi''}
f_i(z', z'') 
= \sum_{k=0}^{\infty} \, {\sum_{|\nu'|=k|a_2|,\atop |\nu''|=1+k|a_1|}}^{{\hspace{-13pt}\prime}} \,\,\,
a^{(i)}_{\nu}(z')^{\nu'}(z'')^{\nu''},
\end{eqnarray}
for $1+q \leq i \leq p+q$.
We focus on the first degree terms of the Laurent expansions.
We put
\begin{eqnarray}\label{Pf''}
\,\,\,\,\,\,\,\,\,\,\,\,\,\,\,\,\,\,\,\,\,
Pf (z) := \left(
{\sum_{|\nu'|=1,\atop |\nu''|=0}}^{{\hspace{-4pt}\prime}} a^{(1)}_{\nu'}z', \ldots,
{\sum_{|\nu'|=1,\atop |\nu''|=0}}^{{\hspace{-4pt}\prime}} a^{(q)}_{\nu'}z',
{\sum_{|\nu'|=0,\atop |\nu''|=1}}^{{\hspace{-4pt}\prime}} a^{(1+q)}_{\nu''}(z'')^{\nu''}, \ldots,
{\sum_{|\nu'|=0,\atop |\nu''|=1}}^{{\hspace{-4pt}\prime}} a^{(p+q)}_{\nu''}(z'')^{\nu'}\right).
\end{eqnarray}
As a matrix we can write
\begin{eqnarray*}
Pf = 
\begin{pmatrix}
a^{(1)}_{e_1} & \cdots & a^{(1)}_{e_q} & 0 & \cdots & 0 &\\
\vdots  & \ddots &  \vdots & \vdots & \ddots & \vdots \\
a^{(q)}_{e_1} & \cdots & a^{(q)}_{e_q} & 0 & \cdots & 0 \\
0 & \cdots & 0 & a^{(1+q)}_{e_{1+q}} & \cdots & a^{(1+q)}_{e_{p+q}}\\
 \vdots  & \ddots &  \vdots & \vdots & \ddots & \vdots\\
0 & \cdots & 0 & a^{(p+q)}_{e_{1+q}} & \cdots & a^{(p+q)}_{e_{p+q}}
\end{pmatrix}.
\end{eqnarray*}
Then it follows from (\ref{form0''}) and (\ref{formi''}) that
\begin{eqnarray*}\label{}
P(f \circ h) = Pf \circ Ph, \,\,\,{\rm and}\,\,\, P{\rm id} = {\rm id},
\end{eqnarray*}
where $h \in \rho(GU(p,q))$,
and therefore 
\begin{eqnarray*}\label{}
Pf  \in GL(p+q, \mathbb{C})
\end{eqnarray*}
since $f$ is an automorphism.
Hence we have a representation of $GU(p,q)$ given by
\begin{eqnarray*}
GU(p,q) \ni g \longmapsto Pf \in GL(p+q,\mathbb{C}),
\end{eqnarray*}
where $f = \rho(g)$.
The restriction of this representation to the simple Lie group $SU(p,q)$
is nontrivial since $\rho(U(q) \times U(p)) = U(q) \times U(p)$.
However this contradicts Lemma~\ref{notsimple}.
Thus $\lambda$ is positive.

Next we prove that $|a_1|=|a_2|=1$.
For this, we suppose $|a_1| \neq 1$.
Since $|\nu'|, |\nu''| \geq 0$ and
$a_1, a_2$ are relatively prime,
we have,
by (\ref{ac0''}),
\begin{eqnarray*}
\nu_{1}^{(i)} + \cdots + \nu_q^{(i)} = 1-k|a_2|
\,\,\,\,\,\mathrm{and}\,\,\,\,\, 
\nu_{1+q}^{(i)} + \cdots + \nu_{p+q}^{(i)} = k|a_1|,
\end{eqnarray*}
for $1 \leq i \leq q$, where $k=0$,
and, in addition, $k=1$ is valid if $|a_2|=1$.
By (\ref{aci''}),
\begin{eqnarray*}
\nu_{1}^{(i)} + \cdots + \nu_q^{(i)} = 0
\,\,\,\,\,\mathrm{and}\,\,\,\,\,
\nu_{1+q}^{(i)} + \cdots + \nu_{p+q}^{(i)} = 1,
\end{eqnarray*}
for $1+q \leq i \leq p+q$.
Hence, the Laurent series of the components of $f \in \rho(GU(p,q))$ are
\begin{eqnarray*}\label{form0'''}
f_i(z', z'') 
= \sum_{k=0}^{1} \, {\sum_{|\nu'|=1+k|a_2|,\atop |\nu''|=k|a_1|}}^{\hspace{-13pt}\prime} \,\,\,
a^{(i)}_{\nu}(z')^{\nu'}(z'')^{\nu''},
\end{eqnarray*}
for $1 \leq i \leq q$, and
\begin{eqnarray}\label{formi'''}
f_i(z', z'') 
= {\sum_{|\nu''|=1}}^{{\hspace{-5pt}\prime}} \,\,\,
a^{(i)}_{\nu}(z'')^{\nu''},
\end{eqnarray}
for $1+q \leq i \leq p+q$.
Thus, by (\ref{formi'''}),
we have a nontrivial linear representation of $GU(p,q)$ by
\begin{eqnarray*}
GU(p,q) \ni g \longmapsto (f_{1+q},\ldots,f_{p+q}) \in GL(p,\mathbb{C}),
\end{eqnarray*}
where $(f_1,\ldots,f_{p+q}) = \rho(g)$.
However this contradicts Lemma~\ref{notsimple}.
Thus we have shown that $|a_1|=1$.
In the same manner, we have $|a_2|=1$,
and therefore Lemma \ref{cla1''} is proven.

\end{proof}

\begin{remark}\label{rem1''}
By Lemma \ref{cla1''}, (\ref{ac0''}) and (\ref{aci''}),
the action of $GU(p,q)$ on $\Omega$ is linear matrix multiplication,
since $\nu^{(i)}_{\nu} \geq 0$ 
for $1 \leq i \leq p+q$.
Therefore 
\begin{eqnarray*}
\rho: GU(p,q) \longrightarrow GL(p+q,\mathbb{C}),
\end{eqnarray*}
and this representation is irreducible by Lemma \ref{notsimple}.
\end{remark}

\medskip
Since $G=\rho(G(U(q) \times U(p)))$ acts as linear transformations on 
$\Omega \subset \mathbb{C}^{p+q}$, 
it preserves the boundary $\partial \Omega$ of $\Omega$.
We now study the action of $G$ on $\partial \Omega$.
The $G$-orbits of points in $\mathbb{C}^{p+q}$ 
consist of four types as follows:

(i) If 
$p = (p_1, \ldots, p_{p+q}) \in
(\mathbb{C}^q \setminus \{0\}) \times (\mathbb{C}^p \setminus \{0\})$,
then
\begin{equation} \label{GU1''}
G \cdot p = \{ (z_1, \ldots, z_{p+q}) \in \mathbb{C}^{p+q} \setminus \{0\}
 : -a(|z_1|^2 + \cdots + |z_q|^2) + |z_{1+q}|^2 + \cdots + |z_{p+q}|^2 = 0\},
\end{equation}
where
$a:=(|p_{1+q}|^2 + \cdots + |p_{p+q}|^2)/(|p_1|^2 + \cdots + |z_q|^2) > 0$.

\smallskip

(ii) If $p' = (0,\ldots,0, p'_{1+q}, \ldots, p'_{p+q}) \in \mathbb{C}^{p+q}
\setminus \{0\}$,
then
\begin{equation} \label{GU2''}
G \cdot p' = \{(0,\ldots,0)\} \times (\mathbb{C}^{p} \setminus \{0\}).
\end{equation}

\smallskip

(iii) If $p'' = (p''_1, \ldots, p''_q, 0, \ldots, 0) \in \mathbb{C}^{p+q}
\setminus \{0\}$,
then
\begin{equation} \label{GU3''}
G \cdot p'' = (\mathbb{C}^q \setminus \{0\}) \times \{0\}.
\end{equation}

\smallskip

(iv) If $p''' = (0, \ldots, 0) \in \mathbb{C}^{p+q}$,
then
\begin{equation} \label{GU4''}
G \cdot p''' = \{0\} \subset \mathbb{C}^{p+q}.
\end{equation}

\smallskip

\begin{lemma}\label{5.4}
If 
\begin{equation*} 
\Omega \cap ((\C^q \setminus \{0\}) \times (\C^p \setminus \{0\})) 
= (\C^q \setminus \{0\}) \times (\C^p \setminus \{0\}),
\end{equation*}
then $\Omega$ is equal to $\C^{p+q}$ or $\C^{p+q} \setminus \{0\}$.
\end{lemma}

\begin{proof}
If $\Omega \cap ((\C^q \setminus \{0\}) \times (\C^p \setminus \{0\}))
= ((\C^q \setminus \{0\}) \times (\C^p \setminus \{0\}))$,
then 
$\Omega$ equals one of the following domains 
by the $G$-actions of the type (\ref{GU2''}), (\ref{GU3''}) and (\ref{GU4''}) above:
\begin{equation*} 
\C^{p+q},
\C^{p+q}  \setminus \{0\},
\C^q \times (\C^p \setminus \{0\}),
(\C^p \setminus \{0\}) \times \C^p
\,\,\mathrm{or} \,\,(\C^q \setminus \{0\}) \times (\C^p \setminus \{0\}).
\end{equation*}
Clearly $GU(p,q)$ acts on $\C^{p+q}$ and $\C^{p+q} \setminus \{0\}$
by matrices multiplications.
However, $GU(p,q)$-action on $\C^{p+q}$ does not preserve  
$\C^q \times (\C^p \setminus \{0\})$,
$(\C^p \setminus \{0\}) \times \C^p$
or $(\C^q \setminus \{0\}) \times (\C^p \setminus \{0\})$,
by {\it Remark} \ref{rem1''}.
Thus Lemma~\ref{5.4} is proved.

\end{proof}

Let us consider the case 
$\Omega \cap ((\C^q \setminus \{0\}) \times (\C^p \setminus \{0\}))
\neq (\C^q \setminus \{0\}) \times (\C^p \setminus \{0\})$.
Then we have
$\partial \Omega \cap ((\C^q \setminus \{0\}) \times (\C^p \setminus \{0\}))
\neq \emptyset$.
Thus we can take a point 
\begin{eqnarray*} 
p=(p_1, \ldots, p_{p+q}) \in \partial \Omega \cap 
((\C^q \setminus \{0\}) \times (\C^p \setminus \{0\})).
\end{eqnarray*} 
Let 
\begin{eqnarray*} 
&&a : =(|p_{1+q}|^2 + \cdots + |p_{p+q}|^2)/(|p_1|^2 + \cdots + |p_q|^2) > 0,\\
&&A_a := \{ (z_1, \ldots, z_{p+q}) \in \mathbb{C}^{p+q} : 
-a(|z_1|^2 + \cdots + |z_q|^2) + |z_{1+q}|^2 + \cdots + |z_{p+q}|^2 = 0\}.
\end{eqnarray*}
Note that 
\begin{equation*} 
\partial \Omega  \supset A_a,
\end{equation*}
by the $G$-action of the type (\ref{GU1''}).
Then $\Omega$ is contained in
\begin{equation*} 
D_a = \{ -a(|z_1|^2 + \cdots + |z_q|^2) + |z_{1+q}|^2 + \cdots + |z_{p+q}|^2 > 0\}
\end{equation*}
or 
\begin{equation*} 
C_a = \{ -a(|z_1|^2 + \cdots + |z_q|^2) + |z_{1+q}|^2 + \cdots + |z_{p+q}|^2 < 0\}.
\end{equation*}

Let us first consider the case $\partial \Omega = A_a$,
that is, $\Omega = D_a$ or $C_a$.

\begin{lemma}\label{}
If $\Omega=D_a$,  
then $\Omega$ is biholomorphic to $D^{p,q}$.
\end{lemma}

\begin{proof}
There exists a biholomorphic map from $\Omega$ to $D^{p,q}$ by
\begin{eqnarray*}
\Phi : \C^{p+q} \ni (z_1, \ldots, z_{p+q}) \mapsto (a^{-1/2}z_1,\ldots,a^{-1/2}z_q, z_{1+q}, \ldots, z_{p+q}) \in \C^{p+q}.
\end{eqnarray*}

\end{proof}

\begin{lemma}\label{eqthm1}
If $\Omega=C_a$,  
then $\Omega$ is biholomorphic to $C^{p,q}$.
\end{lemma}

\begin{proof}
Indeed, there exists a biholomorphic map from $\Omega$ to $C^{p,q}$ by
\begin{eqnarray*}
\Phi' : \C^{p+q} \ni (z_1, \ldots, z_{p+q}) \mapsto 
(a^{-1/2}z_1,\ldots,a^{-1/2}z_q, z_{1+q}, \ldots, z_{p+q}) \in \C^{p+q}.
\end{eqnarray*}

\end{proof}

\medskip
Let us consider the case $\partial \Omega \neq A_{a}$.
We will prove that in this case $GU(p,q)$ does not act
effectively on $\Omega$.

\smallskip

{\bf Case (I'')}:
$(\partial \Omega \setminus A_a) \cap ((\C^q \setminus \{0\}) \times (\C^p \setminus \{0\}))
 = \emptyset$.
\\
In this case,
$\partial \Omega$ is the union of $A_a$ and some of the following sets
\begin{equation}\label{}
\{0\} \times \C^p
\,\,\mathrm{or} \,\,
\C^q \times \{0\},
\end{equation}
by the $G$-actions on the boundary of the type 
(\ref{GU2''}) and  (\ref{GU3''}).

\smallskip
{\bf Case (I''-i)}:
$\Omega \subsetneq D_a$.\\
In this case, $\mathbb{C}^q \times \{0\}$ can not be a subset of the boundary of
$\Omega$.
Thus 
\begin{eqnarray*}
&&\partial \Omega = A_a \cup (\{0\} \times \C^p),\\
&&\Omega = D_a \setminus (\{0\} \times \C^p).
\end{eqnarray*}
However, the $GU(p,q)$-action on $\C^{p+q}$ does not preserve  
$D_a \setminus (\{0\} \times \C^p)$,
by  {\it Remark} \ref{rem1''}.
Thus this case does not occur.

\smallskip
{\bf Case (I''-ii)}:
$\Omega \subsetneq C_a$.\\
In this case, $\{0\} \times (\mathbb{C}^p \setminus \{0\})$ 
can not be a subset of the boundary of $\Omega$.
Thus 
\begin{eqnarray*}
&&\partial \Omega = A_a \cup (\mathbb{C}^q \times \{0\}),\\
&&\Omega = C_a \setminus (\mathbb{C}^q \times \{0\}).
\end{eqnarray*}
However, the $GU(p,q)$-action on $\C^{p+q}$ does not preserve  
$C_a \setminus (\mathbb{C}^q \times \{0\})$,
by {\it Remark} \ref{rem1''}
Thus this case does not occur either.

\bigskip

{\bf Case (I\hspace{-.1em}I'')}: 
$(\partial \Omega \setminus A_a) \cap ((\mathbb{C}^q \setminus \{0\}) \times (\mathbb{C}^p \setminus \{0\})) \neq \emptyset$.
\\
In this case, we can take a point 
$p'=(p'_1, \ldots, p'_{p+q}) \in (\partial \Omega \setminus A_a) \cap ((\mathbb{C}^q \setminus \{0\}) \times (\mathbb{C}^p \setminus \{0\}))$.
We put 
\begin{eqnarray*}
&&b := (|p'_{1+q}|^2 + \cdots + |p'_{p+q}|^2)/(|p'_1|^2 + \cdots + |p'_q|^2) > 0,\\
&&B_b := \{(z_1, \ldots, z_{p+q}) \in \mathbb{C}^{p+q} : -b(|z_1|^2 + \cdots + |z_q|^2 ) + |z_{1+q}|^2 + \cdots + |z_{p+q}|^2 = 0\}.
\end{eqnarray*}
We may assume $a>b$ without loss of generality.

\smallskip

{\bf Case (I\hspace{-.1em}I''-i)}:
$\partial \Omega = A_{a} \cup B_{b}$. \\
Since $\Omega$ is connected, it coincides with
\begin{eqnarray*} 
C_a \cap D_b 
&=& \{ b(|z_1|^2 + \cdots + |z_q|^2 ) < |z_{1+q}|^2 + \cdots + |z_{p+q}|^2 < a(|z_1|^2 + \cdots + |z_q|^2 ) \}\\
&\simeq& \{ \frac{b}{a}(|z_1|^2 + \cdots + |z_q|^2 ) < |z_{1+q}|^2 + \cdots + |z_{p+q}|^2 < |z_1|^2 + \cdots + |z_q|^2 \}.
\end{eqnarray*}
Since $GU(p,q)$ is connected, the $GU(p,q)$-action through $\rho$ preserves
subsets of boundary $\partial \Omega$,
$\{|z_{1+q}|^2 + \cdots + |z_{p+q}|^2 = |z_1|^2 + \cdots + |z_q|^2 \}\setminus\{0\}$
and $\{ \frac{b}{a}(|z_1|^2 + \cdots + |z_q|^2 ) = |z_{1+q}|^2 + \cdots + |z_{p+q}|^2\}\setminus\{0\}$,
which are disjoint.
Then the $GU(p,q)$-action also preserves the domain
$\{|z_{1+q}|^2 + \cdots + |z_{p+q}|^2 > |z_1|^2 + \cdots + |z_q|^2\}\simeq D^{p,q}$.
However, by Lemma~\ref{condGU},
$\rho(GU(p,q))\subset GU(p,q)$, and
therefore,
noticing {\it Remark} \ref{rem1''},
we see that the $GU(p,q)$-action does not preserve
$\{\frac{b}{a}(|z_1|^2 + \cdots + |z_q|^2 ) = |z_{1+q}|^2 + \cdots + |z_{p+q}|^2\}$,
a contradiction.
Thus this case does not occur.

\smallskip

{\bf Case (I\hspace{-.1em}I''-ii)}: 
$\partial \Omega \supsetneq  A_a \cup B_b$.
\\
By the same argument of {\bf Case (I\hspace{-.1em}I-ii)}
in the proof of Theorem \ref{M1},
we see that this case does not occur.

\smallskip
We have shown that
$\Omega$ is biholomorphic to one of the following domains:
\begin{eqnarray*} 
\C^{p+q}, \C^{p+q} \setminus \{0\}, D^{p,q}\,\, \mathrm{or}\,\, C^{p,q}.
\end{eqnarray*}

\end{proof}

\section{The automorphism groups of the domains}
\label{sec:6}

\subsection{The automorphism groups of $\C \times \B^n$ and $\C^* \times \B^n$}
\label{sec:6.1}

\begin{theorem}[\cite{BKS0}]\label{autCB}
For $f=(f_{0}, f_{1}, \ldots, f_{n}) \in \mathrm{Aut}(\C \times \B^n)$,
we have
\begin{align*}
f_{0}(z_{0}, z_{1}, \ldots, z_{n})=
a(z_{1}, \ldots, z_{n}) z_{0} + b(z_{1}, \ldots, z_{n}),
\end{align*}
and
\begin{align*}
f_{i}(z_{0}, z_{1}, \ldots, z_{n}) =
\frac{a_{i 0}+\sum_{j=1}^{n} a_{i j} z_{j}}{a_{00}+\sum_{j=1}^{n} a_{0 j} z_{j}}
\end{align*}
for $i=1, \ldots, n$,
where $a$ is a nowhere vanishing holomorphic function on 
$\mathbb{B}^{n}$, 
$b$ is a holomorphic function on 
$\mathbb{B}^{n}$,
and the matrix $(a_{ij})_{0\leq i, j \leq n}$ is an element of $SU(n,1)$.  
\end{theorem}

\begin{proof}
First we consider $f_{i}$ for $i=1, \ldots, n$.
Fix $(z_{1}, \ldots z_{n}) \in \mathbb{B}^{n}$. 
Then, by the assumption,
$f_{i}(\cdot, z_{1}, \ldots, z_{n})$, for $i=1, \ldots, n$, 
are bounded holomorphic functions on $\mathbb{C}$,
and therefore
constant functions by the Liouville theorem.
Hence 
$f_i$, for $i=1, \ldots, n$, 
do not depend on  $z_{0}$.
In the same manner, we 
see that,
for the inverse of $f$
\begin{equation*}
g = (g_0, g_1,\ldots, g_n) = f^{{-1}} 
\in \mathrm{Aut}(\mathbb{C} \times \mathbb{B}^{n}),
\end{equation*}
$g_i$ for $i=1, \ldots n$ are independent of $z_{0}$.
It follows that 
\begin{equation*}
\overline{f} := (f_{1}, \ldots, f_{n}) \in \mathrm{Aut}(\mathbb{B}^{n}).
\end{equation*}
It is well-known (see \cite{Pyateskii}) that
$\overline{f} \in \mathrm{Aut}(\mathbb{B}^{n})$ 
is a linear fractional transformation, whose components are
of the form 
\begin{equation*}
f_{i}(z_{1}, \ldots, z_{n}) = \frac{a_{i0} + 
\sum_{j=1}^{n} a_{ij}z_{j}}{a_{00} + \sum_{j=1}^{n} a_{0j}z_{j}},
\end{equation*}
for $i=1,2, \ldots, n$,
where the matrix $(a_{i j})_{0\leq i, j \leq n}$ is an element of $SU(n,1)$.

Next we consider $f_0$ of $f$.
For fixed $(z_{1}, z_{2}, \ldots, z_{n}) \in \mathbb{B}^{n}$,
$f_{0}(\cdot, z_{1}, \ldots, z_{n})$ is injective 
on $\mathbb{C}$ with respect to $z_{0}$,
since $f_i$, for $i=1, \ldots, n$, 
do not depend on $z_{0}$.
Therefore $f_0$ is an affine transformation with respect to $z_{0}$,
namely,
$f_{0}(z_1, z_{1}, \ldots, z_{n}) = a(z_1,\ldots,z_n)z_0 + b(z_1,\ldots,z_n)$,
where
$a$ is a nowhere vanishing holomorphic function on $\mathbb{B}^{n}$ and
$b$ is a holomorphic function on $\mathbb{B}^{n}$.  
This completes the proof. 

\end{proof}

For the domain $D^{1,n} \simeq C^{n,1} \simeq \C^* \times \B^n$,
we can show the following theorem in a
similar way of the proof of Theorem~\ref{autCB}
(See \cite{MN}).

\begin{theorem}[\cite{MN}]\label{autC*B}
For $f=(f_{0}, f_{1}, \ldots, f_{n}) \in \mathrm{Aut}(C^{n,1})$,
\begin{eqnarray*}
f_{0}(z_{0}, z_{1}, \ldots, z_{n}) =
c\left(\frac{z_{1}}{z_{0}}, \ldots, \frac{z_{n}}{z_{0}}\right) z_{0} 
 \ \mathrm{or} \ 
c\left(\frac{z_{1}}{z_{0}}, \ldots, \frac{z_{n}}{z_{0}}\right) z_{0}^{-1},
\end{eqnarray*}
and
\begin{eqnarray*}
f_{i}(z_{0}, z_{1}, \ldots, z_{n}) =f_{0}(z_{0}, z_{1}, \ldots, z_{n}) 
\frac{a_{i 0}+\sum_{j=1}^{n} a_{i j} z_{j}}{a_{00}+\sum_{j=1}^{n} a_{0 j} z_{j}},
\end{eqnarray*}
for $i=1, \ldots, n$,
where $c$ is a nowhere vanishing holomorphic function on 
$\mathbb{B}^{n}$,
and the matrix $(a_{ij})_{0\leq i,j \leq n}$ is an element of $SU(n,1)$.  
\end{theorem}

\subsection{The automorphism groups of $D^{p,q}$ for $p>1$}
\label{sec:6.2}

\begin{theorem}\label{autDpq}
$\aut(D^{p,q}) \simeq GU(p,q)$ for $p>1$, $q>0$.
\end{theorem}

\begin{proof}
For fixed $(w_1, \ldots, w_q) \in \C^q$,
\begin{eqnarray*}
D^{p,q}\cap \{ z_1=w_1,\ldots, z_q=w_q \} &\simeq&\\
&&
\hspace{-2cm}
\{(z_{1+q}, \ldots, z_{p+q}) \in \mathbb{C}^p:
|z_{1+q}|^{2}+ \cdots + |z_{p+q}|^{2} > |w_{1}|^{2} + \cdots + |w_{q}|^{2}\},
\end{eqnarray*}
and the compliment of the right-hand side in $\C^p$
is a compact set, $\overline{\B^p_r}$,
where $r^2=|w_{1}|^{2} + \cdots + |w_{q}|^{2}$.
Thus any holomorphic function on $D^{p,q}$
extends holomorphically on $\C^{p+q}$ by Hartogs' extension theorem.
Therefore $f \in \aut(D^{p,q})$ extends to a holomorphic map
from $\C^{p+q}$ to itself, and since $f^{-1} \in \aut(D^{p,q})$
also extends, we have $f \in \aut(\C^{p+q})$ by the uniqueness of analytic 
continuation.
The theorem follows from the next lemma.

\end{proof}

\begin{lemma}\label{1.1}
Let $p, q>0$.
If $f \in \aut(\C^{p+q})$ preserves $D^{p,q}$ and the null cone
\begin{eqnarray*}
\partial D^{p,q} =
\{(z_1, \ldots, z_{p+q}) \in \mathbb{C}^{p+q}:
-|z_{1}|^{2} - \cdots -|z_{q}|^{2} + |z_{1+q}|^{2}+ \cdots + |z_{p+q}|^{2}=0\},
\end{eqnarray*}
then we have $f \in GU(p,q)$.
\end{lemma}

\begin{proof}
Since the origin $0$ is the unique singular point in $\partial D^{p,q}$,
we have $f(0)=0$ by the hypothesis of Lemma~\ref{1.1}.
Put $f = (f_1, \ldots, f_{p+q})$
and the Taylor series of $f_j$ by
\begin{eqnarray*}
f_j(z) 
&=& \sum_{\nu \in \mathbb{Z}_{\geq 0}^{p+q}} a^{(j)}_{\nu}z^{\nu}.
\end{eqnarray*}
For $t \in \C^*$, 
$f(tz)/t$ preserves $\partial D^{p,q}$:
\[
-|f_{1}(tz)/t|^{2} - \cdots -|f_{q}(tz)/t|^{2} + |f_{1+q}(tz)/t|^{2}+ \cdots + |f_{p+q}(tz)/t|^{2} = 0.
\]
If $t$ tends to $0$,
then we have
\[
-\left|\sum_{|\nu|=1} a^{(1)}_{\nu}z^{\nu}\right|^{2} - \cdots -
\left|\sum_{|\nu|=1} a^{(q)}_{\nu}z^{\nu}\right|^{2} +
\left|\sum_{|\nu|=1} a^{(1+q)}_{\nu}z^{\nu}\right|^{2}+ \cdots +
\left|\sum_{|\nu|=1} a^{(p+q)}_{\nu}z^{\nu}\right|^{2} = 0.
\]
This equation means
\begin{eqnarray*}
{\rm Jac}_{\C}f(0) = 
\begin{pmatrix}
a^{(1)}_{e_1} & \cdots & a^{(1)}_{e_{p+q}} \\
\vdots  & \ddots &  \vdots \\
a^{(p+q)}_{e_1} & \cdots & a^{(p+q)}_{e_{p+q}}
\end{pmatrix}
\in GL(p+q,\C),
\end{eqnarray*}
preserves $\partial D^{p,q}$,
where $e_1=(1,0,\ldots,0)$, $e_2=(0,1,0\ldots,0)$, $\ldots,e_{p+q}=(0\ldots,0,1)$ are
the natural basis of $\Z^{p+q}$.
Since ${\rm Jac}_{\C}f(0)$ is non degenerate matrix,
${\rm Jac}_{\C}f(0)$ also preserves $D^{p,q}$.
It follows from Lemma~\ref{condGU} that ${\rm Jac}_{\C}f(0)\in GU(p,q)$.
Considering $f \circ ({\rm Jac}_{\C}f(0))^{-1}$, we may assume that ${\rm Jac}_{\C}f(0) = E_{p+q}$,
and will prove that $f = E_{p+q}$.
To prove Lemma~\ref{1.1},
we show the following:

\begin{lemma}\label{1.2}
Let $f = (f_1, \ldots, f_{p+q}):\C^{p+q} \rightarrow \C^{p+q}$
be a holomorphic map,
and put the Taylor series of $f_j$ as $:$
\begin{eqnarray*}
f_j(z) 
&=& \sum_{\nu \in \mathbb{Z}_{\geq 0}^{p+q}} a^{(j)}_{\nu}z^{\nu}.
\end{eqnarray*}
If the holomorphic map
$f$ preserves the null cone
$\partial D^{p,q}$ and the origin $0$, respectively,
and if ${\rm Jac}_{\C}f(0) = E_{p+q}$,
then we have
$a^{(j)}_{\nu} = 0$ if $\nu-e_j \notin \mathbb{Z}_{\geq 0}^{p+q}$,
and
$a^{(j)}_{\nu+e_j} = a^{(1)}_{\nu+e_1}$.
\end{lemma}

\begin{proof}
By the assumption we have, on the null cone $\partial D^{p,q}$, 
\begin{eqnarray}\label{6.1}
0 &=&
-|f_{1}(z)|^{2} - \cdots -|f_{q}(z)|^{2} + |f_{q+1}(z)|^{2}+ \cdots + |f_{p+q}(z)|^{2} \nonumber \\
&=&
- \sum_{j=1}^{q} z_j \overline{\left(\sum_{|\nu|>1}a^{(j)}_{\nu}z^{\nu} \right)}
- \sum_{j=1}^{q}
\left(\sum_{|\nu|>1}a^{(j)}_{\nu}z^{\nu} \right)\overline{z_j}
- \sum_{j=1}^{q} \left|\sum_{|\nu|>1}a^{(j)}_{\nu}z^{\nu}\right|^2 \\
&&
+
\sum_{j=1+q}^{p+q} z_j \overline{\left(\sum_{|\nu|>1}a^{(j)}_{\nu}z^{\nu} \right)} 
+ 
\sum_{j=1+q}^{p+q} \left(\sum_{|\nu|>1}a^{(j)}_{\nu}z^{\nu} \right)\overline{z_j}
+ \sum_{j=1+q}^{p+q}\left|\sum_{|\nu|>1}a^{(j)}_{\nu}z^{\nu}\right|^2.\nonumber 
\end{eqnarray}
First we prove $a^{(j)}_{\nu} = 0$ for $\nu-e_j \notin \mathbb{Z}_{\geq 0}^{p+q}$.
We restrict the equation (\ref{6.1}) on
\begin{eqnarray*}
\left\{
\left(\sqrt[]{p}e^{\sqrt[]{-1}\theta_1}z_1,\ldots,\sqrt[]{p}e^{\sqrt[]{-1}\theta_{q}}z_{1},
\sqrt[]{q}e^{\sqrt[]{-1}\theta_{1+q}}z_{1},\ldots,
\sqrt[]{q}e^{\sqrt[]{-1}\theta_{p+q}}z_{1}\right): z_1 \in \C, (\theta_1,\ldots,\theta_{p+q}) \in \R^{p+q}
\right\}.
\end{eqnarray*}
Then, considering the right-hand side above as a function of $z_1$ and
$\bar z_1$, each coefficient of $z_1^k \bar z_1^l$
vanishes.
In particular, for the $z_1 \bar z_1^k$ term, we have
\begin{eqnarray*}
0 =
- \sum_{j=1}^{q} e^{\sqrt[]{-1}\theta_j} \sqrt[]{p} z_1 \overline{\left(\sum_{|\nu|=k}a^{(j)}_{\nu}z^{\nu} \right)}
+
\sum_{j=1+q}^{p+q} e^{\sqrt[]{-1}\theta_j} \sqrt[]{q} z_1 \overline{\left(\sum_{|\nu|=k}a^{(j)}_{\nu}z^{\nu} \right)},
\end{eqnarray*}
where $z^{\nu}$ should be read
$(e^{\sqrt[]{-1}\theta_1} \sqrt[]{p}z_1)^{\nu_1} \cdots (e^{\sqrt[]{-1}\theta_q} \sqrt[]{p}z_1)^{\nu_q}
(e^{\sqrt[]{-1}\theta_{1+q}} \sqrt[]{q}z_1)^{\nu_{1+q}} \cdots
(e^{\sqrt[]{-1}\theta_{p+q}} \sqrt[]{q} z_1)^{\nu_{p+q}}$.
Then, for $\nu \notin e_j + \mathbb{Z}_{\geq 0}^{p+q}$,
the coefficient of $e^{\sqrt[]{-1}\theta_j} e^{-\sqrt[]{-1}\langle \nu, \theta \rangle} z_1 \bar z_1^{|\nu|}$
must vanishes, that is, $a^{(j)}_{\nu}=0$ for $\nu \notin e_j + \mathbb{Z}_{\geq 0}^{p+q}$.

Next we prove $a^{(j)}_{\nu+e_j} = a^{(1)}_{\nu+e_1}$ for $j=2,\ldots,p+q$.
Put $z_1=e^{\sqrt[]{-1}\theta}r$ for $\theta \in \R$ and
$r = \sqrt[]{ -|z_2|^{2} - \cdots -|z_{q}|^{2} + |z_{1+q}|^{2}+ \cdots + |z_{p+q}|^{2}}>0$,
and substitute it into (\ref{6.1}).
Then, considering the right-hand side of (\ref{6.1}) as a function of $e^{\sqrt[]{-1}\theta}$,
each coefficient of $e^{k\sqrt[]{-1}\theta}$, $k \in \Z$, vanishes.
For $k=0$, we have
\begin{eqnarray*}
0
&=&
- r \overline{\left(\sum_{|\nu|>1}a^{(1)}_{\nu}z^{\nu'}r \right)}
- \sum_{j=2}^{q} z_j \overline{\left(\sum_{|\nu'|>1}a^{(j)}_{\nu'}z^{\nu'} \right)}
- \left(\sum_{|\nu|>1}a^{(1)}_{\nu}z^{\nu'}r \right)r
\\
&&
- \sum_{j=2}^{q} \left(\sum_{|\nu'|>1}a^{(j)}_{\nu'}z^{\nu'} \right)\overline{z_j}
- \sum_{j=1}^{q}\sum_{l,m>1}\left(\sum_{|\mu|=l}a^{(j)}_{\mu}z^{\mu'}r^{\nu_1} \right)
\overline{\left(\sum_{|\nu|=m}a^{(j)}_{\nu}z^{\nu'}r^{\nu_1} \right)}\\
&&+
\sum_{j=1+q}^{p+q} z_j \overline{\left(\sum_{|\nu'|>1}a^{(j)}_{\nu'}z^{\nu'} \right)} + 
\sum_{j=1+q}^{p+q} \left(\sum_{|\nu'|>1}a^{(j)}_{\nu'}z^{\nu'} \right)\overline{z_j}
\\
&&+ \sum_{j=1+q}^{p+q}\sum_{l,m>1} \left(\sum_{|\mu|=l}a^{(j)}_{\mu}z^{\mu'}r^{\nu_1}\right)
\overline{\left(\sum_{|\nu|=m}a^{(j)}_{\nu}z^{\nu'}r^{\nu_1}\right)},
\end{eqnarray*}
where $\nu'=(0,\nu_2,\ldots,\nu_{p+q})$, $\mu'=(0,\mu_2,\ldots,\mu_{p+q})$
and $\nu=\nu'+\nu_1e_1$, $\mu=\mu'+\nu_1e_1$.
Then the right hand side above is a power series of $z_2,\ldots,z_{p+q},\bar z_2,\ldots,\bar z_{p+q}$,
and therefore each coefficient of $z^{\mu'}\bar z^{\nu'}$ vanishes.
In particular, the coefficient of $z^{\nu'} |z_j|^2$ for $j=2,\ldots,p+q$ is $0= \pm (a^{(1)}_{\nu'+e_1} - a^{(j)}_{\nu'+e_j})$.
Thus we have proven that $a^{(j)}_{\nu+e_j} = a^{(1)}_{\nu+e_1}$ when $\nu_1=0$.

Similarly,  the coefficient of $e^{k\sqrt[]{-1}\theta}$ for $k=1$ is
\begin{eqnarray*}
0
&=&
- r \overline{\left(\sum_{|\nu'|>1}a^{(1)}_{\nu'}z^{\nu'} \right)}
- \left(\sum_{|\nu|>1}a^{(1)}_{\nu}z^{\nu'}r^2 \right)r
- \sum_{j=2}^{q} \left(\sum_{|\nu|>1}a^{(j)}_{\nu}z^{\nu'}r \right)\overline{z_j}
\\
&&
- \sum_{j=1}^{q}\sum_{l,m>1}\left(\sum_{|\mu|=l}a^{(j)}_{\mu}z^{\mu'}r^{\nu_1+1} \right)
\overline{\left(\sum_{|\nu|=m}a^{(j)}_{\nu}z^{\nu'}r^{\nu_1} \right)}
+ \sum_{j=1+q}^{p+q} \left(\sum_{|\nu|>1}a^{(j)}_{\nu}z^{\nu'}r \right)\overline{z_j}
\\
&&
+ \sum_{j=1+q}^{p+q}\sum_{l,m>1} \left(\sum_{|\mu|=l}a^{(j)}_{\mu}z^{\mu'}r^{\nu_1+1}\right)
\overline{\left(\sum_{|\nu|=m}a^{(j)}_{\nu}z^{\nu'}r^{\nu_1}\right)},
\end{eqnarray*}
where $\nu'$, $\mu'$
and $\nu$ are as before,
and
$\mu=\mu'+(\nu_1+1)e_1$.
Dividing this equation by $r$,
the right-hand side becomes a power series of $z_2,\ldots,z_{p+q},\bar z_2,\ldots,\bar z_{p+q}$,
and therefore the coefficient of $z^{\nu'} |z_j|^2$ for $j=2,\ldots,p+q$ vanishes,
that is, $0= \pm (a^{(1)}_{\nu'+2e_1} - a^{(j)}_{\nu'+e_1+e_j})$.
Thus we have proven that $a^{(j)}_{\nu+e_j} = a^{(1)}_{\nu+e_1}$ when $\nu_1=1$.

For $k>1$, the coefficient of $e^{k\sqrt[]{-1}\theta}$ is
\begin{eqnarray*}
0
&=&
- \left(\sum_{|\nu|>1}a^{(1)}_{\nu}z^{\nu'}r^{k+1} \right)r
- \sum_{j=2}^{q} \left(\sum_{|\nu|>1}a^{(j)}_{\nu}z^{\nu'}r^k \right)\overline{z_j}
\\
&&
- \sum_{j=1}^{q}\sum_{l,m>1}\left(\sum_{|\mu|=l}a^{(j)}_{\mu}z^{\mu'}r^{\nu_1+k} \right)
\overline{\left(\sum_{|\nu|=m}a^{(j)}_{\nu}z^{\nu'}r^{\nu_1} \right)}
+ \sum_{j=1+q}^{p+q} \left(\sum_{|\nu|>1}a^{(j)}_{\nu}z^{\nu'}r^k \right)\overline{z_j}
\\
&&
+ \sum_{j=1+q}^{p+q}\sum_{l,m>1} \left(\sum_{|\mu|=l}a^{(j)}_{\mu}z^{\mu'}r^{\nu_1+k}\right)
\overline{\left(\sum_{|\nu|=m}a^{(j)}_{\nu}z^{\nu'}r^{\nu_1}\right)},
\end{eqnarray*}
where $\nu'$, $\mu'$
and $\nu$ are as before,
and
$\mu=\mu'+(\nu_1+k)e_1$.
Dividing this equation by $r^k$,
the right-hand side becomes a power series of $z_2,\ldots,z_{p+q},\bar z_2,\ldots,\bar z_{p+q}$,
and therefore the coefficient of $z^{\nu'} |z_j|^2$ for $j=2,\ldots,p+q$ vanishes,
that is, $0= \pm (a^{(1)}_{\nu'+(k+1)e_1} - a^{(j)}_{\nu'+ke_1+e_j})$.
Thus we have proven $a^{(j)}_{\nu+e_j} = a^{(1)}_{\nu+e_1}$ when $\nu_1=k>1$,
and this completes the proof.
\end{proof}

We continue the proof of Lemma~\ref{1.1}.
We consider $f\in \aut(D^{p,q})$, ${\rm Jac}_{\C}f(0) = E_{p+q}$, on
\[
\left\{
\left(\sqrt[]{p}e^{\sqrt[]{-1}\theta_1}z_1,\ldots,\sqrt[]{p}e^{\sqrt[]{-1}\theta_{q}}z_{1},
\sqrt[]{q}e^{\sqrt[]{-1}\theta_{1+q}}z_{1},\ldots,
\sqrt[]{q}e^{\sqrt[]{-1}\theta_{p+q}}z_{1}\right):
z_1 \in \C, (\theta_1,\ldots,\theta_{p+q}) \in \R^{p+q}
\right\} \subset \partial D^{p,q}.
\]
By Lemma~\ref{1.2}, for each $1 \leq j \leq q$,
we have
\begin{eqnarray}\label{3}
&&f_j\left(\sqrt[]{p}e^{\sqrt[]{-1}\theta_1}z_1,\ldots,\sqrt[]{p}e^{\sqrt[]{-1}\theta_{q}}z_{1},
\sqrt[]{q}e^{\sqrt[]{-1}\theta_{1+q}}z_{1},\ldots,
\sqrt[]{q}e^{\sqrt[]{-1}\theta_{p+q}}z_{1}\right)\\
&& =
e^{\sqrt[]{-1}(\theta_j-\theta_1)}
f_1\left(\sqrt[]{p}e^{\sqrt[]{-1}\theta_1}z_1,\ldots,\sqrt[]{p}e^{\sqrt[]{-1}\theta_{q}}z_{1},
\sqrt[]{q}e^{\sqrt[]{-1}\theta_{1+q}}z_{1},\ldots,
\sqrt[]{q}e^{\sqrt[]{-1}\theta_{p+q}}z_{1}\right),\nonumber
\end{eqnarray}
and for each $1+q \leq j \leq p+q$,
we have
\begin{eqnarray}\label{3'}
&&f_j\left(\sqrt[]{p}e^{\sqrt[]{-1}\theta_1}z_1,\ldots,\sqrt[]{p}e^{\sqrt[]{-1}\theta_{q}}z_{1},
\sqrt[]{q}e^{\sqrt[]{-1}\theta_{1+q}}z_{1},\ldots,
\sqrt[]{q}e^{\sqrt[]{-1}\theta_{p+q}}z_{1}\right)\\
&& =
e^{\sqrt[]{-1}(\theta_j-\theta_1)}\frac{\sqrt[]{q}}{\sqrt[]{p}}
f_1\left(\sqrt[]{p}e^{\sqrt[]{-1}\theta_1}z_1,\ldots,\sqrt[]{p}e^{\sqrt[]{-1}\theta_{q}}z_{1},
\sqrt[]{q}e^{\sqrt[]{-1}\theta_{1+q}}z_{1},\ldots,
\sqrt[]{q}e^{\sqrt[]{-1}\theta_{p+q}}z_{1}\right).\nonumber
\end{eqnarray}
Since the holomorphic map
\begin{eqnarray*}
&&
\hspace{-6mm}
\C \ni z_1 \rightarrow
\biggl(f_1\left(\sqrt[]{p}e^{\sqrt[]{-1}\theta_1}z_1,\ldots,\sqrt[]{p}e^{\sqrt[]{-1}\theta_{q}}z_{1},
\sqrt[]{q}e^{\sqrt[]{-1}\theta_{1+q}}z_{1},\ldots,
\sqrt[]{q}e^{\sqrt[]{-1}\theta_{p+q}}z_{1}\right),
\ldots,\\
&&
\hspace{18mm}
f_{p+q}\left(\sqrt[]{p}e^{\sqrt[]{-1}\theta_1}z_1,\ldots,\sqrt[]{p}e^{\sqrt[]{-1}\theta_{q}}z_{1},
\sqrt[]{q}e^{\sqrt[]{-1}\theta_{1+q}}z_{1},\ldots,
\sqrt[]{q}e^{\sqrt[]{-1}\theta_{p+q}}z_{1}\right) \biggr) \in \C^{p+q}
\end{eqnarray*}
is injective on $\C$, the holomorphic map
\[
\C \ni z_1 \rightarrow
f_1\left(\sqrt[]{p}e^{\sqrt[]{-1}\theta_1}z_1,\ldots,\sqrt[]{p}e^{\sqrt[]{-1}\theta_{q}}z_{1},
\sqrt[]{q}e^{\sqrt[]{-1}\theta_{1+q}}z_{1},\ldots,
\sqrt[]{q}e^{\sqrt[]{-1}\theta_{p+q}}z_{1}\right)
\in \C
\]
must be injective by (\ref{3}) and (\ref{3'}).
Any holomorphic injective map from $\C$ to itself is an
affine transformation.
Thus we have
\[
f_1\left(\sqrt[]{p}e^{\sqrt[]{-1}\theta_1}z_1,\ldots,\sqrt[]{p}e^{\sqrt[]{-1}\theta_{q}}z_{1},
\sqrt[]{q}e^{\sqrt[]{-1}\theta_{1+q}}z_{1},\ldots,
\sqrt[]{q}e^{\sqrt[]{-1}\theta_{p+q}}z_{1}\right) = cz_1,
\]
with some non-zero constant $c \in \C$,
since $f(0) = 0$.
Therefore we have
\begin{eqnarray*}
\sum_{|\nu|=k} e^{\sqrt[]{-1}\langle \nu, \theta \rangle}
a^{(1)}_{\nu}(\sqrt[]{p})^{|\nu'|}(\sqrt[]{q})^{|\nu''|}
= 0,
\end{eqnarray*}
for each $k>1$, 
$\theta=(\theta_1,\ldots,\theta_{p+q}) \in \R^{p+q}$,
where $\nu=(\nu',\nu'') \in \Z_{\geq 0}^q \times \Z_{\geq 0}^p$.
Since these equations hold for any $\theta \in \R^{p+q}$,
we have $a^{(1)}_{\nu}=0$ for $|\nu|>1$.
This implies that, for any $1 \leq j \leq p+q$, $a^{(j)}_{\nu}=0$ for $|\nu|>1$
by Lemma~\ref{1.2}.
Thus we have proven that $f=E_{p+q}$,
and this completes the proof.

\end{proof}


\begin{thebibliography}{999}


\bibitem{BKS0}
J. Byun, A. Kodama, S. Shimizu, A group-theoretic characterization of the direct product of a ball and a Euclidean space,
Forum Math. \textbf{18}, no. 6, 983--1009 (2006)

\bibitem{BKS}
J. Byun, A. Kodama, S. Shimizu, A group-theoretic characterization of the direct product of a ball and punctured planes, Tohoku Math.\ J. (2) \textbf{62}, no. 4, 485--507 (2010)


\bibitem{IK}
A. V. Isaev and N. G. Kruzhilin,  Effective actions of the unitary group on complex manifolds, Canad.\ J. Math. \textbf{54}, no. 6, 1254--1279 (2002)



\bibitem{KS1}
A. Kodama and S. Shimizu, A group-theoretic characterization of the space obtained by omitting the coordinate hyperplanes from the complex Euclidean space, Osaka J. Math. \textbf{41}, no. 1, 85--95 (2004)


\bibitem{KS}
A. Kodama and S. Shimizu, Standardization of certain compact group actions and the automorphism group of the complex Euclidean space, Complex Var.\ Elliptic Equ. \textbf{53}, no. 3, 215--220 (2008)



\bibitem{Kutzsche}
F. Kutzschebauch,
Compact and reductive subgroups of the group of holomorphic automorphisms of $\C^n$, 
Singularities and complex analytic geometry (Kyoto, 1997). 
Surikaisekikenkyusho Kokyuroku No. 1033, 81--93 (1998)

\bibitem{MN}
J. Mukuno, Y. Nagata, On a characterization of unbounded homogeneous domains with boundaries of light cone type, arXiv:1412.3217. to appear Tohoku Math.\ J.

\bibitem{Pyateskii}
I. I. Pyateskii-Shapiro,
Automorphic functions and the geometry of classical domains, 
viii+264 pp, Gordon and Breach Science Publishers, 
New York-London-Paris (1969)




\end{thebibliography}
\end{document}